\documentclass[11pt]{article}

\usepackage{setspace}
\def\naive{na\"ive}

\include{psfig}
\usepackage{pictex, latexsym,amsmath,amssymb,amsbsy,amsfonts,amsthm,verbatim}
\usepackage[pdfpagemode=UseOutlines,colorlinks=true,pdfnewwindow=true,urlcolor=red]{hyperref}
\usepackage{color}
\usepackage{algorithm}
\usepackage{algorithmic}
\usepackage{subfigure}
\usepackage{multirow}
\usepackage[pdftex]{graphicx}
\usepackage{graphics}

\newcommand{\beq}{\begin{equation}}
\newcommand{\eeq}{\end{equation}}
\newcommand{\bea}{\begin{eqnarray*}}
\newcommand{\eea}{\end{eqnarray*}}

\newcommand{\case}[1]{{\noindent{\bf Case #1:}}}
\newcommand{\subcase}[1]{{\noindent{\it Case #1}:}}

\input colordvi
\newcommand{\be}{\begin{equation}}
\newcommand{\ee}{\end{equation}}
\newcommand{\bx}{\begin{bmatrix}}
\newcommand{\ex}{\end{bmatrix}}

\theoremstyle{definition}

\theoremstyle{theorem}
\newtheorem{theorem}{Theorem}[section]
\newtheorem{conjecture}[theorem]{Conjecture}

\newtheorem{corollary}[theorem]{Corollary}
\newtheorem{lemma}[theorem]{Lemma}
\newtheorem{proposition}[theorem]{Proposition}

\newtheorem{claim}[theorem]{Claim}

 \newenvironment{cem}
{
    \begin{enumerate}
        \setlength{\topsep}{0pt}
        \setlength{\parskip}{0pt}
        \setlength{\partopsep}{0pt}
        \setlength{\parsep}{0pt}         
        \setlength{\itemsep}{0pt} 
}
{
    \end{enumerate} 
}

\def\cF{\mathcal{F}}
\def\cE{\mathcal{E}}

\def\cP{\mathcal{P}}
\def\cB{\mathcal{B}}
\def\eps{\varepsilon}

\begin{document}

\title{Generating $p$-extremal graphs}
\author{	Derrick Stolee\thanks{ %
	The author is supported in part by %
	the National Science Foundation grants %
	CCF-0916525 and DMS-0914815.%
	}\\ 
	Department of Mathematics\\
	Department of Computer Science\\
	University of Nebraska--Lincoln\\
	 \texttt{s-dstolee1@math.unl.edu} 
}

\maketitle

\begin{abstract}
	Define $f(n,p)$ to be the maximum number of edges in a graph
		on $n$ vertices with $p$ perfect matchings.
	Dudek and Schmitt proved  there exist constants $n_p$ and $c_p$
		so that for  even $n \geq n_p$, 
		$f(n,p) = \frac{n^2}{4}+c_p$. 
	A graph is \emph{$p$-extremal} if
		it has $p$ perfect matchings and $\frac{n^2}{4}+c_p$ edges.
	Based on  Lov\'asz's Two Ear Theorem
		and structural results of Hartke, Stolee, West, and Yancey,
		we develop a computational method for 
		determining $c_p$ and
		generating the finite set of graphs which 
		describe the infinite family of $p$-extremal graphs.
	This method extends the knowledge of
		the size and structure of $p$-extremal graphs
		from $p \leq 10$
		to $p \leq 27$.
	These values provide further evidence towards 
		a conjectured upper bound and
		prove  the sequence $c_p$ is not
		monotonic.
\end{abstract}

\def\arbitrarypagebreak{}

\section{Introduction}

A \emph{perfect matching}
	is a set of disjoint edges which 
	cover all vertices.
Let $\Phi(G)$ be the number of perfect matchings in a graph $G$.
For even $n$ and positive $p$,
	the function $f(n,p)$ is 
	the maximum number of edges in a graph $G$
	on $n$ vertices with $\Phi(G) = p$.
The exact behavior of $f(n,p)$
	is not completely understood.
This work extends the current knowledge on this problem
	by applying Lov\'asz'a Two Ear Theorem \cite{LovaszTwoEars}
	and structure theorems of
	Hartke, Stolee, West, and Yancey~\cite{HSWY}
	with an isomorph-free generation scheme~\cite{Stolee11}
	in order to compute $f(n,p)$
	for all $p \leq 27$
	as well as determine the exact structure of graphs 
	meeting these extremal values.



Hetyei first characterized 
	the extremal graphs with a single perfect matching and $n$ vertices
	(unpublished; see~\cite[Corollary 5.3.14]{LovaszPlummer})
	giving $f(n,1) = \frac{n^2}{4}$ for all even $n$.
Dudek and Schmitt~\cite{DudekSchmitt} generalized the problem
	for an arbitrary constant $p$
	and found the general form of $f(n,p)$ for
	sufficiently large $n$.
	
\begin{theorem}[Dudek, Schmitt~\cite{DudekSchmitt}]\label{thm:fnpform}
	For every $p \geq 1$, there exist constants $c_p, n_p$ so that
		for all even $n \geq n_p$,
		$f(n,p) = \frac{n^2}{4} + c_p$.
\end{theorem}

From this theorem, 
	understanding the behavior of $f(n,p)$ relies
	on understanding the sequences $n_p$ and $c_p$.
An important step to proving Theorem \ref{thm:fnpform} is that
	if a graph exists with $p$ perfect matchings,
	$n$ vertices, and $\frac{n^2}{4} + c$ edges, 	
	then $f(m,p) \geq \frac{m^2}{4} + c$
	for all even $m \geq n$~\cite[Lemma 2.1]{DudekSchmitt}).
This motivates defining the \emph{excess} of $G$
		to be $c(G) = e(G) - \frac{1}{4} n(G)^2$,
		since any graph $G$ with $p$ perfect matchings
	gives a lower bound $c(G) \leq c_p$.
A graph $G$ with $\Phi(G) = p$ is \emph{$p$-extremal} 
	if it has $\frac{n^2}{4}+c_p$ edges.
In other words, $p$-extremal graphs have the largest excess
	out of all graphs with $p$ perfect matchings.
		
Dudek and Schmitt~\cite{DudekSchmitt} 
	computed $c_p$ for $2 \leq p \leq 6$
	and provided the exact structure of $p$-extremal graphs
	for $p \in \{2,3\}$.
Hartke, Stolee, West, and Yancey~\cite{HSWY}
	analyzed the structure of $p$-extremal graphs
	and found that for a fixed $p$, the infinite family
	is constructable from a finite set of 
	fundamental graphs.
Using McKay's graph generation program \texttt{geng} \cite{nauty}, 
	they discovered these funcdamental graphs
	then computed $c_p$ for $7 \leq p \leq 10$
	and described the structure of $p$-extremal graphs
	for $4 \leq p \leq 10$.
These values of $c_p$ are given in Table \ref{tab:OldValuesCP}.
The current-best known upper bound\footnote{
	This upper bound is given by $c_p \leq \max_{q < p} c_q + 1$~\cite[Lemma 2.4]{DudekSchmitt}
	and that $c_{p} \leq 4$ for all $p \leq 10$~\cite{HSWY}.
	} 
	is $c_p \leq p - 6$ for $p \geq 11$,
	while the best known lower bound 
	is $c_p \geq 1$ for all $p \geq 2$~\cite[Theorem 2.3]{HSWY}. 
	
\begin{table}[ht]
	\centering	
	\begin{tabular}[h]{|r|c|c|c|c|c|c|c|c|c|c|}
	\hline
		$p$ & 1 & 2 & 3 & 4 & 5 & 6 & 7 & 8 & 9 & 10  \\
		\hline
		$n_p$ & 2 & 4 & 4 & 6 & 6 & 6 & 6 & 6& 6& 6 \\
		$c_p$ & 0 & 1 & 2 & 2 & 2 & 3 & 3 & 3& 4& 4  \\
		\hline
		 & \multicolumn{6}{c|}{\cite{DudekSchmitt}} 
		 		& \multicolumn{4}{c|}{\cite{HSWY}} \\
		\hline
 	\end{tabular}
	\caption{\label{tab:OldValuesCP}Known values of $n_p$ and $c_p$.}
\end{table}

For an integer $n$, the \emph{double factorial} $n!!$
	is the product of all numbers at most $n$ with the same parity.
The values of $c_p$ are conjectured to be 
	on order $O\left(\left(\frac{\ln p}{\ln\ln p}\right)^2\right)$,
	given below.

\begin{conjecture}[Hartke, Stolee, West, Yancey~\cite{HSWY}]\label{conj:starremoval}
Let $p, k, t$ be integers so that $k \in \{1,\dots, 2t\}$ and
	$k(2t-1)!! \leq p < (k+1)(2t-1)!!$
	and set $C_p = t^2 - t + k - 1$.
Always $c_p \leq C_p$.
\end{conjecture}

When  $p = k(2t-1)!!$, 
	a construction shows that   
	$c_p \geq t^2 - t + k - 1$~\cite{HSWY}.
		
This work utilizes the structure of the fundamental graphs
	to design a computer search.
By executing this search, we compute the sequence $c_p$ 
	and describe the structure of $p$-extremal graphs for all $p \leq 27$.
We begin by discussing the structure of $p$-extremal graphs in Section \ref{ssec:structurepm}.
Sections \ref{ssec:pmdeletion} through \ref{ssec:pruning} contain the 
	description of the computational technique,
	which are described in more detail at the end of Section \ref{ssec:structurepm}.
In Section \ref{ssec:results}, we discuss the results of executing 
	the computer search.

\subsection*{Notation}

In this work, $H$ and $G$ are graphs, 
	all of which will are simple: there are no loops or multi-edges.
For a graph $G$, $V(G)$ is the vertex set and $E(G)$ is the edge set.
The number of vertices is denoted $n(G)$ while $e(G)$ is the number of edges.

\section{Structure of $p$-Extremal Graphs}
\label{ssec:structurepm}

In this section, we describe the structure of $p$-extremal graphs
	as demonstrated by Hartke, Stolee, West, and Yancey~\cite{HSWY}.
	
	A graph is \emph{matchable} if it has a perfect matching.
	An edge $e \in E(G)$ is \emph{extendable}
		if there exists a perfect matching of $G$
		which contains $e$.
	Otherwise, $e$ is \emph{free}.
	The \emph{extendable subgraph} (\emph{free subgraph}) of $G$ is 
		the spanning subgraph containing all extendable (free) edges of $G$.

\def\odd{\operatorname{odd}}
	If the extendable subgraph of a matchable graph $G$ is connected,
		then $G$ is \emph{elementary}.
	A set $S \subset V(G)$ is a \emph{barrier}\footnote{We take the convention that the empty set is a barrier.}
		if the number of connected components 
		with an odd number of vertices
		in 
		$G - S$ (denoted $\odd(G-S)$) is equal to $|S|$.
Recall Tutte's Theorem~\cite{Tutte} states 
	$G$ is matchable if and only if $|S| \geq \odd(G-S)$ for
	all subsets $S \subseteq V(G)$,
	so barriers are the sets 
	which make this condition sharp.
Note that the singletons $\{v\}$ for each $v \in V(G)$
	is a barrier.

Elementary graphs and their barriers share important structure,
	which will be investigated thoroughly in Section \ref{ssec:barriers}.
If $G$ is both elementary and $p$-extremal, 
	then $n(G)$ is bounded
	by a function of $p$ and $c_p$.

\begin{theorem}[Corollary 5.8~\cite{HSWY}]\label{thm:sizebound}
	Let $p \geq 2$.
	If $G$ is a $p$-extremal elementary graph, then 
		$G$ has at most $N_p$ vertices,
		where $N_p$ is the largest even integer at most
		$3 + \sqrt{16p-8c_p-23}$.
\end{theorem}

Any excess $c(G)$
	for a graph $G$ with $\Phi(G) = p$
	can replace $c_p$ 
	in Theorem \ref{thm:sizebound}
	to give an upper bound on the order
	of a $p$-extremal elementary graph.

	Let $G$ be a graph with $\Phi(G) > 0$.
	A \emph{chamber} 
		is a subgraph of $G$
		induced by 
		a connected component of the extendable 
		subgraph of $G$.
Chambers are the maximal elementary subgraphs of $G$.
	Let $G_1,\dots,G_k$ be elementary graphs
		and for each $i$ let $X_i \subseteq V(G_i)$ 
		be a barrier in $G_i$.
	The \emph{spire} generated by $G_1,\dots,G_k$ on $X_1,\dots,X_k$
		is the graph given by disjoint union of $G_1,\dots, G_k$
		and edges $x_iv_j$ for all $x_i \in X_i$ and $v_j \in V(G_j)$,
		whenever $i < j$.
The following theorem states that all $p$-extremal graphs
	are spires with very specific conditions on $G_1,\dots,G_k$ and
	$X_1,\dots,X_k$.

\begin{theorem}[Theorem 5.9~\cite{HSWY}]\label{thm:hswystructure}
	Consider $p \geq 1$. 
	For each $p$-extremal graph $G$ in $\cF_p$:
	\begin{cem}
		\item $G$ is a spire generated by elementary graphs $G_1,\dots,G_k$
					on barriers $X_1,\dots,X_k$.
		\item The chambers of $G$ are $G_1,\dots, G_k$.
		\item For each $i < k$, $X_i$ is a barrier of maximum size in $G_i$.
		\item For all $i < j$, $\frac{|X_i|}{n(G_i)} \geq \frac{|X_j|}{n(G_j)}$
			and if equality holds, $G_i$ and $G_j$ can be swapped 
			to form another $p$-extremal graph.
		\item $c(G) \leq \sum_{i=1}^k c(G_i)$. 
			Equality holds if and only if
				$\frac{|X_i|}{n(G_i)} = \frac{1}{2}$ for all $i < k$.
		\item Let $p_i = \Phi(G_i)$. $p = \prod_{i=1}^k p_i$.
		\item For all $i$, $n(G_i) \leq N_{p_i}$.
		\item If $p_i = 1$, then $G_i \cong K_2$, $|X_i| = 1$.
	\end{cem}
\end{theorem}

Theorem \ref{thm:hswystructure} 
	provides an automated procedure for describing 
	all $p$-extremal graphs.
Begin by determining all $q$-extremal elementary graphs
	for each factor $q$ of $p$.
Then, compute the maximum-size barriers for these graphs.
For all factorizations $p = \prod_{i=1}^k p_i$
	and combinations of $p_i$-extremal elementary graphs $G_i$ with maximum-size barrier $X_i$,
	compute $c(G)$ for the spire generated by $G_1,\dots,G_k$ on $X_1,\dots,X_k$.
The maximum excess of these graphs is the value $c_p$.	
Larger graphs are built by adding elementary graphs isomorphic to $K_2$
	and ordering the list of elementary graphs
	by the relative barrier size $\frac{|X_i|}{n(G_i)}$.

The most difficult part of this procedure is determining
	all $q$-extremal elementary graphs,
	which is the focus of the remainder of this work.
In \cite{HSWY}, the authors
	found the $q$-extremal elementary graphs by enumerating all graphs of 
	order $N_q$ using McKay's \texttt{geng} program \cite{nauty}
	until $N_q \geq 14$ for $q \geq 11$, where this technique
	became infeasible.
Our method greatly extends the range of computable values.
We split elementary graphs into  
	extendable and free subgraphs, which
	are generated in two stages of a computer search.
We begin by investigating the
	structure of extendable subgraphs
	in Section \ref{ssec:pmdeletion}.
In Section \ref{ssec:canondelete}, we utilize this structure
	to design an algorithm
	for generating all possible extendable subgraphs
	of $q$-extremal elementary graphs
	which focuses the search to a very 
	sparse family of graphs.
In Section \ref{ssec:barriers},
	we investigate the structure of free subgraphs and
	design an algorithm to generate
	maximal graphs with a given extendable subgraph.
This algorithm requires the full list of barriers for an extendable subgraph,
	so we describe in Section \ref{ssec:evolutionofbarriers}
	an on-line algorithm for computing this list.
In Section \ref{ssec:pruning}, we combine these techniques
	to bound the possible excess reachable
	from a given graph in order
	to significantly prune the search space.
These algorithms are combined to a final implementation
	and the results of the computation
	are summarized in Section \ref{ssec:results}.

\section{Structure of Extendable Subgraphs}
\label{ssec:pmdeletion}

	A connected graph is \emph{1-extendable} if every edge is extendable\footnote{
		This term comes from $k$-extendable graphs, where
			every matching of size $k$ extends to a perfect matching.
	}.
By the definition of elementary graph, 
	the extendable subgraph of an elementary graph is 1-extendable.
	
	A graph $H$ with $n(H) \geq 3$ is \emph{2-connected} if 
		there is no vertex $x \in V(G)$ so that $H-x$ is disconnected.

\begin{proposition}
	If $H$ is 1-extendable with $\Phi(H) \geq 2$,
		then $H$ is 2-connected. 
\end{proposition}

\begin{proof}
	Since $\Phi(H) \geq 2$, there are at least four vertices in $H$.
	Suppose $H$ was not 2-connected.
	Then, there exists a vertex $x \in V(H)$ so that $H-x$ has multiple components.
	Since $H$ has an even number of vertices,
		at least one component of $H-x$ must have an odd number of vertices.
	Since $H$ has perfect matchings, 
		Tutte's Theorem implies 
		exactly one such component $C$ has an odd number of vertices.
	Moreover, in every perfect matching of $H$, $x$ is matched to some vertex in $C$.
	Hence, the edges from $x$ to the other components never appear in perfect matchings,
		contradicting that $H$ was 1-extendable.
\end{proof}

2-connected graphs are characterized by \emph{ear decompositions}.
An \emph{ear} is a path given by vertices $x_0,x_1,\dots,x_k$
	so that $x_0$ and $x_k$ have degree at least three
	and $x_i$ has degree exactly two for all $i \in \{1,\dots,k-1\}$.
The vertices $x_0$ and $x_k$ are \emph{branch vertices}
	while $x_1,\dots,x_{k-1}$ are \emph{internal vertices}.
In the case of a cycle, the entire graph is considered to be an ear.
For an ear $\eps$, the \emph{length} of $\eps$ is the number of edges between
	the endpoints
	and its \emph{order} is the number of internal vertices between the endpoints.
We will focus on the order of an ear.
An ear of order $0$ (length $1$) is a single edge, 
	called a \emph{trivial} ear.

An \emph{ear augmentation} 
	is the addition of a path 
	between two vertices of the graph.
This process is invertible: an \emph{ear deletion} 
	takes an ear $x_0,x_1,\dots,x_k$ in a graph
	and deletes all vertices $x_1,\dots,x_{k-1}$
	(or the edge $x_0x_1$ if $k = 1$).
For a graph $H$, an ear augmentation is denoted $H+\eps$ 
	while an ear deletion is denoted $H-\eps$.
Every 2-connected graph $H$ has a sequence of graphs 
	$H_1 \subset \cdots \subset H_\ell = H$
	so that 
	$H_1$ is a cycle
	and for all $i \in \{1,\dots,\ell-1\}$,
	$H^{(i+1)} = H^{(i)}+\eps_i$ for some ear $\eps_i$~\cite{EarDecompositions}.

Lov\'asz's Two Ear Theorem gives
	the vital structural decomposition of 1-extendable graphs
	using a very restricted type of ear decomposition.
	A sequence $H_0 \subset H_1 \subset H_2 \subset \cdots \subset H_k$
		of ear augmentations
		is a \emph{graded ear decomposition}
		if each $H^{(i)}$ is 1-extendable.
	The decomposition is \emph{non-refinable} if
		for all $i < k$, there is no 1-extendable graph $H'$
		so that $H^{(i)} \subset H' \subset H^{(i+1)}$
		is a graded ear decomposition.

\begin{theorem}[Two Ear Theorem~\cite{LovaszTwoEars}; %
	See also~\cite{LovaszPlummer,TwoEarsProof}]
	\label{thm:lovasztwoears}
	If $H$ is 1-extendable, then there is a non-refinable graded ear decomposition
		$H_0 \subset H_1 \subset \cdots \subset H_k$
	so that $H_0 \cong C_{2\ell}$ for some $\ell$ 
		and each ear augmentation $H^{(i)} \subset H^{(i+1)}$
		uses one or two new ears, each with 
		an even number of internal vertices.
\end{theorem}

We will consider making single-ear augmentations to build 1-extendable
	graphs, so we classify the graphs which appear after the
	first ear of a two-ear augmentation.
	A graph $H$ is \emph{almost 1-extendable} if 
		the free edges of $H$ appear in a single 
		ear of $H$.
The following corollary is a restatement of the Two Ear Theorem
	using almost 1-extendable graphs.

\begin{corollary}\label{cor:twoears}
	If $H$ is 1-extendable, then there is an ear decomposition
		$H_0 \subset H_1 \subset \cdots \subset H_k$
	so that $H_0 \cong C_{2\ell}$ for some $\ell$,
		each ear augmentation $H^{(i)} \subset H^{(i+1)}$ 
			uses a single ear of even order,
		each $H^{(i)}$ is either 1-extendable or almost 1-extendable,
		and if $H^{(i)}$ is almost 1-extendable
			then $H_{i-1}$ and $H^{(i+1)}$ are 1-extendable.
\end{corollary}

An important property of graded ear decompositions 
	is that $\Phi(H^{(i)}) < \Phi(H^{(i+1)})$,
	since the perfect matchings in $H^{(i)}$ 
	extend to perfect matchings of $H^{(i+1)}$ 
	using alternating paths within the augmented ear(s)
	and the other edges must appear in a previously uncounted
	perfect matching.

We use this theorem to develop our search space for
	the canonical deletion technique,
	forming the first stage of the search.
The second stage adds free edges
	to a 1-extendable graph with $p$ perfect matchings.
The structure of free edges is even more restricted,
	as shown in the following proposition.

\begin{proposition}[Theorems 5.2.2 \& 5.3.4~\cite{LovaszPlummer}]
	\label{prop:freeextendable}
	Let $G$ be an elementary graph.
	An edge $e$ is free if and only if the 
		endpoints are in the same barrier.
	If adding any missing edge to $G$ 
		increases the number of perfect matchings,
		then every barrier in $G$ of size at least two
		is a clique of free edges.
\end{proposition}

In Section \ref{ssec:barriers},
	we describe a technique for adding free edges
	to a 1-extendable graph.
In order to better understand the first stage, 
	we investigate what types of ear augmentations
	are allowed in a non-refinable graded ear decomposition.
	
\begin{lemma}\label{lma:oneartype}
	Let $H \subset H + \eps$ be a one-ear augmentation between
		1-extendable graphs $H$ and $H+\eps$.
	The endpoints of $\eps$ are 
				in disjoint maximal barriers.
\end{lemma}

\begin{proof}		
	If the endpoints of $\eps$ were not in disjoint maximal barriers, then they
			are contained in the same maximal barrier.
		If an edge were added between these vertices,
			Proposition \ref{prop:freeextendable}
			states that this edge would be free.
		Since $\eps$ is an even subdivision of such an edge,
			the edges incident to the endpoints are
			not extendable,
		making $H+\eps$ not 1-extendable.
\end{proof}	

\begin{lemma}\label{lma:twoeartype}
	Let $H \subset H + \eps_1 + \eps_2$ be a non-refinable 
		two ear augmentation between 1-extendable graphs.
	\begin{cem}
			\item The endpoints of $\eps_1$ are within a maximal barrier of $H$.
			\item The endpoints of $\eps_2$ are within a different maximal barrier of $H$.
	\end{cem}
\end{lemma}

\begin{proof}
	(1) If the endpoints $a, b$ of $\eps_1$ span two different maximal barriers,		
		adding the edge $ab$ would add an extendable edge 
			by Proposition \ref{prop:freeextendable}.
		The perfect matchings of $H + ab$ and $H+\eps_1$ 
			would be in bijection depending on if $ab$ was used:
			if a perfect matching $M$ in $H+ab$ does not contain $ab$, 
			$M$ extends to a perfect matching in $H+\eps_1$ by
			taking alternating edges within $\eps_1$,
			with the edges incident to $a$ and $b$ not used;
			if $M$ used $ab$, the alternating edges along $\eps_1$ would
				use the edges incident to $a$ and $b$.
		Hence, $H + \eps_1$ is 1-extendable and this is a refinable graded
			ear decomposition.
		This contradiction shows that $\eps_1$ spans vertices within a single
			maximal barrier.
	
	(2) The endpoints $x, y$ of $\eps_2$ must be within a single maximal barrier
		by the same proof as (2), since otherwise $H+\eps_2$ would be 1-extendable
		and the augmentation is refinable. 
		However, if both $\eps_1$ and $\eps_2$ 
			span the same maximal barrier, 
			$H + \eps_1 + \eps_2$ is not 1-extendable.
		By Proposition \ref{prop:freeextendable},
			edges within a barrier are free.
		Hence, the perfect matchings of $H + \eps_1 + \eps_2$ 
			do not use the internal edges of $\eps_1$ and $\eps_2$ which are 
			incident to their endpoints.
		This contradicts 1-extendability, 
			so the endpoints of $\eps_2$ 
			are in a different maximal barrier
			than the endpoints of $\eps_1$.
\end{proof}

\section{Searching for $p$-extremal elementary graphs}
\label{ssec:canondelete}

Given $p$ and $c$, we aim to generate
	all elementary graphs $G$ with 
	$\Phi(G) = p$
	and $c(G) \geq c$.
If $c \leq c_p$, Theorem \ref{thm:sizebound}
	implies $n(G) \leq N_p \leq 3 + \sqrt{16p-8c-23}$.
In order to discover these graphs,
	we use the isomorph-free generation algorithm of~\cite{Stolee11}
	to generate 1-extendable graphs with up to $p$ perfect matchings
	and up to $N_p$ vertices.
This algorithm is based on Brendan McKay's 
	canonical deletion technique~\cite{McKayIsomorphFree}
	and generates graphs using ear augmentations while visiting each unlabeled graph only once.
This technique will generate 1-extendable graphs and almost 1-extendable graphs.
\def\cM{\mathcal{M}}
	Let $\cM^p$ be the set of 2-connected graphs $G$
		with $\Phi(G) \in \{2,\dots, p\}$
		that are either
		1-extendable or {almost 1-extendable}.
	$\cM^p_{ N_p}$ is the set of
		graphs in $\cM^p$ with at most $N_p$ vertices.

The following lemma is immediate from Corollary \ref{cor:twoears}.

\begin{lemma}
	\label{lma:pmdeletionclosed}
	For each graph $H \in \cM^p$, 
		either $H$ is an even cycle
		or there exists an ear $\eps$ so that $H-\eps$ is in $\cM^p$.
\end{lemma}

\def\del{\operatorname{del}}
\def\lab{\operatorname{lab}}
With this property, all graphs in $\cM^p_{N_p}$ can be generated 
	by a recursive process:
Begin at an even cycle $H_0 = C_{2\ell}$.
For each $H^{(i)}$, try adding each all ears $\eps$ of order $r$ to all pairs of vertices
	in $H^{(i)}$ where $0 \leq r \leq N_p - n(H^i)$ to 
	form $H^{(i+1)} + \eps$.
If $H^{(i+1)}$ is 1-extendable or $H^{(i)}$ is 1-extendable and $H^{(i+1)}$ is almost 1-extendable,
	recurse on $H^{(i+1)}$ until $\Phi(H^{(i+1)}) > p$.
While this technique will generate all graphs in $\cM^p_{N_p}$, 
	it will generate each \emph{unlabeled} graph several times.
In fact, the number of times an unlabeled $H \in \cM^p_{N_p}$
	appears is \emph{at least} the number of ear decompositions
	$H_0 \subset \cdots \subset H_{k} \subset H$
	which match the conditions of Corollary \ref{cor:twoears}.
	
We will remove these redundancies in two ways.
First, we will augment using pair orbits of vertices in $H^{(i)}$.
Second, we will \emph{reject} some augmentations
	if they do not correspond with a ``canonical" 
	ear decomposition of the larger graph.
This technique is described in detail in \cite{Stolee11}.	
	
	Let $\del(H)$ be a function which takes a graph $H \in \cM^p$
		and returns an ear $\eps$ in $H$
		so that $H - \eps$ is in $\cM^p$.
	This function $\del(H)$ is a \emph{canonical deletion}
		if for any two $H_1, H_2 \in \cM^p$ so that $H_1 \cong H_2$,
		there exists an isomorphism $\sigma : H_1\to H_2$
		that  maps $\del(H_1)$ to $\del(H_2)$.
	
	Given a canonical deletion $\del(H)$, 
		the \emph{canonical ear decomposition} at $H$
		is given by the following iterative construction:
		(i) Set $H_0 = H$ and $i = 0$.
		(ii) While $H^{(i)}$ is not a cycle, define $H_{i-1} = H^{(i)} - \del(H^{(i)})$ 
		and decrement $i$.
	When this process terminates,
		what results is an ear decomposition 
	$ H_{-k} \subset H_{-(k-1)} \subset \cdots \subset H_{-1} \subset H_0$
	where $H_{-k}$ is isomorphic to a cycle and $H_0 = H$.

A simple consequence of this definition is that
	if $H_{-1} = H - \del(H)$, then
	the canonical ear decomposition of $H$
	begins with the canonical ear decomposition of $H_{-1}$
	then proceeds with the augmentation $H_{-1} \subset H_{-1} + \del(H) = H$.
Applying isomorph-free generation algorithm of~\cite{Stolee11}
	will generate all unlabeled graphs in $\cM^p$
	without duplication 
	by generating ear decompositions using
	all possible ear augmentations
	and rejecting any augmentations
	which are not isomorphic to the canonical deletion.
	
In order to guarantee the canonical deletion $\del(H)$
	satisfies the isomorphism requirement,
	the choice will depend on a canonical labeling.
	A function $\lab(H)$ which takes a labeled graph $H$
		and outputs a bijection $\sigma_H : V(H) \to \{1,\dots,n(H)\}$
		is a \emph{canonical labeling} 
		if for all $H_1 \cong H_2$
		the map $\pi : V(H_1)\to V(H_2)$
		defined as $\pi(x) = \sigma_{H_2}^{-1}(\sigma_{H_1}(x))$
		is an isomorphism.
	The canonical labeling $\sigma_H = \lab(H)$ on the vertex set 
		induces a label $\gamma_H$ on the ears of $H$.
	Given an ear $\eps$ of order $r$ between endpoints $x$ and $y$,
		let $\gamma_H(\eps)$ be the triple $(r, \min\{\sigma_H(x),\sigma_H(y)\}, \max\{\sigma_H(x),\sigma_H(y)\})$.
	These labels have a natural lexicographic ordering
		which minimizes the order of an ear and then minimizes
		the pair of canonical labels of the endpoints.
In this work, the canonical labeling $\lab(H)$ 
	is computed using McKay's \texttt{nauty} library~\cite{nauty, HRnauty}.
We now describe the canonical deletion $\del(H)$ 
	which will generate a canonical ear decomposition
	matching Corollary \ref{cor:twoears} 
	whenever given a graph $H \in \cM^p$.
	
By the proof of Lemma \ref{lma:pmdeletionclosed},
	we need all almost 1-extendable graphs $H$ to have $H-\eps$ 
	be 1-extendable, but 1-extendable graphs $H$ may have $H-\eps$
	be 1-extendable or almost 1-extendable, depending on 
	availability.
Also, since we are only augmenting by ears of even order,
	we must select the deletion to have this parity.
	
The following sequence of choices 
	describe the method for selecting 
	a canonical ear to delete from a graph $H$ in $\cM_{N_p}^p$:

\begin{enumerate}
	\item If $H$ is almost 1-extendable, 
		select an ear $\eps$ so that $H - \eps$ is 1-extendable.
		By the definition of almost 1-extendable graphs, there is a unique such choice.
		
	\item If $H$ is 1-extendable and
		there exists an ear $\eps$ so that $H - \eps$ is 1-extendable,
			then select such an ear with minimum value $\gamma_H(\eps)$.
			
	\item If $H$ is 1-extendable and 
			no single ear $\eps$ has the deletion $H - \eps$ 1-extendable, then
			select an even-order ear $\eps$ so that there is a disjoint even-order ear $\eps'$ 
			so that $H-\eps$ is almost 1-extendable
			and
			$H - \eps - \eps'$ is 1-extendable.
			Out of these choices for $\eps$, 
				select $\eps$ with minimum value $\gamma_H(\eps)$.
\end{enumerate}

See \cite{Stolee11} for a more detailed description of the isomorph-free
	properties of the canonical deletion strategy.

The full generation algorithm, including augmentations,
	checking canonical deletions, as well as some 
	optimizations and pruning techniques, is described in Section \ref{ssec:results}.	
We now investigate how to find $p$-extremal elementary graphs 
	using 1-extendable graphs in $\cM^p$.
In the following section, we discuss how to fill a 1-extendable
	graph $H$ with free edges without increasing the number of perfect matchings.

\section{Structure of Free Subgraphs}
\label{ssec:barriers}

By Proposition \ref{prop:freeextendable},
	the free edges within an elementary graph
	have endpoints within a common barrier.
This implies that the structure of the free edges
	is coupled with the structure of barriers in $G$.
In this section, we demonstrate that 
	the structure of the free subgraph of
	a $p$-extremal elementary graph 
	depends entirely on the structure 
	of the barriers in the extendable subgraph.
This leads to a method to generate
	all maximal sets of free edges that can be added
	to a 1-extendable graph.
Section \ref{ssec:evolutionofbarriers} describes
	a method for quickly computing the list of barriers
	of a 1-extendable graph using
	an ear decomposition.
In particular, this provides an on-line algorithm
	which is implemented along with the generation
	of canonical ear decompositions.
Finally, Section \ref{ssec:pruning} 
	combines the results of these sections into 
	a very strict condition which is used to prune
	the search tree.

	Let $G$ be an elementary graph.
	The \emph{barrier set} $\cB(G)$ 
		is the set of all barriers in $G$.
	The \emph{barrier partition} $\cP(G)$ 
		is the set of all maximal barriers in $G$.
The following lemmas give some properties of
	$\cP(G)$ and $\cB(G)$ 
	when $G$ is elementary.

\begin{lemma}[Lemma 5.2.1~\cite{LovaszPlummer}]
	\label{lma:barrierpartition}
	For an elementary graph $G$, $\cP(G)$
		is a partition of $V(G)$.
\end{lemma}

\begin{lemma}[Theorem 5.1.6~\cite{LovaszPlummer}]
	\label{lma:allodd}
	For an elementary graph $G$ and $B \in \cB(G)$, $B \neq \emptyset$,
		all components of $G-B$ have odd order.
\end{lemma}

	Given an elementary graph $H$, let $\cE(H)$ be the set 
		of elementary supergraphs 
		with the same extendable subgraph:
		$\cE(H) = \left\{ G \supseteq H : V(G) = V(H), \Phi(G) = \Phi(H)\right\}$.
	We will refer to maximal elements of $\cE(H)$ using
		the subgraph relation $\subseteq$, giving a poset $(\cE(H),\subseteq)$.
	Note that $(\cE(H),\subseteq)$ has a unique
		minimal element, $H$.

\begin{proposition}\label{prop:barrierpartition}
	Let $H$ be a 1-extendable graph.
	If $G$ is a maximal element in $\cE(H)$,
		then every barrier in $\cP(G)$ is a clique of free edges in $G$.
\end{proposition}

\begin{proof}
	If some maximal barrier $X$ in $\cP(G)$ is not a clique, 
			then there is a missing edge $e$ between
			vertices $x,y$ of $X$.
	Since $|X| = \odd(G-X)$, 
		all perfect matchings of $G+e$ 
		must match at least one vertex of 
		each odd component to some vertex in $X$,
		saturating $X$.
	This means that $e$ is not extendable in $G+e$, and $G+e \in \cE(H)$.
	This contradicts that $G$ was maximal in $\cE(H)$.
	
	By Proposition \ref{prop:freeextendable}, 
		the edges within the barriers are free.
\end{proof}

\begin{lemma}
	\label{lma:barrierrefinement}
	Let $H$ be a 1-extendable graph and $G \in \cE(H)$
		be an elementary supergraph of $H$.
	Every barrier $B$ of $G$ is also a barrier of $H$.
\end{lemma}

\begin{proof}
	Each odd component of $G-B$ 
		is a combination of components of $H-B$,
		an odd number of which are odd components,
		giving $\odd(H-B) \geq \odd(G-B)$.
	There are no new vertices in $G$, so 
		the components of $G-B$ partition $V(H)-B$
		so that the partition of components of $H-B$ is a refinement of $G-B$.
		
	Since $B$ is a barrier of $G$, $\odd(G-B) = |B|$.
	Since $H$ is matchable, Tutte's Theorem implies
		$\odd(H-B) \leq |B|$.
	Thus $|B| = \odd(G-B) \leq \odd(H-B) \leq |B|$
	and equality holds, making $B$ a barrier of $H$.
\end{proof}

\def\cC{\mathcal{C}}
\def\cI{\mathcal{I}}

	Given a 1-extendable graph $H$, 
		 barriers $B_1$ and $B_2$ \emph{conflict} if
		(a) $B_1 \cap B_2 \neq \emptyset$,
		(b) $B_1$ spans multiple components of $H - B_2$, or
		(c) $B_2$ spans multiple components of $H - B_1$.
A set $\cI$ of barriers in $\cB(H)$
		is a \emph{cover set} if each pair $B_1, B_2$ of barriers in $\cI$
		are non-conflicting and $\cI$ is a partition of $V(H)$.
	Let $\cC(H)$ be the family of cover sets in $\cB(H)$.
	If $\cI_1, \cI_2 \in \cC(H)$ are cover sets,
		let the relation $\cI_1 \preceq \cI_2$
		hold
		if for each set $B_1 \in \cI_1$ there exists a set $B_2 \in \cI_2$
		so that $B_1 \subseteq B_2$.
	This defines a partial order on $\cC(H)$ and the poset $(\cC(H), \preceq)$
		has a unique minimal element given by the partition of $V(H)$ 
		into singletons.

\begin{theorem}
	\label{thm:barrierconflictgraph}
	Let $H$ be a 1-extendable graph.
	A graph $G \in \cE(H)$ is maximal in $(\cE(H),\subseteq)$
		if and only if
		each $B \in \cP(G)$ is a clique,
		$\cP(G)$ is a cover set, and
		$\cP(G)$ is maximal in $(\cC(H),\preceq)$.
\end{theorem}

\begin{proof}
	We define a bijection between
		$\cC(H)$
		and a subset of elementary supergraphs in $\cE(H)$,
		as given in the following claim.
		
	\begin{claim}
		\label{clm:coverbijection}
		Cover sets $\cI \in \cC(H)$
			are in bijective correspondence with 
			elementary graphs $G \in \cE(H)$
			where the free subgraph of $G$
			is a disjoint union of cliques,
			each of which is a (not necessarily maximal) barrier of $G$.
	\end{claim}
	
	Let $G$ be a graph in $\cE(H)$ 
		where the free subgraph of $G$
		is a disjoint union of cliques $X_1$, $X_2$, $\dots$, $X_k$,
		where each $X_i$ is a barrier of $G$.
	Then, let $\cI = \{ X_1, \dots, X_k\}$ be the set of barriers.
	Note that $\cI$ is a partition of $V(H)$,
		each part of which is a barrier of $G$
		which is a barrier of $H$ by Lemma \ref{lma:barrierrefinement}.
	Consider two barriers $B_1, B_2 \in \cI$.
	Since we selected $\cI$ to be a partition, $B_1 \cap B_2 = \emptyset$.
	If $B_2$ spans multiple components of $H-B_1$, 
		then the vertices from these components are a single
		component in $G-B_1$, where $B_2$ is a clique of edges.
	However, Lemma \ref{lma:allodd} gives that all components of
		$H-B_1$ and $G-B_1$ are odd, since $B_1$ is a barrier.
	This implies that $|B_1| = \odd(H-B_1) > \odd(G-B_1) = |B_1|$, 
		a contradiction. 
	Hence, $B_2$ is contained within a single component of $H-B_1$,
		so $B_1$ and $B_2$ do not conflict in $H$.
	This gives that $\cI$ is a cover set in $\cC(H)$.
		
	This map from $G \in \cE(H)$ to $\cI \in \cC(H)$
		is invertible by taking 
		a cover set $\cI \in \cC(H)$
		and filling each barrier $B \in \cI$ with edges, 
		forming a graph	$H_{\cI}$.
	Since each pair of barriers $B_1, B_2$ in $\cI$ are non-conflicting,
		the components of $H-B_1$ do not change by 
		adding edges between vertices in $B_2$.
	Therefore, each set $B \in \cI$ is also a barrier in $H_{\cI}$.
	By Proposition \ref{prop:freeextendable}, 
		the edges within each barrier of $H_{\cI}$ are free,
		so all extendable edges of $H_{\cI}$ are exactly those in $H$.
	This gives that $\Phi(H_{\cI}) = \Phi(H)$ and $H_{\cI} \in \cE(H)$.
	The map from $\cI$ to $H_{\cI}$ is the inverse of the 
		earlier map from $G \in \cE(H)$ with free edges forming disjoint cliques
		to $\cI \in \cC(H)$.
	Hence, this is a bijection, proving the claim.
		
	An important point in the previous claim is that the free edges
		formed cliques which are barriers, but those 
		cliques were not necessarily \emph{maximal} barriers.	
	We now show that the above bijection maps edge-maximal graphs in 
		$\cE(H)$ to maximal cover sets in $\cC(H)$.
	
	\begin{claim}
		Let $\cI$ be a cover set in $\cC(H)$.
		The following are equivalent:
		\begin{cem}
			\item[(i)] $\cI$ is maximal in $(\cC(H),\preceq)$.
			\item[(ii)] $H_{\cI}$ is maximal in $(\cE(H),\subseteq)$.
			\item[(iii)] $\cP(H_\cI) = \cI$.
		\end{cem}	
	\end{claim}
	
	(ii) $\Rightarrow$ (iii) This is immediate from Proposition \ref{prop:barrierpartition}.
	
	(iii) $\Rightarrow$ (ii) If $\cP(H_{\cI})$, then
		any edge $e \notin E(H_{\cI})$ must span
		two maximal barriers.
		By Proposition \ref{prop:freeextendable},
		$e$ is allowable in $H_{\cI}+e$, 
			so $H_{\cI}$ is maximal in $(\cE(H),\subseteq)$.
	
	(i) $\Rightarrow$ (ii) Let $\cI$ be a maximal cover set of
		barriers in $\cB(H)$ and $H_{\cI}$ the corresponding elementary supergraph
		in $\cE(H)$.
	Suppose there exists a supergraph $H'\supset H_{\cI}$ in $\cE(H)$.
	Then, there is an edge $e$ in $E(H') \setminus E(H_{\cI})$ so that 
		$e$ is free in $H_{\cI} + e$.
	This implies that $e$ spans vertices within the same barrier $B$ of $H_{\cI}+e$ 
		(by Proposition \ref{prop:freeextendable}),
		and $B$ is also a barrier of $H_{\cI}$.
	However, $B$ is split into $k$ barriers $B_1, \dots, B_k$ in $\cI$, 
		for some $k \geq 2$.
	Therefore, the set $\cI' = (\cI \setminus \{B_1,\dots, B_k\})\cup \{B\}$
		is a refinement of $\cI$.
		
	We now show that $\cI'$ is a cover set in $\cC(H)$.
	Note that any two barriers $X_1, X_2 \in \cI'$
		where neither is equal to $B$ is still non-conflicting.
	For any barrier $X \neq B$ in $\cI'$, 
		notice that $X$ does not span more than one component
		of $H - B$, since $B$ is a barrier in $H_{\cI}$ and $H_{\cI'}$.
	Also, if $B$ spanned multiple components of $H-X$, then
		those components would be combined in $H_{\cI'}-X$, 
		but since $X$ is a barrier, $|X| = \odd(H_{\cI'}-X) \leq \odd(H-X) = |X|$.
	Therefore, $B$ does not conflict with any other barrier $X$ in $\cI'$
		giving $\cI'$ is a cover set and $\cI \preceq \cI'$.
	This contradicts maximality of $\cI$, so $H_{\cI}$ is maximal.

	(ii) $\Rightarrow$ (i)	
	By (iii), $\cI = \cP(H_{\cI})$.
	Let $\cI'$ be a cover set so that 
		$\cI \preceq \cI'$.
	$\cI'$ also partitions $V(H)$, so
		$\cP(H_{\cI})$ is a refinement of $\cI'$.
	Then, the graph $H_{\cI'}$ is a proper supergraph of $G$.
	By the maximality of $G$, $H_{\cI'}$ must not be an elementary supergraph
		in $\cE(H)$.
	By the bijection of Claim \ref{clm:coverbijection}, 
		$\cI'$ must not be a cover set
		of $H$.
	Therefore, $\cI$ is a maximal covering set in $\cC(H)$.

	This proves the claim and the theorem follows.
\end{proof}

The previous theorem provides a method to search for the maximum elements 
	of $\cE(H)$ by generating all cover sets $\{B_1,\dots,B_k\}$ in $\cC(H)$
	and 
	maximizing the sum $\sum_{i=1}^k {|B_i| \choose 2}$.
	
The {\naive} independent set generation algorithm 
	runs with an exponential blowup on the number of barriers.
This can be remedied in two ways.
First, we notice empirically that the number of barriers frequently drops
	as more edges and ears are added,
	especially for dense extendable graphs.
Second, the number of barriers is largest when
	the graph is bipartite, as there are exactly two maximal barriers
	each containing half of the vertices,
	with many subsets which are possibly barriers.
We directly adress the case when $H$ is bipartite
	as there are exactly two maximum elements of $\cE(H)$.

\begin{lemma}[Corollary 5.2~\cite{HSWY}]
	\label{lma:maxfree}
	The maximum number of free edges in an elementary graph
		with $n$ vertices is ${n/2\choose 2}$.
\end{lemma}
	
Not only is this a general bound, but
	it is attainable for bipartite graphs.
In a bipartite graph $H$, there are exactly two 
	graphs in $\cE(H)$ which attain this number of free edges.

\begin{lemma}
	\label{lma:bipartiteextremal}
	If $H$ is a bipartite 1-extendable graph,
		then there are exactly two maximal barriers, $X_1$ and $X_2$.
	Also, there
		are exactly two maximum elements $G_1, G_2$ of $\cE(H)$.
	Each graph $G_i$ is given by adding all possible edges within $X_i$.
\end{lemma}

\begin{proof}
	Let $X_1$ and $X_2$ be the two sides of the bipartition of $H$.
	Since $H$ is matchable, $|X_1| = |X_2|$ and $V(H - X_1) = X_2$ and $V(H-X_2) = X_1$.
	Thus $X_1$ and $X_2$ are both barriers which partition $V(H)$
		and by Lemma \ref{lma:barrierpartition}
		these must be the maximal barriers of $H$.
	
	The sets $\cI_1 = \{X_1\} \cup \{\{v\} : v \in X_2\}$ 
		and $\cI_2 = \{X_2\} \cup \{\{v\} : v \in X_1\}$
		are maximal
		cover sets in $\cC(H)$.
	Using the bijection of Theorem \ref{thm:barrierconflictgraph},
		$\cI_1$ corresponds with the maximal elementary graph $G_1$ in $\cE(H)$
		where all possible edges are added to $X_1$.
	Similarly, $\cI_2$ corresponds to adding all possible edges to $X_2$,
		producing $G_2$.
	Each of these graphs has ${n(H)/2\choose 2}$ free edges,
		the maximum possible for graphs in $\cE(H)$
		by Lemma \ref{lma:maxfree}.
		
	We must show that any other graph $G$ in $\cE(H)$ 
		has fewer free edges.
	We again use the bijection of Theorem \ref{thm:barrierconflictgraph}
		in order to obtain a maximal cover set $\cI$ 
		in $\cB(H)$ which are filled with free edges in $G$.
	Then, the number of free edges in $G$ 
		is given by $s(\cI) = \sum_{B \in \cI} {|B|\choose 2}$.
	
	Without loss of generality,
		the barrier $A$ of largest size within $\cI$
		is a subset of $X_1$.
	For convenience, we use $m = n(H)/2$ to be the size of each 
		part $X_1, X_2$ and $k = |A|$, with $1 \leq k < m$.
	Note that in $H_{\cI}$, no free edges 
		have endpoints in both $A$ and $X_1 \setminus A$,
		leaving at least $k(m-k) = mk-k^2$ 
		fewer free edges within $X_1$ in $G$ than in $G_1$.
	If $H_{\cI}$ has ${n(H)/2\choose 2}$ edges, 
		then the barriers in $X_2$ add at least $mk-k^2$ free edges.
	
	The problem of maximizing $s(\cI)$ over all maximal cover sets
		can be relaxed to a linear program
		with quadratic optimization function as follows:
	First, fix the barriers of $\cI$ within $X_1$, including the largest barrier, $A$.
	Then, fix the number of barriers of $\cI$ within $X_2$ to be some integer $\ell$. 
	Then, let $\{B_1,\dots,B_\ell\}$ be the list of barriers 
		in $X_2$.
	Now, create variables $x_i = |B_i|$ for all $i \in \{1,\dots, \ell\}$.
	The barriers in $X_1$ are fixed, so to maximize $s(\cI)$,
		we must maximize $\sum_{i=1}^\ell {x_i \choose 2}$.
	
	We now set some constraints on the $x_i$.
	Since the barriers $B_i$ are not empty, we require $x_i \geq 1$.
	Since $B_i$ does not conflict with $A$, 
		each $B_i$ is within a single component of $H-A$.
	Since there are $|A|$ such components, there are at least $|A|-1$ other vertices
		in $X_2$ that are not in $B_i$, giving $x_i \leq m-k+1$.
	Also, since $A$ is the largest barrier, $x_i \leq k$.
	Finally, the barriers $B_i$ partition $X_2$, giving $\sum_{i=1}^\ell x_i = m$
		and that there are is at least one barrier per component, giving
		$\ell \geq k$.
		
	Since for $x < y$, ${x-1\choose 2} + {y+1\choose 2} > {x\choose 2} + {y\choose 2}$,
		optimium solutions to this linear program have maximum value 
		when the maximum number of variables have maximum feasible value.
	Suppose $1 \leq x_i \leq t$ are the tightest bounds on the variables $x_1,\dots,x_\ell$.
	Then $\frac{m-\ell}{t-1}{t\choose 2}$ is an upper bound on the
		value of the system.
	
	\case{1} Suppose $k \geq m-k+1$.
		Now, the useful constraints are 
	$\sum_{i=1}^\ell x_i = m, 1 \leq x_i \leq m-k+1$
	and we are trying to maximize $\sum_{i=1}^\ell {x_i \choose 2}$.
	The optimum value
		is bounded by $\frac{m-\ell}{m-k}{m-k+1\choose 2}$.
	As a function of $\ell$, 
		this bound is maximized by the smallest feasible value of $\ell$,
		being $\ell = k$.
	Hence, we have an optimum value at most
		$\frac{m-k}{m-k}\frac{(m-k+1)(m-k)}{2} =  \frac{1}{2}m\left(m+1\right) - \left(m+\frac{1}{2}\right)k - \frac{1}{2}k^2$.
	Since $k \geq m-k+1$,
		the inequality $k \geq \frac{1}{2}(m+1)$ holds, and
	(ii) $\Rightarrow$ (iii) This is immediate from Proposition \ref{prop:barrierpartition}.
	
		the optimum value of this program is at most 
		
	\[   \frac{1}{2}m\left(m+\frac{1}{2}\right) - \left(m+\frac{1}{2}\right)k - \frac{1}{2}k^2
		 \leq 
		\underbrace{mk}_{k \geq \frac{1}{2}(m+1)} 
		- \underbrace{k^2}_{k \leq m} - \frac{1}{2}k^2 < mk-k^2.\]
	
	Therefore, $H_\cI$ must not have ${n(H)/2\choose 2}$ free edges.
	
	\case{2} Suppose $k < m-k+1$.
		The constraints are now
	$ \sum_{i=1}^\ell x_i = m, 1 \leq x_i \leq k$
	while maximizing $\sum_{i=1}^\ell {x_i \choose 2}$.
	This program has optimum value bounded above by
		$\frac{m-\ell}{k-1}{k\choose 2}$,
		which is maximized by the smallest feasible value of $\ell$.
	If $m/k > k$ and $\ell < m/k$, the program is not even feasible,
		as a sum of $\ell$ integers at most $k$ could not sum to $m$.
	Hence, $\ell \geq \max\{ k, m/k\}$.
		
	\subcase{2.a} Suppose $k \geq m/k$.
		Setting $\ell = k$ gives a bound of
		$\frac{m-k}{k-1}{k\choose 2} = \frac{1}{2}(mk-k^2)$.
		This is clearly below $mk-k^2$, so
			$H_{\cI}$ does not have ${n(H)/2\choose 2}$ free edges.
			
	\subcase{2.b} Suppose $k < m/k$.
		Setting $\ell = \lceil m/k\rceil $ gives a bound of
			$\frac{m-\lceil m/k\rceil}{k-1}{k\choose 2} = %
				\frac{1}{2}(mk-m)$.
		Since $k < m/k$, $k^2 < m$ and
			$\frac{1}{2}(mk - m) \leq mk-k^2$.
		Hence, $H_{\cI}$ does not have ${n(H)/2\choose 2}$ free edges.
\end{proof}
		
Experimentation over the graphs used during the
	generation algorithm for $p$-extremal graphs
	shows that
	a {\naive} generation of cover sets in $\cC(H)$
	is sufficiently fast to compute the maximum
	excess in $\cE(H)$
	when the list of barriers $\cB(H)$ is known.
The following section describes a method for computing $\cB(H)$
	very quickly using the canonical ear decomposition.

\section{The Evolution of Barriers}
\label{ssec:evolutionofbarriers}

In this section, 
	we describe a method to efficiently compute 
	the barrier list $\cB(H)$
	of a 1-extendable graph $H$
	utilizing a graded ear decomposition.
Consider a non-refinable graded ear decomposition 
	$H_0 \subset H_1 \subset H_2 \subset \cdots \subset H_k = H$
	of a 1-extendable graph $H$
	starting at a cycle $C_{2\ell} = H_0$.
Not only are the maximal barriers of $C_{2\ell}$ easy to compute (the 
	sets $X, Y$ forming the bipartition) 
	but also the barrier list (every non-empty subset of $X$ and $Y$ is a barrier).

\begin{lemma}\label{lma:pushdownbarrier}
	Let $H^{(i)} \subset H^{(i+1)}$ be a non-refinable ear decomposition of 
		a 1-extendable graph $H^{(i+1)}$ from a 
		1-extendable graph $H^{(i)}$
		 using one or two ears.
	If $B'$ is a barrier in $H^{(i+1)}$, 
		then $B = B' \cap V(H)$ is a barrier in $H^{(i)}$.
\end{lemma}

\begin{proof}
	There are $|B'|$ odd components in $H^{(i+1)} - B'$.
	There are at most $|B|$ odd components in $H^{(i)} - B$,
		which may combine when the ear(s) are added to make $H^{(i+1)}$.
	
	Let $x_1,x_2,\dots,x_r$ be the vertices 
		in $B' \setminus B$.
	Each $x_i$ is not in $V(H^{(i)})$ so it is an internal vertex
		of an augmented ear.
	Therefore, $x_i$ has degree two in $H^{(i+1)}$,
		so removing $x_i$ from $H^{(i+1)} - (B \cup \{x_1,\dots,x_{i-1}\})$
		increases the number of odd components by at most one.
	Hence, the number of odd components of $H^{(i+1)} - B'$
		is at most the number of odd components of $H^{(i)}-B$
		plus the number of vertices in $B' \setminus B$.
	These combine to form the inequalities
	\[ |B'| = \odd(H^{(i+1)}-B') 
		\leq \odd(H^{(i+1)}-B) + |B'\setminus B| 
		\leq \odd(H^{(i)}-B) + |B'\setminus B| 
		\leq |B| + |B'\setminus B| = |B'|.\]
	Equality holds above, so $B$ is a barrier in $H^{(i)}$.
\end{proof}	
	
As one-ear augmentations and two-ear augmentations are applied to each $H^{(i)}$,
	we update the list $\cB(H^{(i+1)})$ of barriers in $H^{(i+1)}$
	using the list $\cB(H^{(i)})$ of barriers in $H^{(i)}$.
	
\begin{lemma}\label{lma:evolutionofbarriersEVEN}
	Let $B$ be a barrier of a 1-extendable graph $H^{(i)}$.
	Let $H^{(i)} \subset H^{(i+1)}$ be a 1-extendable ear augmentation
		of $H^{(i)}$ using one ($\eps_1$) or two ($\eps_1, \eps_2$) ears.
	\begin{cem}
		\item If any augmenting ear connects
				vertices from different components of $H^{(i)}-B$,
				then $B$ is not a barrier in $H^{(i+1)}$,
				and neither is any $B' \supset B$
				where $B = B' \cap V(H^{(i)})$.
		\item Otherwise, if $B$ does not contain any endpoint of
				the augmented ear(s), then 
				$B$ is a barrier of $H^{(i+1)}$,
				but $B \cup S$ for any non-empty subset
				$S \subseteq V(H^{(i+1)})\setminus V(H^{(i)})$
				is not a barrier of $H^{(i+1)}$.
		\item If $B$ contains both endpoints of some ear $\eps_i$, then
			$B$ is not a barrier in $H^{(i+1)}$ and neither is any $B' \supset B$.
			
		\item If $B$ contains exactly one endpoint ($x$) of one of the augmented ears
			($\eps_j$),
			then 
			\begin{cem}
				\item $B$ is a barrier of $H^{(i+1)}$.
				\item For $S \subseteq V(H^{(i+1)})\setminus V(H^{(i)})$, 
						$B \cup S$ is a barrier of $H^{(i+1)}$ if and only if
						$S$ contains only internal vertices of $\eps_j$ of even distance
						from $x$ along $\eps_j$.
			\end{cem}
		\item If $B = \emptyset$, 
				then for any subset $S \subseteq V(H^{(i+1)})\setminus V(H^{(i)})$
				$B \cup S$ is a barrier of $H^{(i+1)}$ 
				if and only if
				the vertices in $S$ are on a single
				ear $\eps_j$ and the pairwise distances along
				$\eps_j$ are even.
		\end{cem}
\end{lemma}

\begin{proof}
	Let $B'$ be a barrier in $H^{(i+1)}$.
	Lemma \ref{lma:pushdownbarrier} gives
		$B = B' \cap V(H^{(i)})$ is a barrier of $H^{(i)}$,
		and $H^{(i)} - B$ has $|B|$ odd connected components.
	Thus, the barriers of $H^{(i+1)}$ are built from barriers $B$ in $H^{(i)}$
		and adding edges from the new ear(s).
	
	\case{1} If an ear $\eps_j$ spans two components of $H^{(i)} - B$,
		then the number of components in $H^{(i+1)} - B$
		is at most $|B| - 2$.
	Any vertices from $\eps_j$ added to $B$ can only increase
		the number of odd components by at most one at a time,
		but also increases the size of $B$ by one.
	Hence, vertices in $V(H^{(i+1)})\setminus V(H^{(i)})$ can be added
		to $B$ to form a barrier in $H^{(i+1)}$.
		
	\case{2} If each ear $\eps_j$ spans points in the same components of $H^{(i)} - B$,
		then the number of odd connected components in $H^{(i+1)} - B$ 
		is the same as in $H^{(i)}-B$, which is $|B|$.
	Hence, $B$ is a barrier of $H^{(i+1)}$.
	However, adding a single vertex from any $e_i$
		does not separate any component of $H^{(i+1)} - B$, 
		but adds a count of one to $|B|$.
	Adding any other vertices from $\eps_j$ 
		to $B$ can only increase the number of components by one
		but increases $|B|$ by one.
	Hence, no non-empty set of vertices from the augmented ears
		can be added to $B$ to form a barrier of $H^{(i+1)}$.
	
	\case{3} Suppose $B$ contains both endpoints of an ear $\eps_j$.
		If $\eps_j$ is a trivial ear, then it is an extendable edge.
		If $B'\supseteq B$ is a barrier in $H^{(i+1)}$, 
			this violates Proposition \ref{prop:freeextendable}
			which states edges within barriers are free edges.
		If $\eps_j$ has internal vertices, they 
			form an even component in $H^{(i+1)} - B$.
		By Lemma \ref{lma:allodd},
			this implies that $B$ is not a barrier.
		Any addition of internal vertices from $\eps_j$
			to form $B' \supset B$ will add at most one odd component each,
			but leave an even component in $H^{(i+1)} - B'$.
		It follows that no such $B'$ is a barrier in $H^{(i+1)}$.
	
	\case{4} Note that  If an ear $\eps_j$ has an endpoint within $B$,
			then in $H^{(i+1)} - B$, the internal vertices of $\eps_j$
			are attached to the odd component of $H^{(i+1)}-B$
			containing the opposite endpoint.
		Since there are an even number of internal vertices on $\eps_j$,
			then $H^{(i+1)}-B$ has the same number of odd connected components 
			as $H^{(i)}-B$, which is $|B|$.
		Hence, $B$ is a barrier in $H^{(i+1)}$.
		
		Let the ear $\eps_j$ be given as a path of vertices $x_0x_1x_2\dots x_k$, 
			where $x_0 = x$ and $x_k$ is the other endpoint of $\eps_j$.
		Let $S$ be a subset of $\{x_1,\dots,x_{k-1}\}$,
			the internal vertices of $\eps_j$.
		The number of components given by removing $S$ 
			from the path $x_1x_2\cdots x_{k-1}x_k$
			is equal to the number of \emph{gaps} in $S$:
			the values $a$ so that $x_a$ is in $S$ 
			and $x_{a+1}$ is not in $S$.
		These components are all odd if and only if 
			for each $x_a$ and $x_{a'}$ in $S$,
			$|a - a'|$ is even.
		Thus, $B \cup S$ is a barrier in $H^{(i+1)}$ if and only if
			$S$ is a subset of the internal vertices
			which are an even distance from $x_0$.
\end{proof}

Lemma \ref{lma:evolutionofbarriersEVEN}
	describes all the ways a barrier $B \in \cB(H)$ 
	can extend to a barrier $B' \in \cB(H+\eps_1)$ or $B' \in \cB(H+\eps_1+\eps_2)$.
Note that the barriers $B'$ which use the internal vertices of $\eps_1$ are 
	independent of those which use the internal vertices of $\eps_2$,
	unless one of the ears spans multiple components of $H+\eps_1+\eps_2 - B'$.
This allows us to define a \emph{pseudo-barrier list} $\cB(H+\eps)$ for
	almost 1-extendable graphs $H+\eps$, where $H$ is 1-extendable.
During the generation algorithm, we consider a single-ear augmentation
	$H^{(i)} \subset H^{(i)} + \eps_i = H^{(i+1)}$.
Regardless of if $H^{(i)}$ or $H^{(i+1)}$ is almost 1-extendable,
	we can update $\cB(H^{(i+1)})$ by 
	taking each $B \in \cB(H^{(i)})$ 
	and adding each $B \cup S$ that satisfies Lemma \ref{lma:evolutionofbarriersEVEN}
	to $\cB(H^{(i+1)})$. 
This process generates all barriers $B' \in \cB(H^{(i+1)})$ so that
	$B' \cap V(H^{(i)}) = B$, so each barrier is generated exactly once.

In addition to updating the barrier list in an ear augmentation $H^{(i)} \subset H^{(i+1)}$, 
	we determine the conflicts between these barriers.

\begin{lemma}\label{lma:conflictsinear}
	Let $H^{(i)} \subset H^{(i+1)}$ be a 1-extendable ear augmentation using one $(\eps_1)$ 
		or two $(\eps_1,\eps_2)$ ears.
	Suppose $B_1'$ and $B_2'$ are barriers in $H^{(i+1)}$
		with barriers $B_1 = V(H^{(i)})\cap B_1'$ and $B_2= V(H^{(i)})\cap B_2'$
		of $H^{(i)}$.
	The barriers $B_1'$ and $B_2'$ conflict in $H^{(i+1)}$ if and only if 
		one of the following holds:
		(1) $B_1' \cap B_2' \neq \emptyset$,
		(2) $B_1$ and $B_2$ conflict in $H^{(i)}$, or
		(3) $B_1'$ and $B_2'$ share vertices in some ear ($\eps_j$), 
			with vertices $x_0x_1x_2\dots x_k$,
			and there exist indices $0 \leq a_1 < a_2 < a_3 < a_4 \leq k$ so that 
			 $x_{a_1}, x_{a_3} \in B_1'$ and $x_{a_2}, x_{a_4} \in B_2'$.
\end{lemma}

\begin{proof}
	Note that by definition, if $B_1'\cap B_2' \neq\emptyset$, 
		then $B_1'$ and $B_2'$ conflict.
	We now assume that $B_1' \cap B_2' = \emptyset$.
	
	If $B_1$ or $B_2$ conflict in $H^{(i)}$, then without loss of generality,
		$B_2$ has vertices in multiple components of $H^{(i)}-B_1$.
	Since $B_1'$ is a barrier in $H^{(i+1)}$, 
		Lemma \ref{lma:evolutionofbarriersEVEN} 
		gives that no ear $\eps_j$ spans multiple components
		of $H^{(i)}-B_1$, and the components of $H^{(i)}-B_1$
		correspond to components of $H^{(i+1)}-B_1$.
	Hence, $B_2$ also spans multiple components 
		of $H^{(i+1)}-B_1$ and $B_1'$ and $B_2'$ 
		conflict in $H^{(i+1)}$.

	Now, consider the case that the disjoint barriers $B_1$ and $B_2$ 
		did not conflict in $H^{(i)}$.
	Since $B_1$ and $B_2$ are barriers of $H^{(i)}$, then
		the vertices in $B_1'\setminus B_1$ are 
		limited to one ear $\eps_{j_1}$ of the augmentation,
		and similarly the vertices of $B_2'\setminus B_2$
		are within a single ear $\eps_{j_2}$.
	Since $B_1$ and $B_2$ do not conflict, all of the vertices within
		$B_2$ lie in a single component of $H^{(i)} - B_1$:
		the component containing the ear $\eps_{j_1}$.
	Similarly, the vertices of $B_1$ 
		are contained in the component of $H^{(i)} - B_2$
		that contains the endpoints of $\eps_{j_2}$.
		
	The components of $H^{(i+1)} - B_1$ 
		are components in $H^{(i+1)}-B_1'$
		except the component containing
		the ear $\eps_{j_1}$
		is cut into smaller components
		for each vertex in $\eps_{j_1}$ and $B_1'$.
	In order to span these new components, 
		$B_2'$ must have a vertex within $\eps_{j_1}$.
	Therefore, the ears $\eps_{j_1}$ and $\eps_{j_2}$ are the same ear,
		given by vertices $x_0, x_1, \dots, x_{k}$.

	Suppose there exist indices $0 \leq a_1 < a_2 < a_3 < a_4 \leq k$ so that
		$x_{a_1}$ and $x_{a_3}$ are in $B_1'$ and
		$x_{a_2}$ and $x_{a_4}$ are in $B_2'$.
	Then, the vertices $x_{a_1}$ and $x_{a_3}$ of $B_1'$
		are in different components of $H^{(i+1)}-B_2'$, since
		every path from $x_{a_3}$ to $x_{a_1}$ in 
		$H^{(i+1)}$ passes through one of the vertices $x_{a_2}$ or $x_{a_4}$.
	Hence, $B_1'$ and $B_2'$ conflict.
	
	If $B_1'$ and $B_2'$ do not admit such indices $a_1, \dots, a_4$, then
		listing the vertices $x_0, x_1, x_2, \dots, x_k$ in order
		will visit those in $B_1'$ and $B_2'$ in two contiguous blocks.
	In $H^{(i+1)}-B_1'$, the block containing the vertices in $B_2'$ 		
		remain connected to the endpoint closest to the block, and
		hence $B_2'$ will not span more than one component of $H^{(i+1)} - B_1'$.
	Similarly, $B_1'$ will not span more than one component of $H^{(i+1)} - B_2'$.
	$B_1'$ and $B_2'$ do not conflict in this case.
\end{proof}

The following corollary is
	crucial to the bound in Lemma \ref{lma:pruning}.

\begin{corollary}
	\label{cor:maxbarrierinear}
	Let $H^{(i)} \subset H^{(i+1)}$ be a 1-extendable ear augmentation using one $(\eps_1)$ 
		or two $(\eps_1,\eps_2)$ ears.
	Let $\cI$ be a maximal cover set in $\cC(H^{(i+1)})$
		and
		$S$ be the set of internal vertices $x$ of an ear $\eps_j$
		such that the barrier in $\cI$ containing $x$
		has at least one vertex in $V(H^{(i)})$.
	Then, $S$ contains at most half of the internal vertices of $\eps_j$.
\end{corollary}

\begin{proof}
	Let $A' \subset \cI$ be the set of barriers containing a vertex $x$ in $\eps_j$ and 
		a vertex $y$ in $V(H^{(i)})$.
	For a barrier $B$ to contain an internal vertex of $\eps_j$
		and a vertex in $V(H^{(i)})$, Lemma \ref{lma:evolutionofbarriersEVEN}
		states that $B$ must contain
		at least one of the endpoints of the ear $\eps_j$.
	Since each barrier in $A'$ contains and endpoint of $\eps_j$ and 
		non-conflicting barriers are non-intersecting, there are at most
		two barriers in $A'$.
	
	If there is exactly one barrier $B$ in $A'$, 
		by Lemma \ref{lma:evolutionofbarriersEVEN}
		it must contain vertices an even distance away from the endpoint
		contained in $B$, and hence at most half of the internal vertices 
		of $\eps_j$ are contained in $B$.
	
	If there are two non-conflicting barriers $B_1$ and $B_2$ in $A'$,
		then by Lemma \ref{lma:conflictsinear}
		the vertices of $B_1$ and $B_2$ within $\eps_j$
		come in two consecutive blocks along $\eps_j$.
	Since each barrier includes only vertices of even distance apart, 
		$B_1$ contains at most half of the vertices in one block
		and $B_2$ contains at most half of the vertices in the other block.
	Hence, there are at most half of the internal vertices of $\eps_j$ in $S$.
\end{proof}


\section{Bounding the maximum reachable excess}
\label{ssec:pruning}

In order to prune search nodes, we wish to detect when	
	it is impossible to extend the current 1-extendable graph $H$
	with $q$ perfect matchings to 
	a 1-extendable graph $H'$ with $p$ perfect matchings
	so that $H'$ has an elementary supergraph $G' \in \cE(H')$
	with excess $c(G') \geq c$.
The following lemma gives a method for 
	bounding $c(G')$ using
	the maximum excess $c(G)$ over all elementary supergraphs $G$ in $\cE(H)$.

\begin{lemma}\label{lma:pruning}
	Let $H$ be a 1-extendable graph on $n$ vertices with $\Phi(H) = q$. 
	Let $H'$ be a 1-extendable supergraph of $H$ 
		built from $H$ by a graded ear decomposition.
	Let $\Phi(H') = p > q$
		and $N = n(H')$.
	Choose $G \in \cE(H)$ and $G' \in \cE(H')$ 	
		with the maximum number of edges in each set.
	Then, \[c(G') \leq c(G) + 2(p-q) - \frac{1}{4}(N-n)(n-2).\]
\end{lemma}

\begin{proof}
	Let
	\[ H = H^{(0)} \subset H^{(1)} \subset \cdots \subset H^{(k-1)} \subset H^{(k)} = H'\]
	be a non-refinable graded ear decomposition 
		as in Theorem \ref{thm:lovasztwoears}.
	For each $i \in \{0,1,\dots, k\}$, let $G^{(i)} \in \cE(H^{(i)})$ be of maximum size.
	Without loss of generality, assume $G^{(0)} = G$ and $G^{(k)} = G'$.
	The following claims bound the excess $c(G^{(i)})$ in terms of $c(G^{(i-1)})$
		using the ear augmentation $H^{(i-1)} \subset H^{(i)}$.
	
\begin{claim}\label{clm:singleear}
	If $H^{(i-1)} \subset H^{(i)}$ is a single ear augmentation $H^{(i)} = H^{(i-1)} + \eps_1$
		where $\eps_1$ has order $\ell^{(i)}$, then
		\[c(G^{(i)}) \leq c(G^{(i-1)}) + 1 + \frac{3}{4}\ell^{(i)} - \frac{1}{8}(\ell^{(i)})^2 - \frac{1}{4}\ell^{(i)}n(H^{({i-1})}).\]
\end{claim}	
	
	
	By Lemma \ref{lma:oneartype},
		$\eps_1$ spans two maximal barriers $X, Y \in \cP(H^{(i)})$.
	$H^{(i)}$ has $\ell^{(i)}+1$ more extendable edges than $H^{(i-1)}$.
	
	We now bound the number of free edges $G^{(i)}$ has
		compared to the number of free edges in $G^{(i-1)}$.
	The elementary supergraph $G^{(i)}$ 
		has a clique partition of free edges
		given by a maximal cover set $\cI$ in $\cC(H^{(i)})$.
	For each barrier $B \in \cI$, the set $B \cap V(H^{(i-1)})$
		is also a barrier of $H^{(i-1)}$, by Lemma \ref{lma:pushdownbarrier}.
	Through this transformation, the maximal cover set $\cI$ 
		admits an cover set $\cI' = \{ B \cap V(H^{(i-1)}) : B \in \cI\}$
		in $\cC(H^{(i-1)})$.
	This cover set $\cI'$ generates an elementary supergraph $G^{(i-1)}_* \in \cE(H^{(i-1)})$ 
		through the bijection in Claim \ref{clm:coverbijection}.
	The free edges in $G^{(i)}$ 
		which span endpoints within $V(H^{(i-1)})$
		are exactly the free edges of $G_*^{(i-1)}$.
	By the selection of $G^{(i-1)}$, $e(G_*^{(i-1)}) \leq e(G^{(i-1)})$.
	
	When $\ell^{(i)}  > 0$,
		the $\ell^{(i)}$ internal vertices of $\eps_1$ may be incident to
		free edges whose other endpoints
		lie in the barriers $X$ and $Y$.
	By Corollary \ref{cor:maxbarrierinear},
		at most half of the vertices in $\eps_1$
		have free edges to vertices in $X$ and $Y$.
	Since the barriers $X$ and $Y$ are in $H^{(i-1)}$, they have size at most $\frac{n(H^{(i-1)})}{2}$.
	So, there are at most $\frac{\ell^{(i)}}{2}\frac{n(H^{(i-1)})}{2}$	
		free edges between these internal vertices and 
		the rest of the graph.
	Also, there are at most ${\ell^{(i)}/2\choose 2} = \frac{1}{8}(\ell^{(i)})^2 - \frac{1}{4}\ell^{(i)}$
		free edges between the internal vertices themselves.
	Combining these edge counts leads to the following inequalities:
		\begin{align*}
			c(G^{(i)}) 
				&= e(G^{(i)}) - \frac{(n(H^{(i-1)})+\ell^{(i)})^2}{4}\\
				&\leq \left[e(G_*^{(i-1)})+\left(1+\ell^{(i)}\right)
					+ \frac{n(H^{(i-1)})\ell^{(i)}}{2}
					+ \frac{1}{8}(\ell^{(i)})^2 - \frac{1}{4}\ell^{(i)} \right]\\
				&\qquad\qquad	- \left[\frac{n(H^{(i-1)})^2}{4}
					+ \frac{n(H^{(i-1)})\ell^{(i)}}{2} 
					+ \frac{(\ell^{(i)})^2}{4}  \right]\\
				&\leq e(G^{(i-1)})+1+\frac{3}{4}\ell^{(i)} - \frac{n(H^{(i-1)})^2}{4} - \frac{1}{8}(\ell^{(i)})^2\\
				&= c(G^{(i-1)}) 
						+ 1 
						+ \frac{3}{4}\ell^{(i)} 
						- \frac{1}{8}(\ell^{(i)})^2
					    - \frac{1}{4}\ell^{(i)}n_{i-1}.
		\end{align*}
	This proves Claim \ref{clm:singleear}.
	We now investigate a similar bound for two-ear autmentations.
	
\begin{claim}\label{clm:doubleear}
	Let $H^{(i-1)} \subset H^{(i)}$ be a two-ear augmentation 
		$H^{(i)} = H^{(i-1)} + \eps_1 + \eps_2$
		where the ears $\eps_1$ and $\eps_2$ have $\ell_1^{(i)}$ and $\ell_2^{(i)}$ internal vertices, respectively.
	Set $\ell^{(i)} = \ell_1^{(i)} + \ell_2^{(i)}$.
	Then, 
		\[c(G^{(i)}) \leq c(G^{(i)}) + 2 + \frac{3}{4}\ell^{(i)}
			 - \frac{1}{8}(\ell^{(i)})^2 
			 - \frac{1}{4}\ell_1^{(i)}\ell_2^{(i)}
			 - \frac{1}{4}\ell^{(i)}n(H^{(i-1)}.\]
\end{claim}
	
	By Lemma \ref{lma:twoeartype}, 
		the first ear spans endpoints $x_1, x_2$ in a maximal barrier $X \in \cP(H^{(i)})$
		and the second ear spans endpoints $y_1, y_2$ in a different maximal barrier $Y \in \cP(H^{(i)})$.
	Note that after these augmentations, $x_1$ and $x_2$ are not in the same barrier,
		and neither are $y_1$ and $y_2$, by Lemma \ref{lma:evolutionofbarriersEVEN}.
	
	The graph $G^{(i)}$ is an elementary supergraph of $H^{(i)}$ 
		given by adding cliques of free edges 
		corresponding to a maximal cover set $\cI$ in $\cC(H^{(i)})$.
	By Lemma \ref{lma:pushdownbarrier}, 
		each barrier $B \in \cI$ generates the barrier $B \cap V(H^{(i-1)})$ in $V(H^{(i-1)})$.
	This induces a cover set $\cI' = \{ B \cap V(H^{(i-1)}) : B \in \cI\}$
		in $\cC(H^{(i-1)})$ which in turn defines
		an elementary supergraph $G^{(i-1)}_*$
		through the bijection in Claim \ref{clm:coverbijection}.
	By the choice of $G^{(i-1)}$, 
		$e(G_*^{(i-1)}) \leq e(G^{(i-1)})$.
		
	Consider the number of free edges in $G^{(i)}$
		compared to the free edges in $G_*^{(i-1)}$.
	First, the number of edges between the $\ell_1^{(i)}+\ell_2^{(i)}$ new vertices 
		and the $n(H^{(i-1)})$ original vertices is 
		at most $\left(\frac{\ell_1^{(i)}}{2}+\frac{\ell_2^{(i)}}{2}\right)\frac{n(H^{(i-1)})}{2}$,
		since the additions must occur within barriers,
		at most half of the internal vertices of each ear can be used 
		(by Corollary \ref{cor:maxbarrierinear}),
		and barriers in $H^{(i-1)}$ have at most $\frac{n(H^{(i-1)})}{2}$ vertices.
	Second, consider the number of free edges within the $\ell_1^{(i)}+\ell_2^{(i)}$ vertices.
	Note that no free edges can be added between $\eps_1$ and $\eps_2$
		since the internal vertices of $\eps_1$ and $\eps_2$ are in different maximal barriers of $H^{(i)}$.
	Thus, there are at most ${\ell_1^{(i)}/2\choose 2}+{\ell_2^{(i)}/2\choose 2}$
		free edges between the internal vertices.
	Since ${\ell_1^{(i)}/2\choose 2}+{\ell_2^{(i)}/2\choose 2} %
		=\frac{1}{8}(\ell_1^{(i)}+\ell_2^{(i)})^2 %
		-\frac{1}{4}(\ell_1^{(i)}+\ell_2^{(i)}+\ell_1^{(i)}\ell_2^{(i)})$, 
		we have	
		\begin{align*}
			c(G^{(i)}) 
				&= e(G^{(i)}) - \frac{(n_{i-1})
					+\ell_1^{(i)}
					+\ell_2^{(i)})^2}{4}\\
				&\leq e(G_*^{(i-1)})
						+\left(1+\ell_1^{(i)}+1+\ell_2^{(i)}\right)
						+\frac{n(H^{(i-1)})(\ell_1^{(i)}+\ell_2^{(i)})}{4}\\
				&\qquad\qquad+\frac{1}{8}\left(\ell_1^{(i)}
							+\ell_2^{(i)}\right)^2
						-\frac{1}{4}\left(\ell_1^{(i)}
						+\ell_2^{(i)}+\ell_1^{(i)}\ell_2^{(i)}
						\right)\\
				&\qquad\qquad	- \left[\frac{n(H^{(i-1)})^2}{4}
						+\frac{n(H^{(i-1)})(\ell_1^{(i)}+\ell_2^{(i)})}{2} 
						+ \frac{(\ell_1^{(i)}+\ell_2^{(i)})^2}{4}\right]\\
				&\leq e(G^{(i-1)})
					- \frac{n(H^{(i-1)})^2}{4} 
					+ \left(2+\ell_1^{(i)}+\ell_2^{(i)}\right) \\
				&\qquad\qquad
					- \frac{1}{4}\left(\ell_1^{(i)}+\ell_2^{(i)}\right)
					- \frac{1}{8}\left(\ell_1^{(i)}
					+ \ell_2^{(i)}\right)^2
					- \frac{1}{4}\ell_1^{(i)}\ell_2^{(i)}
					- \frac{1}{4}n(H^{(i-1)})\ell^{(i)}\\
				&= c(G^{(i-1)}) + 2
				 	+ \frac{3}{4}\left(\ell^{(i)}\right) 
					- \frac{(\ell^{(i)})^2}{8} 
					- \frac{1}{4}\ell_1^{(i)}\ell_2^{(i)}
					- \frac{1}{4}n(H^{(i-1)})\ell^{(i)}.
		\end{align*}
	We have now proven Claim \ref{clm:doubleear}.
	We now combine a sequence of these bounds to show the global bound.	
	
	Since each ear augmentation forces $\Phi(H^{(i)}) > \Phi(H^{(i-1)})$, there
		are at most $p-q$ augmentations.
	Moreover, the increase in $c(G^{(i)})$ at each step 
		is bounded by $1 + \frac{3}{4}\ell^{(i)} - \frac{1}{8}(\ell^{(i)})^2 - \frac{1}{4}\ell^{(i)}n(H^{({i-1})})$ 
			in a single ear augmentation 
		and $2 + \frac{3}{4}\ell^{(i)}
			 - \frac{1}{8}(\ell^{(i)})^2 
			 - \frac{1}{4}\ell_1^{(i)}\ell_2^{(i)}
			 - \frac{1}{4}\ell^{(i)}n(H^{(i-1)})$ in a double ear augmentation.
	Independent of the number of ears,
		\[ c(G^{(i)}) - c(G^{(i-1)}) \leq 2 + \ell^{(i)} - \frac{1}{8}(\ell^{(i)})^2 - \frac{1}{4}\ell^{(i)}n(H^{(i-1)}).\]
	Note also that if $\ell^{(i)}$ is positive, then it is at least two.
	Combining those inequalities gives
		\begin{align*}
			c(G') &\leq c(G) + \sum_{i=1}^k \left(
					 2
					 + \frac{3}{4}\ell^{(i)}
					 - \frac{1}{8}(\ell^{(i)})^2
					 - \frac{1}{4}\ell^{(i)}n(H^{(i-1)})
					\right)\\
				&\leq c(G) 
					+ \sum_{i=1}^k 2 
					+ \frac{3}{4}\sum_{i=1}^k \ell^{(i)}
					- \frac{1}{8}\sum_{i=1}^k (\ell^{(i)})^2
					- \frac{1}{4}\sum_{i=1}^k \ell^{(i)}n(H^{(i-1)})\\
				&\leq c(G) 
					+ 2k 
					+ \frac{3}{4}(N-n) 
					- \frac{1}{8}\sum_{i=1}^k 2\ell^{(i)}
					- \frac{1}{4}\sum_{i=1}^k \ell^{(i)}n\\
				&\leq c(G) + 2(p-q) - \frac{1}{4}(N-n)(n-2).
		\end{align*}
	This proves the result.
\end{proof}	

\begin{corollary}\label{cor:pruningcondition}
	Let $p, c \geq 1$ be integers.
	If $H$ is a 1-extendable graph with 
		$q = \Phi(H)$,
		$c'$
		is the maximum excess $c(G)$ over all graphs $G \in \cE(H)$,
		and $c' + 2(p-q) < c$, 
		then there is no 1-extendable graph $H \subset H'$ 
		reachable from $H$ by a graded ear decomposition
		so that $\Phi(H') = p$
		and there exits a graph $G' \in \cE(H')$ with excess $c(G') \geq c$.
\end{corollary}

Corollary \ref{cor:pruningcondition} gives 
	the condition  
	to test if we can prune the current node, 
	since there does not exist a sequence of ear augmentations
	that lead to a graph with excess at least our known lower bound on $c_p$.
Moreover, Lemma \ref{lma:pruning} 
	provides a dynamic bound on the number $N$  of vertices 
	that can be added to the current
	graph while maintaining the possibility of finding a graph with excess at least the
	known lower bound on $c_p$,
	by selecting $N$ to be maximum so that $c' + 2(p-\Phi(H)) - \frac{1}{4}(N-n)(n-2) \geq c$.

\section{Results and Data}
\label{ssec:results}

The full algorithm to search for $p$-extremal
	elementary graphs
	combines three types of algorithms.
First, the canonical deletion from Section \ref{ssec:pmdeletion}
	is used to enumerate the search space 
	with no duplication of isomorphism classes.
Second, the pruning procedure from Section \ref{ssec:pruning}
	greatly reduces the number of generated graphs 
	by backtracking when no dense graphs are reachable.
Third, Section \ref{ssec:barriers}	
	provided a method for adding free edges to 
	a 1-extendable graph with $p$ perfect matchings 
	to find maximal elementary supergraphs.
	
The recursive generation algorithm Search$(H^{(i)},N,p,c)$ 
	is given in Algorithm \ref{alg:recursivesearch}.
Given a previously computed lower bound $c \leq c_p$,
	the full search Generate$(p,c)$ (Algorithm \ref{alg:fullsearch})
	operates
	by running Search$(C_{2k}, N_p, p, c)$
	for each even cycle $C_{2k}$ with $4 \leq 2k \leq N_p$.
All elementary graphs $G$
	with $\Phi(G) = p$ and $c(G) \geq c$ 
	are generated by this process.

\begin{algorithm}[p]
	\caption{\label{alg:recursivesearch}Search$(H^{(i)},N^{(i)},p,c)$}
	\begin{algorithmic}
		\STATE \emph{Check all pairs of vertices, up to symmetries}
		\FOR{all vertex-pair orbits $\mathcal{O}$ in $H^{(i)}$}
			\STATE $\{x,y\} \leftarrow$ representative pair of $\mathcal{O}$
			\STATE \emph{Augment by ears of all even orders}
			\FOR{all orders $r \in \{0,2,\dots,N^{(i)}-n(H^{(i)})\}$}
				\STATE $H^{(i+1)} \leftarrow H^{(i)} + \operatorname{Ear}(x,y,r)$
				\IF{$H^{(i)}$ is almost 1-extendable and $H^{(i+1)}$ is not 1-extendable}
					\STATE{\emph{Skip $H^{(i+1)}$ (decomposition is not graded).}}
				\ELSIF{$\Phi(H^{(i+1)}) > p$}
					\STATE \emph{Skip $H^{(i+1)}$.}
				\ELSE
					\STATE{\emph{Check the canonical deletion}}
					\STATE $(x',y',r') \leftarrow \del(H^{(i+1)})$
					\IF{$r = r'$ {\bf and} $\{x',y'\} \in \mathcal{O}$}
						\STATE{\emph{This augmentation matches the canonical deletion}}
						\STATE $n^{(i+1)} \leftarrow n(H^{(i+1)})$.
						\STATE $p^{(i+1)} \leftarrow \Phi(H^{(i+1)})$.
						\STATE $c^{(i+1)} \leftarrow \max\{c\left(H^{(i+1)}_{\cI}\right) : \cI \in \cC(H^{(i+1)})\}$.
						\IF{$p^{(i+1)} = p$ and $c^{(i+1)} \geq c$}
							\STATE \emph{There are solutions within $\cE(H^{(i+1)})$.}
							\FOR{all cover sets $\cI \in \cC(H^{(i+1)})$}
								\IF{$c\left(H_{\cI}^{(i+1)}\right) \geq c$}
									\STATE Output $H_{\cI}$.
								\ENDIF
							\ENDFOR
						\ELSIF{$q < p$ {\bf and} $c^{(i+1)} + 2(p-p^{(i+1)}) \geq c$}
							\STATE \emph{Use Lemma \ref{lma:pruning} to bound the number of vertices for future augmentations.}
						 	\STATE $N^{(i+1)} = \max\{ N' : c^{(i+1)} + 2(p-p^{(i+1)}) - \frac{1}{4}(N'-n^{(i+1)}))(n^{(i+1)}-2) \geq c\}$.
							\STATE Search$(H^{(i+1)}, N^{(i+1)}, p, c)$.
						\ENDIF
					\ENDIF
				\ENDIF
			\ENDFOR
		\ENDFOR
		\RETURN
	\end{algorithmic}
\end{algorithm}

\begin{algorithm}[p]
	\caption{\label{alg:fullsearch}Generate$(p,c)$}
	\begin{algorithmic}
		\STATE $N \leftarrow \max\{2r : 2r \leq 3+\sqrt{16p-8c-23}\}$.
		\FOR{$r \in \{1,\dots,N/2\}$}
			\STATE Search$(C_{2r},N,p,c)$
		\ENDFOR
		\RETURN
	\end{algorithmic}
\end{algorithm}

\begin{theorem}\label{thm:correctness}
	Given $p$ and $c \leq c_p$, 
		Generate($p,c$) (Algorithm \ref{alg:fullsearch})
		outputs all unlabeled elementary graphs with $p$
		perfect matchings and excess at least $c$.
\end{theorem}

\begin{proof}
	Given an unlabeled graph $G$ with $\Phi(G) = p$ and $c(G) \geq c$,
		note that Theorem \ref{thm:sizebound}
		implies $n(G) \leq 3 + \sqrt{16p - 8c - 23}$.
	With respect to the canonical deletion $\del(H)$,
		let 
		$ H^{(0)} \subset H^{(1)} \subset \cdots \subset H^{(k)}$
		be the canonical ear decomposition of the extendable subgraph $H$ in $G$.
	By the choice of canonical deletion, this decomposition takes the form
		of Corollary \ref{cor:twoears}.
	Moreover, $H^{(0)}$ is an even cycle $C_{2r}$ for some $r$.
	The Generate$(p,c)$ method initializes Search$(C_{2r}, N, p, c)$.
	
	By the definition of canonical ear decomposition, 
		the canonical ear $\eps^{(i)} = \del(H^{(i)})$ of $H^{(i)}$
		is the ear used to augment from $H^{(i-1)}$ to $H^{(i)}$.
	Let $x^{(i)}, y^{(i)}$ be the endpoints of $\eps^{(i)}$.
	When Search($H^{(i)},N^{(i)},p,c$) is called,
		the pair orbit ${\mathcal O}$ containing $\{x^{(i+1)}, y^{(i+1)}\}$
		is visited and an ear $\eps$ of the same order as $\eps^{(i+1)}$
		is augmented to $H^{(i)}$ to form 
		a graph $H_*^{(i+1)}$.
	Note that $H_*^{(i+1)} \cong H^{(i+1)}$
		with an isomorphism mapping $\eps$ to $\eps^{(i+1)}$.
	By the definition of the canonical deletion $\del(H)$,
		the algorithm accepts this augmentation.
		
	For each $i$, let $G^{(i)}$ be a maximum-size elementary supergraph in $\cE(H^{(i)})$.
	By Theorem \ref{thm:barrierconflictgraph}, there exists
		a maximal cover set $\cI \in \cC(H^{(i)})$
		so that $G^{(i)} = H^{(i)}_{\cI}$.
	Since $c(G^{(k)}) = c(G) \geq c$, 
		Lemma \ref{lma:pruning}
		gives $c(G) \leq c(G^{(i+1)}) + 2(p-p^{(i+1)}) - \frac{1}{4}(n(G) - n(H^{(i+1)}))(n(H^{(i+1)})-2)$,
		so the algorithm recurses with $N^{(i+1)} \geq n(G)$.
		
	When $H^{(k)}$ is reached, the algorithm notices that $\Phi(H^{(k)}) = p$
		and enumerates all cover sets $\cI\in \cC(H^{(k)})$ 
		which generates the elementary supergraphs $H^{(k)}_{\cI} \in \cE(H^{(k)})$ 
		with excess at least $c$.
	Since $H^{(k)}$ is the extendable subgraph of $G$
		and $c(G) \geq c$, this procedure will output $G$.
\end{proof}

The framework for this search
	was implemented within the EarSearch library\footnote{The EarSearch library is available at \href{https://github.com/derrickstolee/TreeSearch}{https://github.com/derrickstolee/EarSearch}}.
This implementation was executed
	on the Open Science Grid~\cite{OpenScienceGrid}
	using the University of Nebraska Campus Grid~\cite{WeitzelMS}.
The nodes available on the University of Nebraska Campus Grid
	consist of Xeon and Opteron processors with 
	a speed range of between 2.0 and 2.8 GHz.

Combining this algorithm with known lower bounds on $c_p$
	for $p \in \{11,\dots, 27\}$
	provided a full enumeration of $p$-extremal graphs
	for this range of $p$.
The resulting values of $c_p$ and $n_p$ are
	given in Table \ref{tab:NewValuesCP}.
The computation time for these values ranged from less than a minute to more than 
	10 years.
Table \ref{tab:pmtiming} gives the full list of computation times.
The resulting $p$-extremal elementary graphs for $11 \leq p \leq 27$ 
	are given 
	in Figure \ref{fig:ExtremalElemGraphs}.

\begin{table}[t]
	\centering	
	\begin{tabular}[h]{|r|c|c|c|c|c|c|c|c|c|c|c|c|c|c|c|c|c|}
	\hline
		$p$ & 11 & 12 & 13 & 14 & 15 & 16 & 17 & 18 & 19 & 20 & 21 & 22 & 23 & 24 & 25 & 26 & 27 \\
		\hline
		$n_p$ & 8 & 6 & 8 & 8 & 6 & 8 & 8 & 8 & 8 & 8 & 8 & 8 & 8 & 8 & 8 & 8 &8 \\
		$c_p$ & {3} & 5 & {3} & {4} & 6 & {4} & {4} & {5} & {4} & {5} & {5} & {5} & 5 & 6 & 5 & 5 & 6  \\
		\hline
		$C_p$ &  4& 5& 5& 5 & 6& 6& 6& 6& 6& 6 & 6 & 6 & 6 & 6 & 6& 6 & 6\\
		\hline
 	\end{tabular}
	\caption{\label{tab:NewValuesCP}New values of $n_p$ and $c_p$.
		Conjecture \ref{conj:starremoval} states that $c_p \leq C_p$.}
\end{table}

\begin{table}[t]
	\centering
	\begin{tabular}{c|c|c|r@{ }r@{ }r@{.}l}
		$p$ & $N_p$ & $c_p$
				& \multicolumn{4}{c}{Total CPU Time} \\
		\hline
		 $5$ &  8 & 2     &   &    &   0&02s \\
		 $6$ & 10 & 3     &   &    &   0&04s \\
		 $7$ & 10 & 3    &   &    &   0&18s \\
		 $8$ & 12 & 3     &   &    &   0&72s \\
		 $9$ & 12 & 4    &   &    &   1&46s \\
		$10$ & 12 & 4    &   &    &  5&95s \\
		$11$ & 14 & 3      &     &      & 43&29s \\
		$12$ & 14 & 5     &     &      & 44&01s \\
		$13$ & 14 & 3  &  &  6m & 39&80s \\
		$14$ & 16 & 4  &  & 12m & 10&40s \\
		$15$ & 16 & 6  &  & 12m & 42&72s \\
	 	\multicolumn{5}{c}{}&
	\end{tabular}	
	\qquad
	\begin{tabular}{c|c|c|r@{ }r@{ }r@{ }r@{ }r@{.}l}
		$p$ & $N_p$ & $c_p$
				& \multicolumn{4}{c}{Total CPU Time} \\
		\hline
		$16$ & 16 & 4 && &  2h & 07m & 58&60s \\
		$17$ & 16 & 4 && &  6h & 46m & 07&72s \\
		$18$ & 18 & 5 && & 11h & 45m & 01&95s \\
		$19$ & 18 & 4 & &   2d & 17h & 12m & 31&85s \\
		$20$ & 18 & 5 & &   4d & 05h & 28m & 11&79s \\
		$21$ & 18 & 5 & &  13d & 17h & 29m & 12&45s \\ 
		$22$ & 20 & 5 & &  42d & 20h & 40m & 30&41s \\
		$23$ & 20 & 5 & & 118d & 07h & 38m & 36&84s \\
		$24$ & 20 & 6 & & 209d & 10h & 09m & 54&98s \\
		$25$ & 20 & 5 &  2y & 187d & 21h & 48m & 46&31s \\
		$26$ & 20 & 5 &  7y &  75d & 13h & 55m & 10&27s \\
		$27$ & 22 & 6 & 10y & 247d & 21h &  3m & 13&94s \\
	\end{tabular}
	\caption{\label{tab:pmtiming}Time analysis of the search for varying $p$ values.}
\end{table}

\def\figtabwidth{0.1\textwidth}
\begin{figure}[t]
	\begin{tabular}[h]{cccccccc}
		 \includegraphics[width=\figtabwidth]{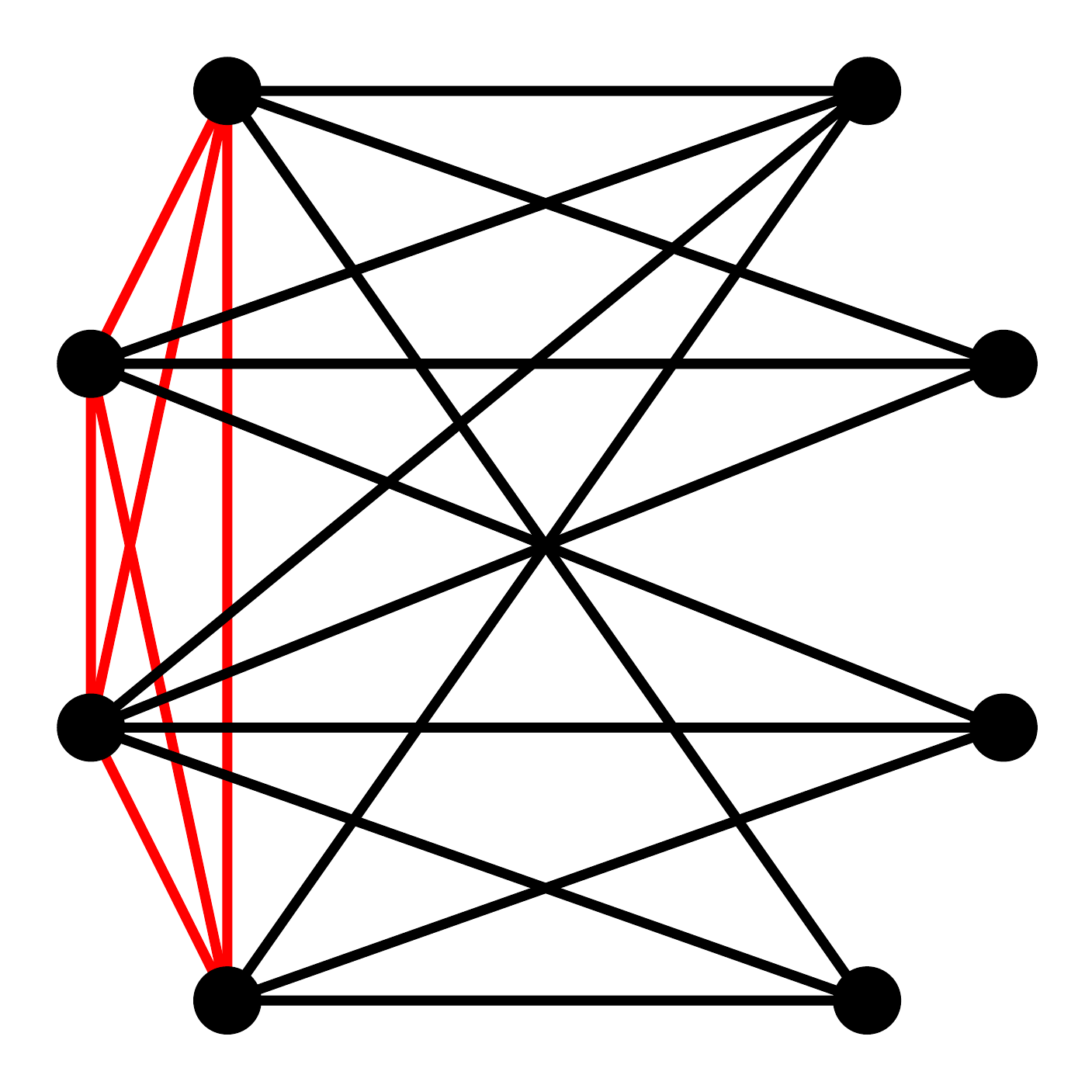}
			&  \includegraphics[width=\figtabwidth]{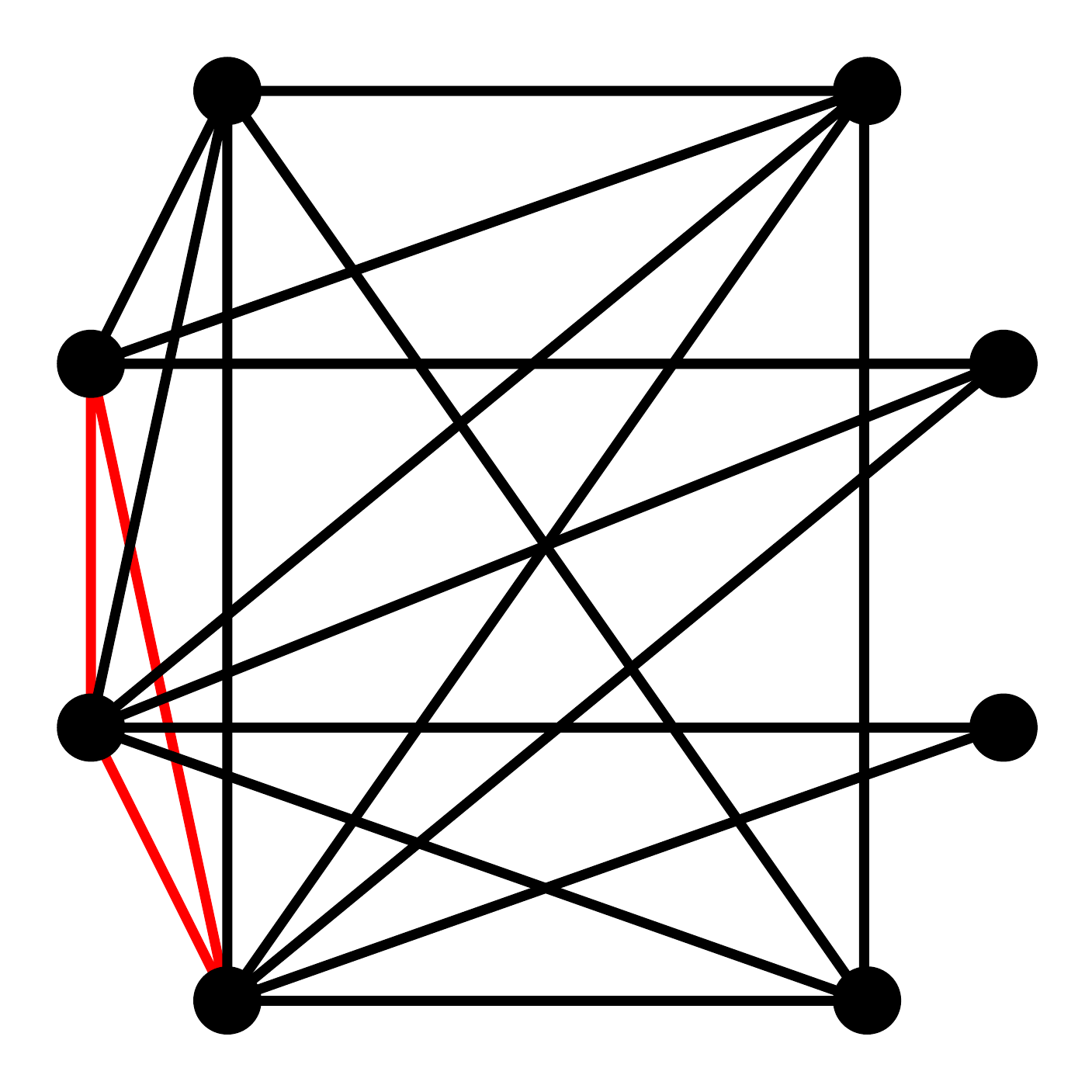}
		    &  \includegraphics[width=\figtabwidth]{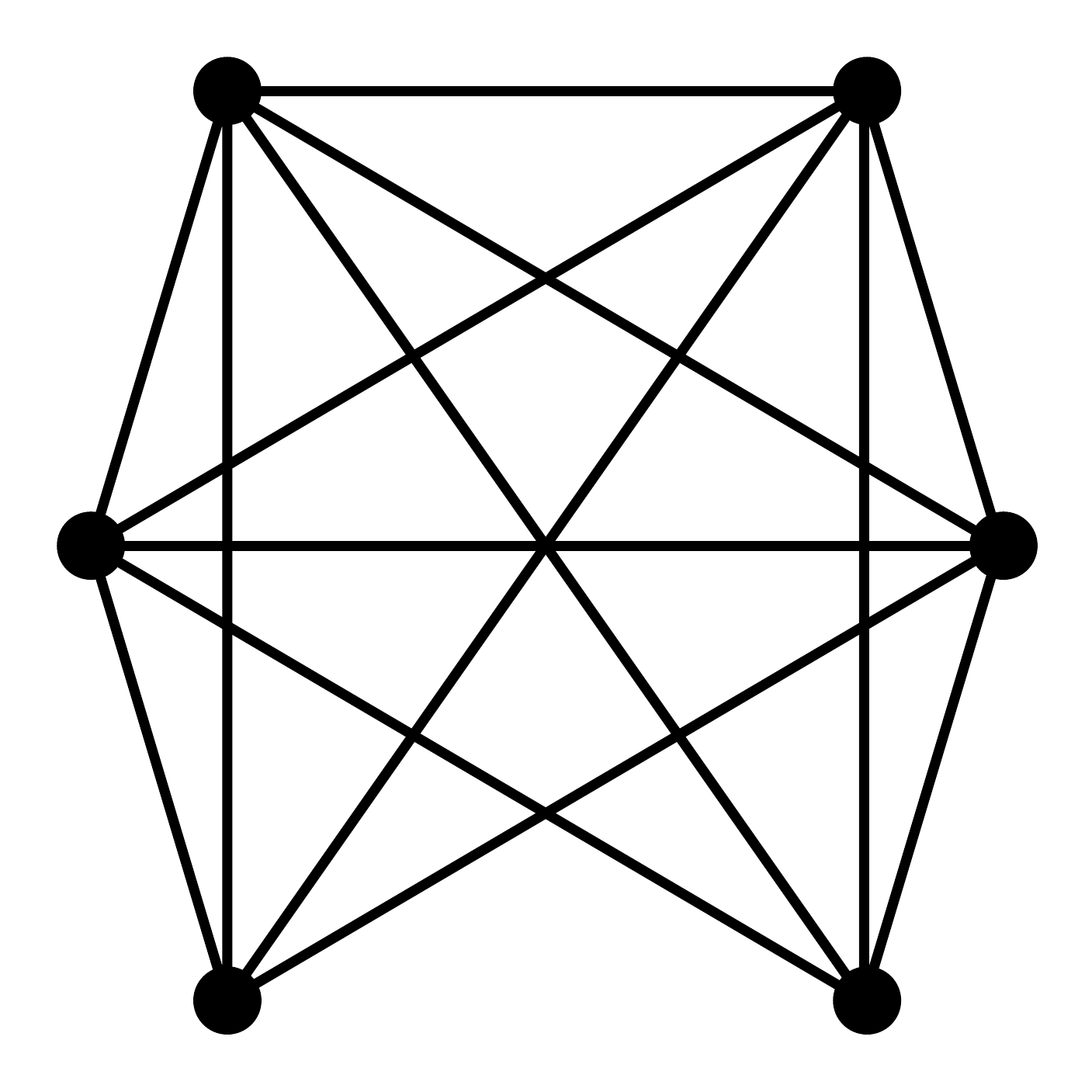}
		    & \includegraphics[width=\figtabwidth]{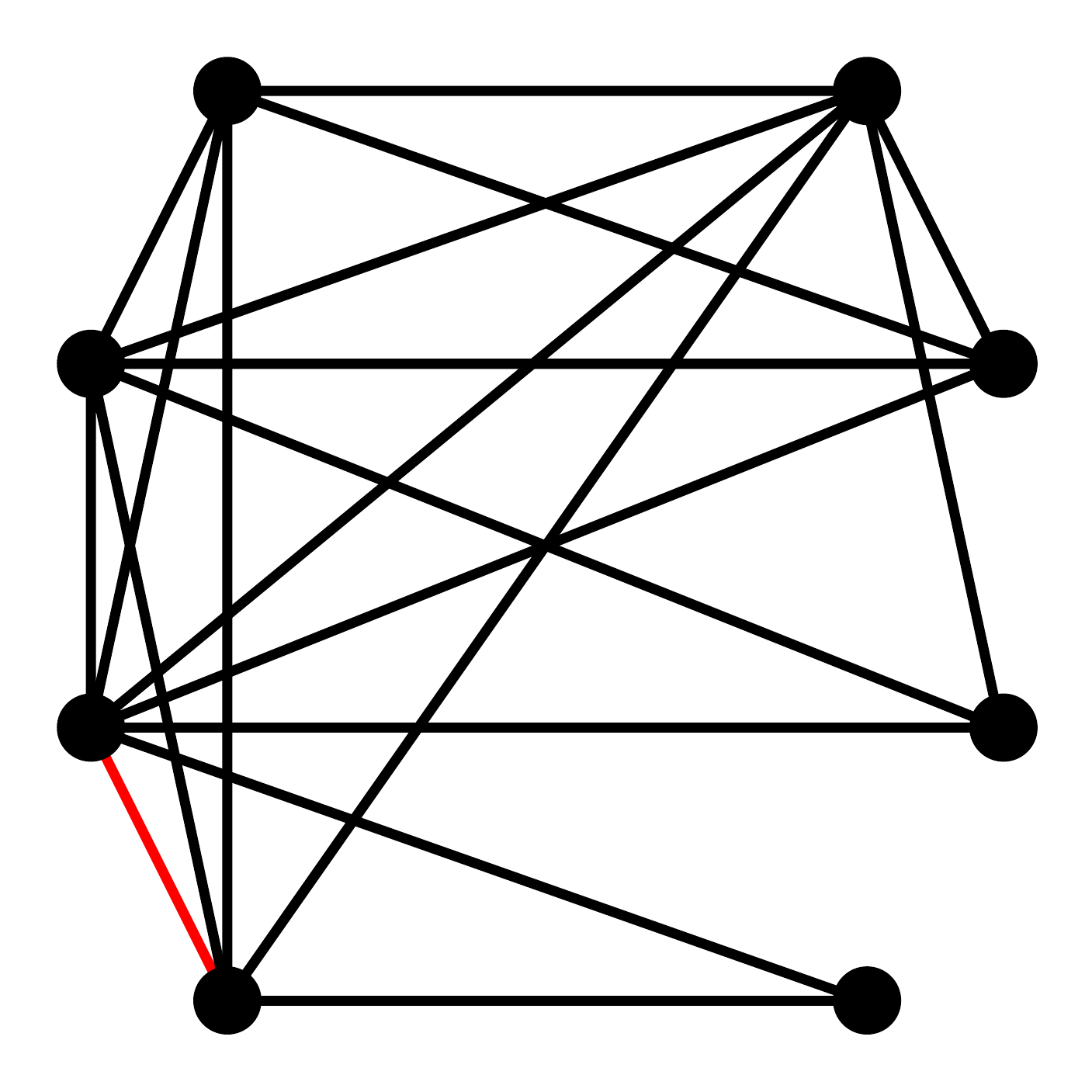}
		    & \includegraphics[width=\figtabwidth]{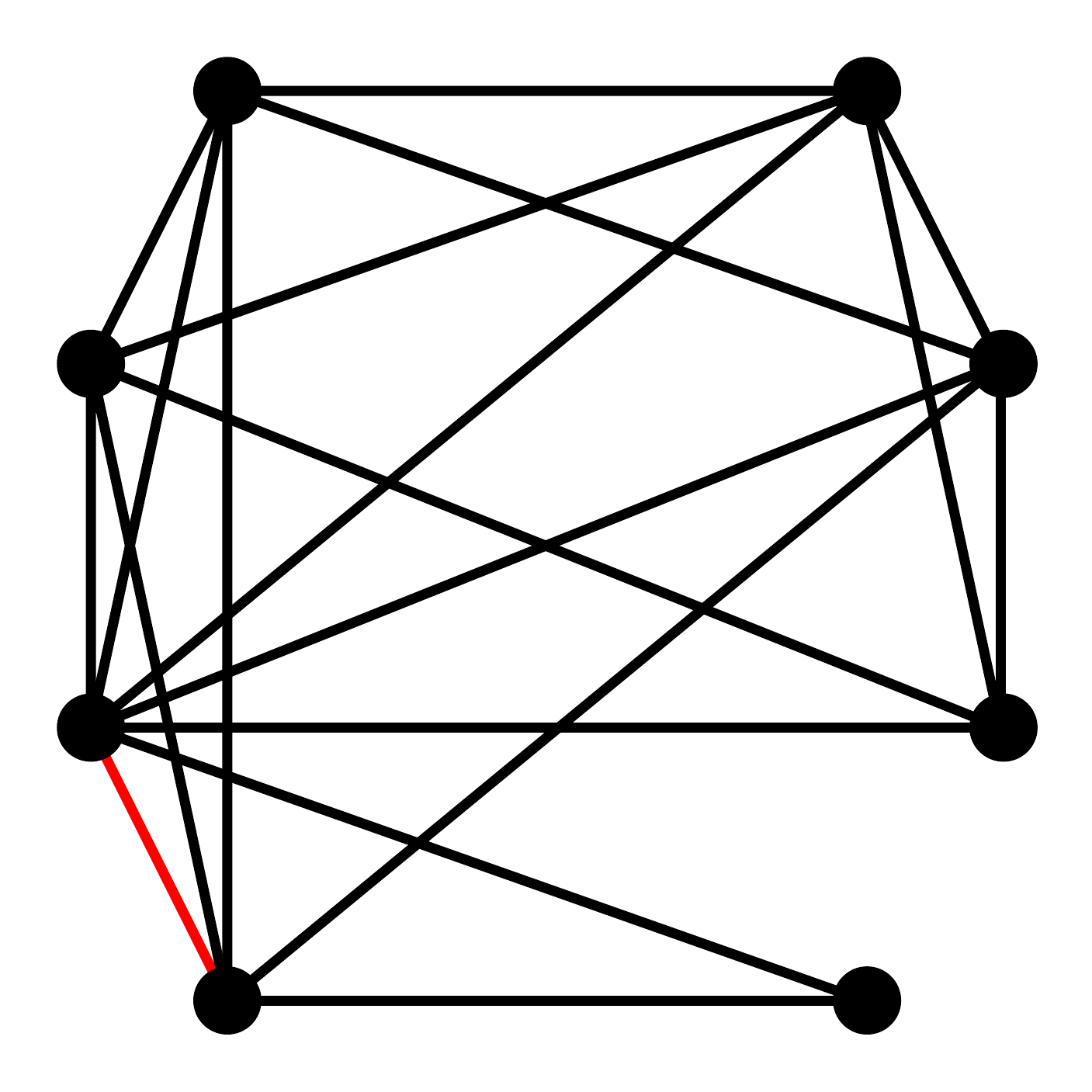}
		    & \includegraphics[width=\figtabwidth]{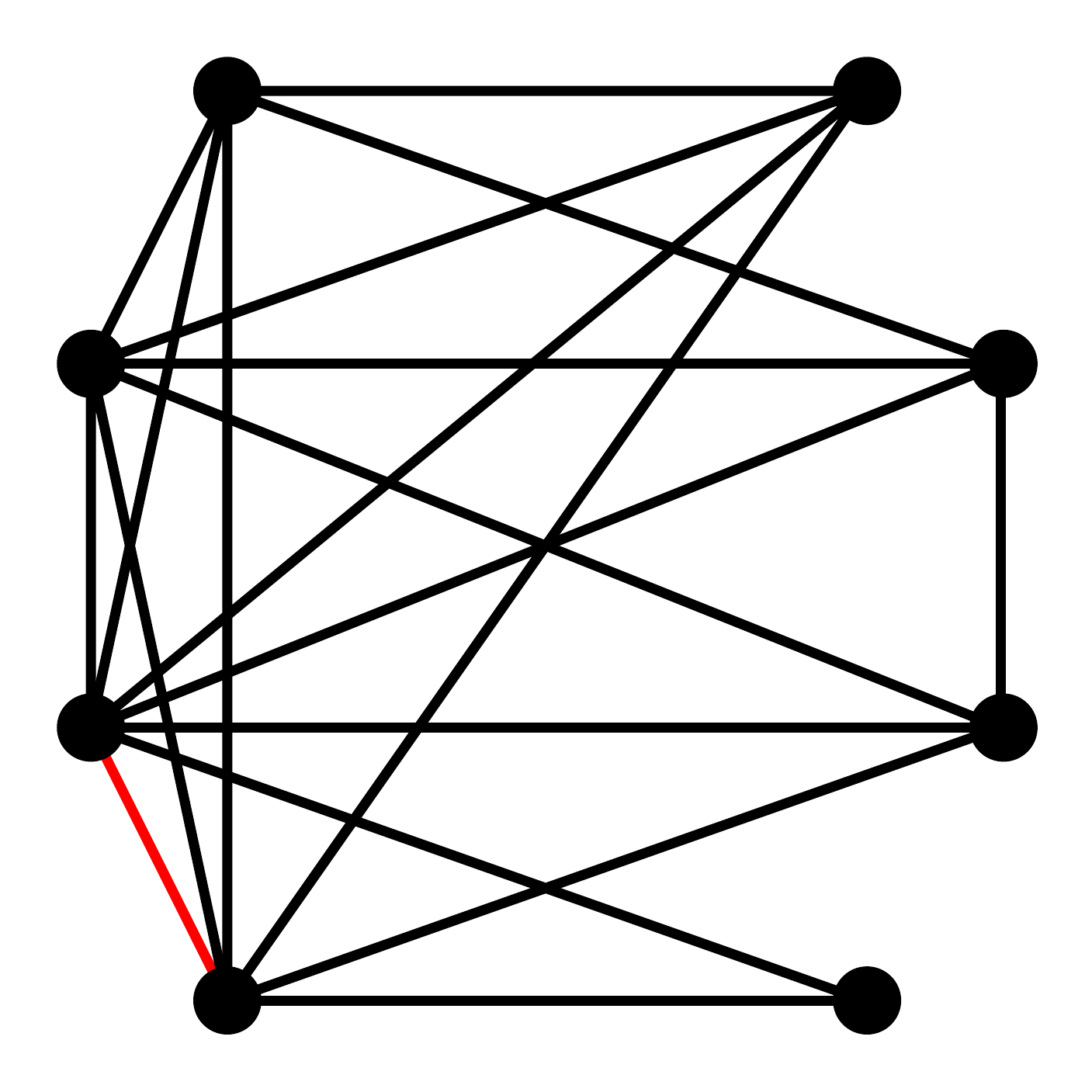}
		    & \includegraphics[width=\figtabwidth]{figs-delete/p13c3n8-3}
			& \includegraphics[width=\figtabwidth]{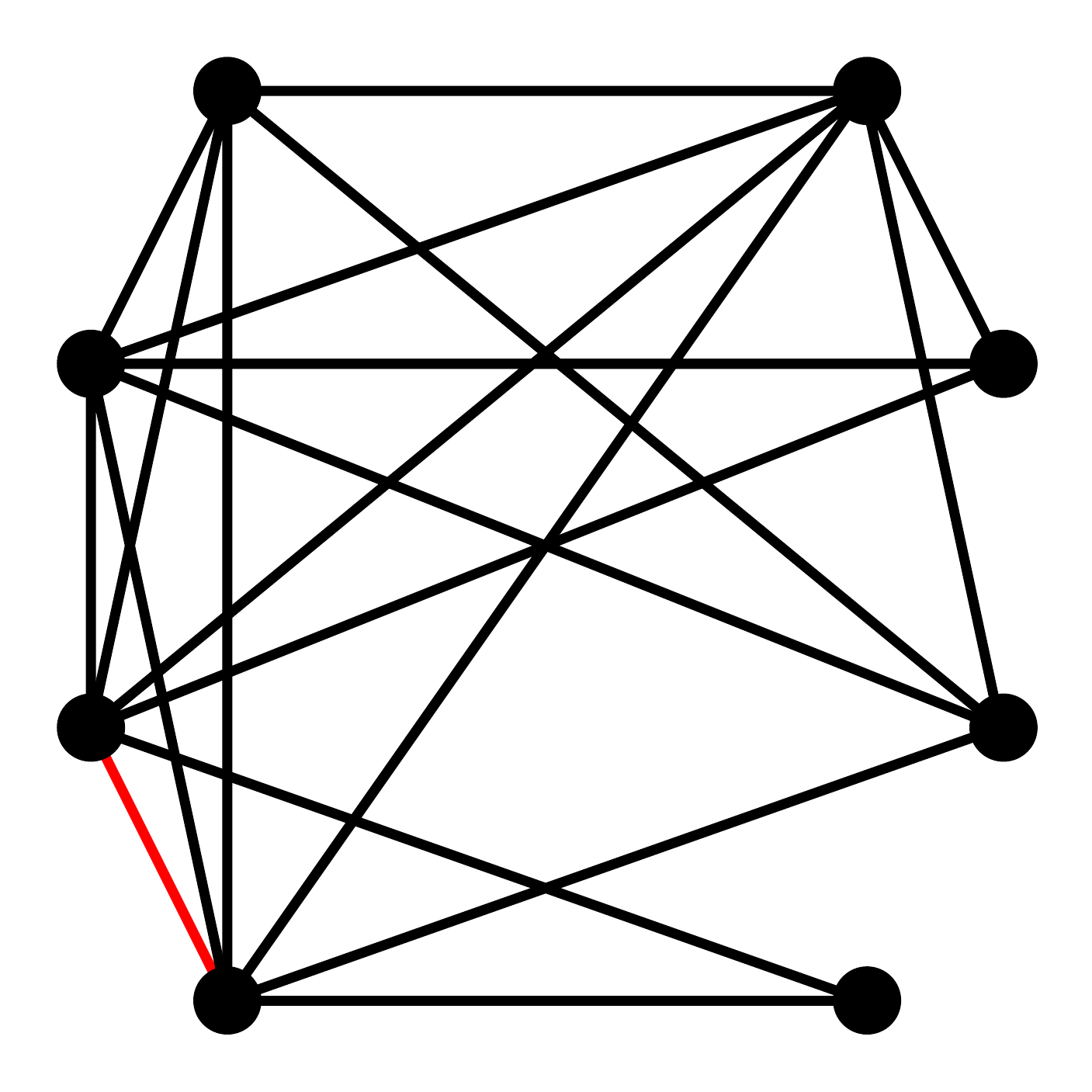}
			\\
		$p=11$
			& $p=11$
			& $p=12$
			& $p=13$
			& $p=13$
			& $p=13$
			& $p=13$
			& $p=13$
			\\
			 \includegraphics[width=\figtabwidth]{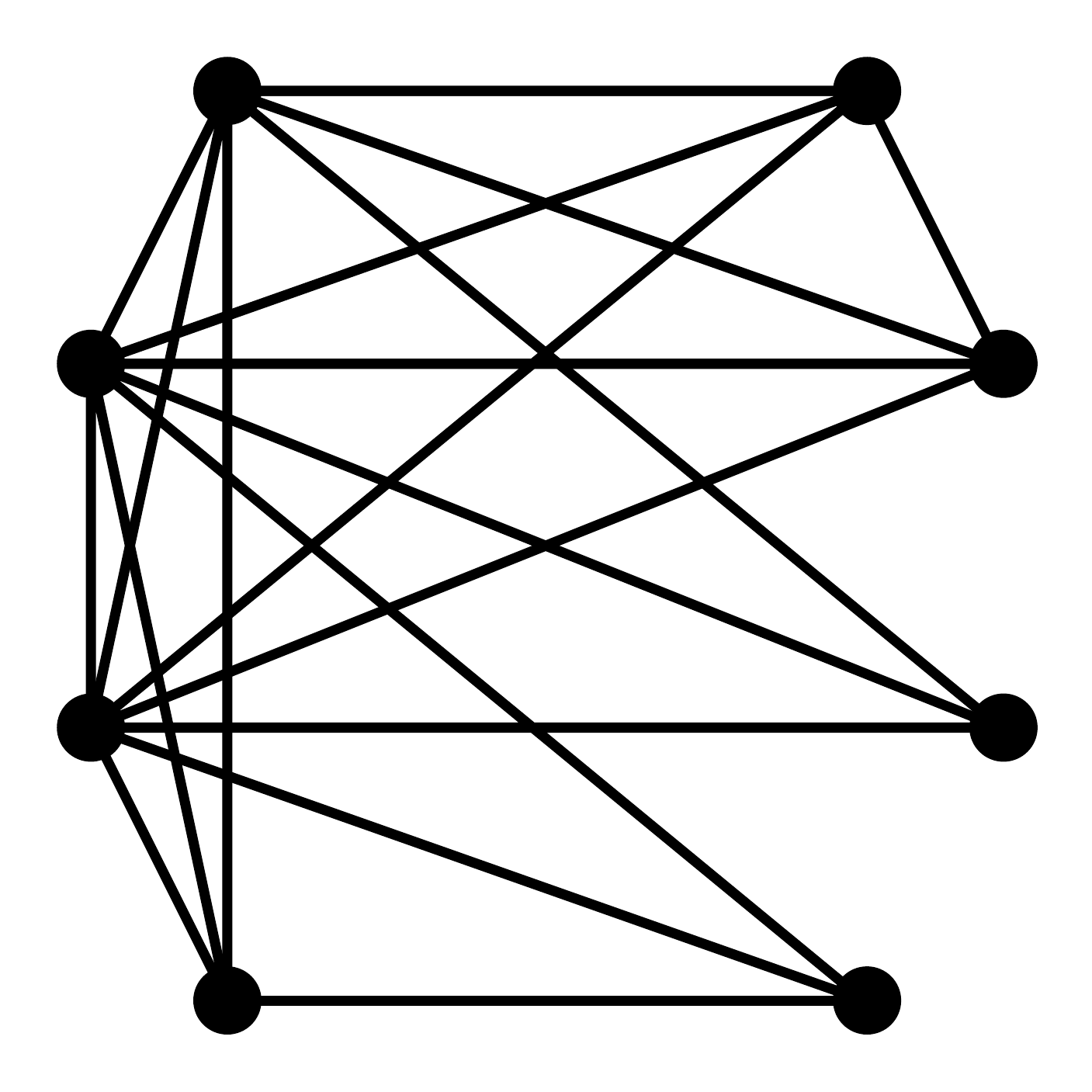}
			& \includegraphics[width=\figtabwidth]{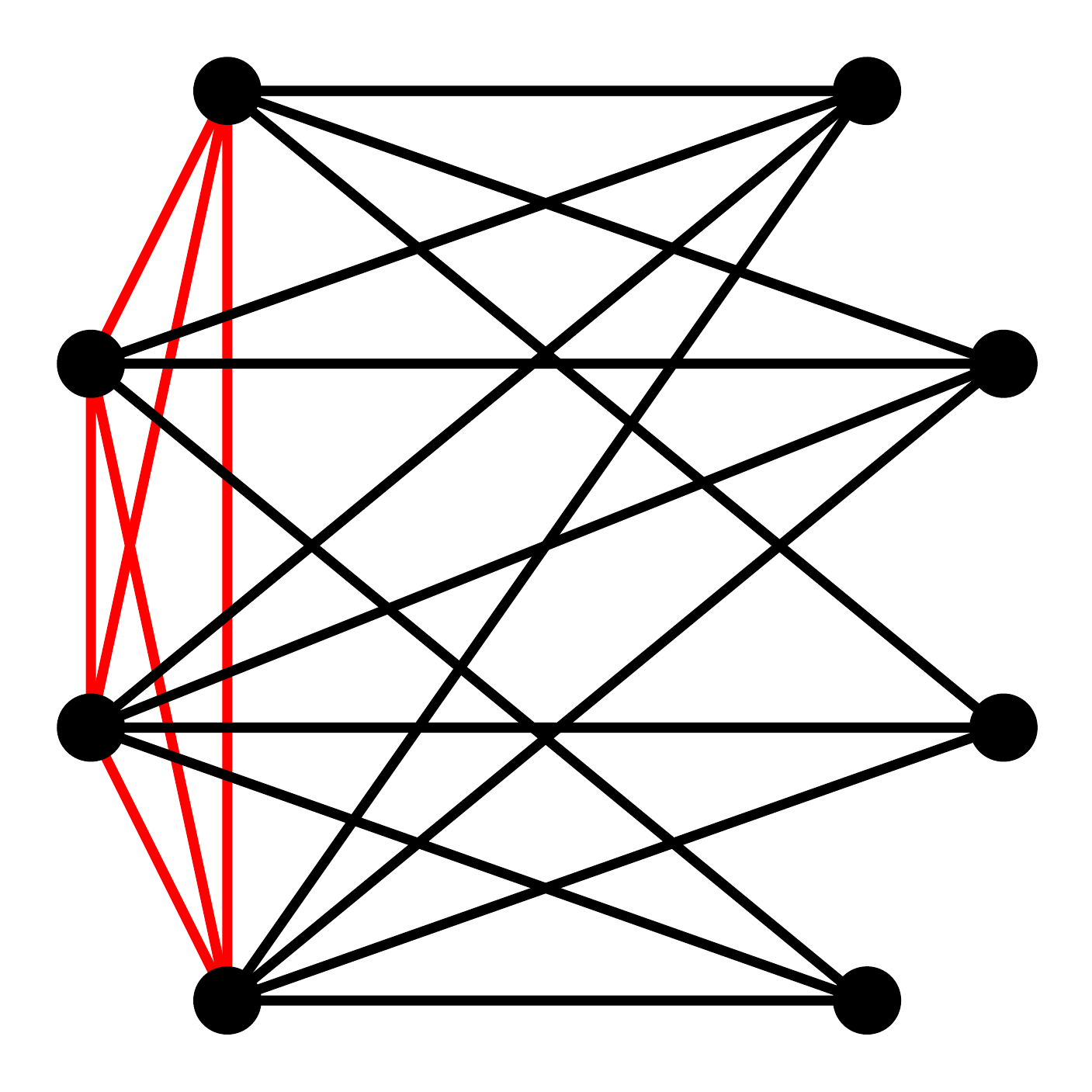}
			& \includegraphics[width=\figtabwidth]{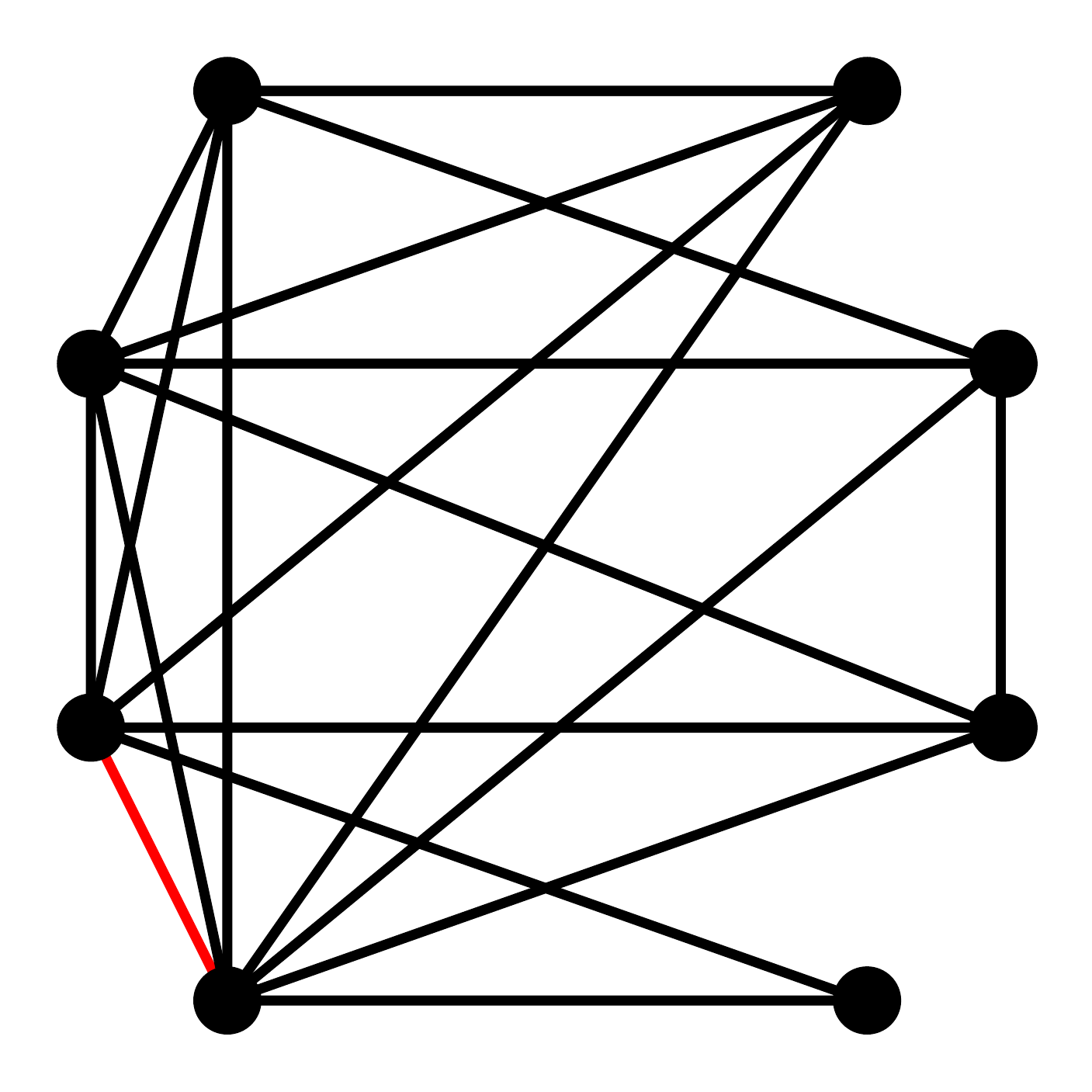}
			& \includegraphics[width=\figtabwidth]{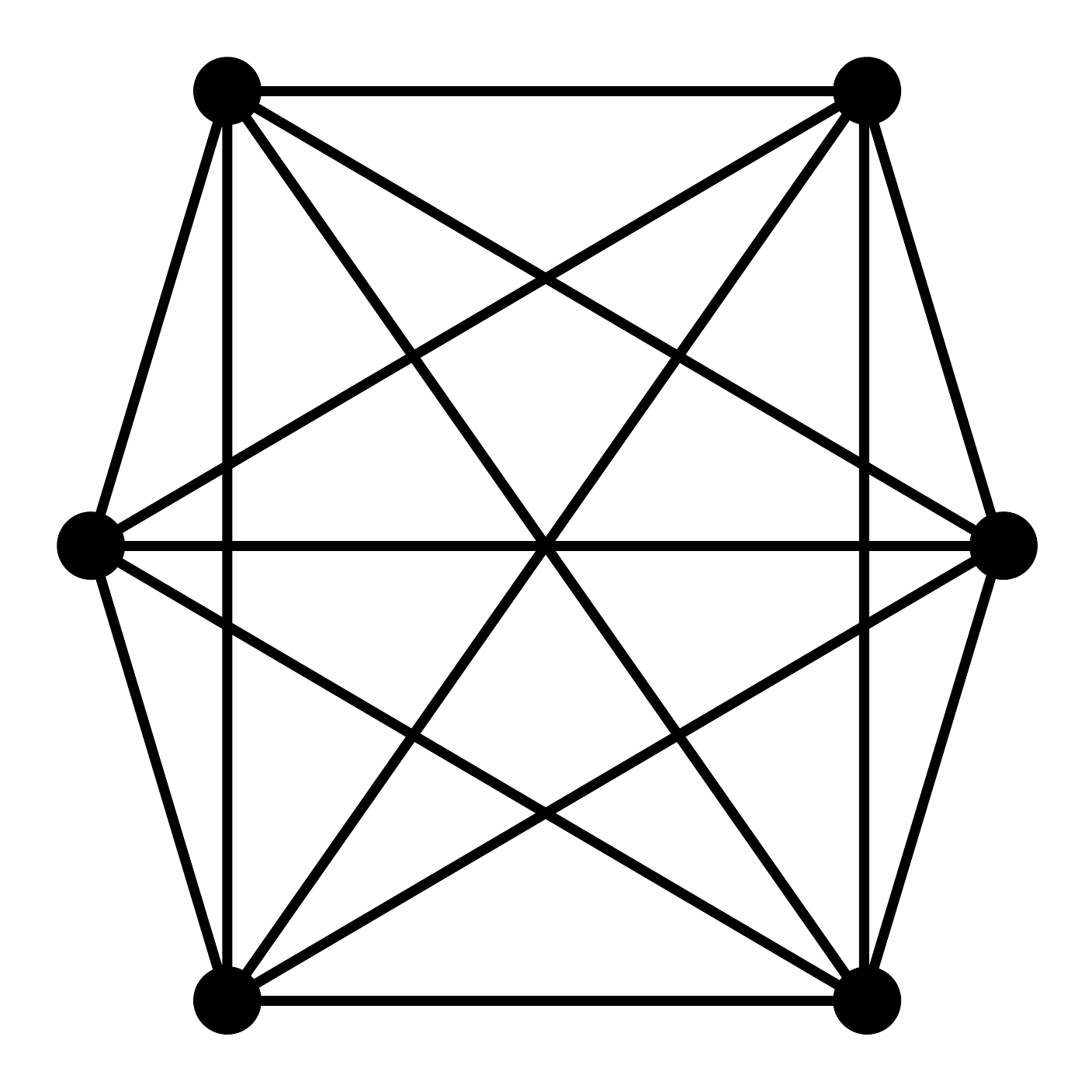}
			& \includegraphics[width= \figtabwidth]{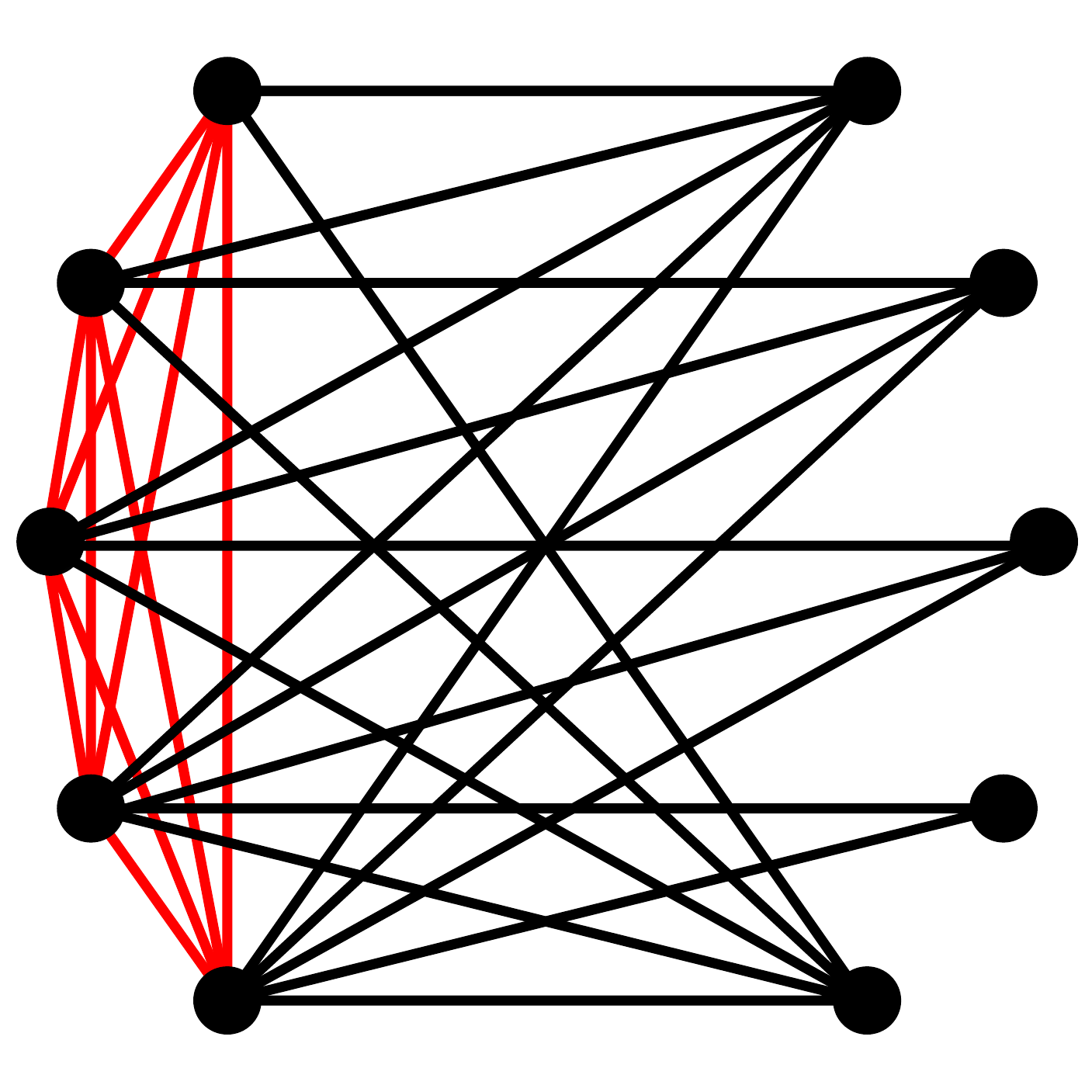}
			& \includegraphics[width= \figtabwidth]{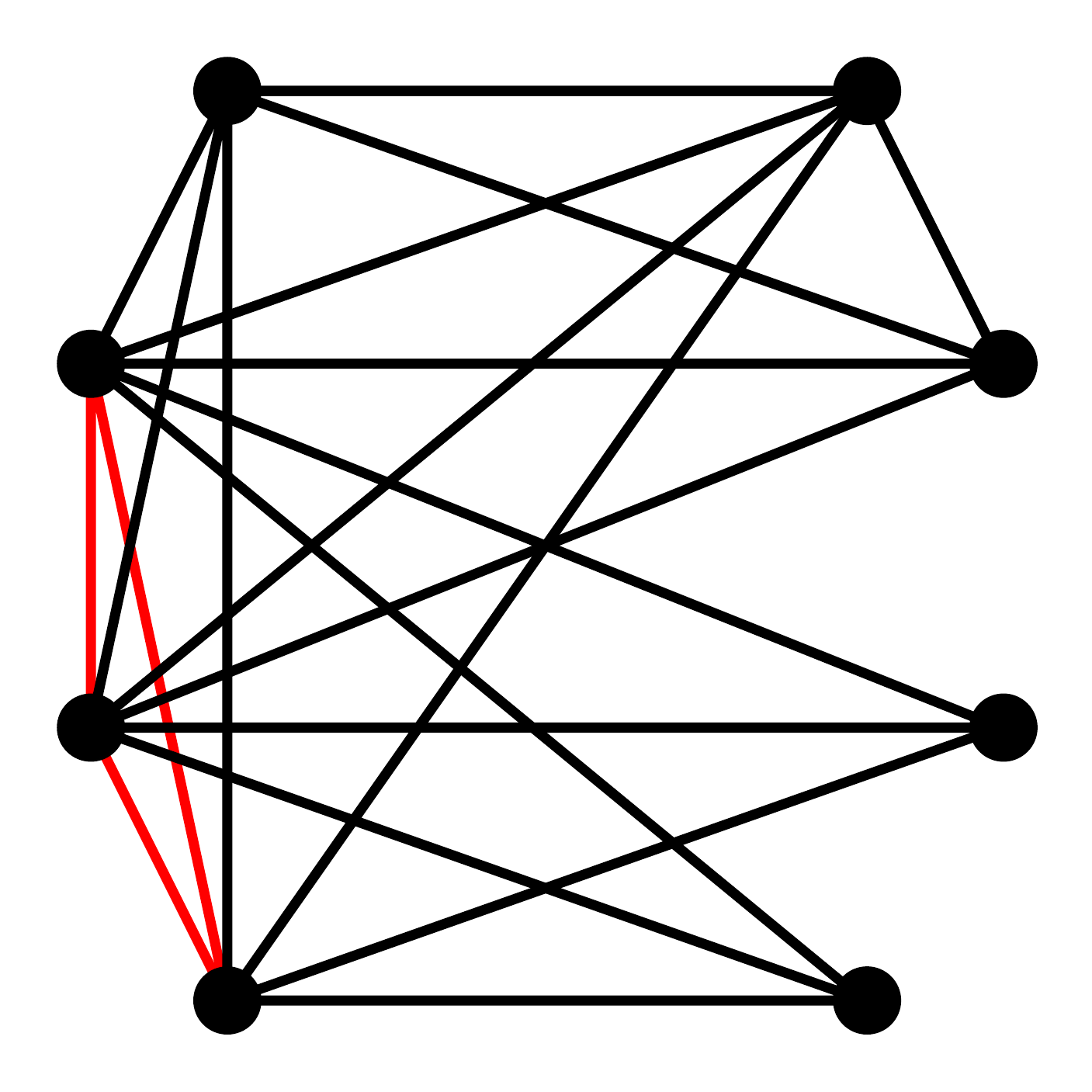}
			& \includegraphics[width= \figtabwidth]{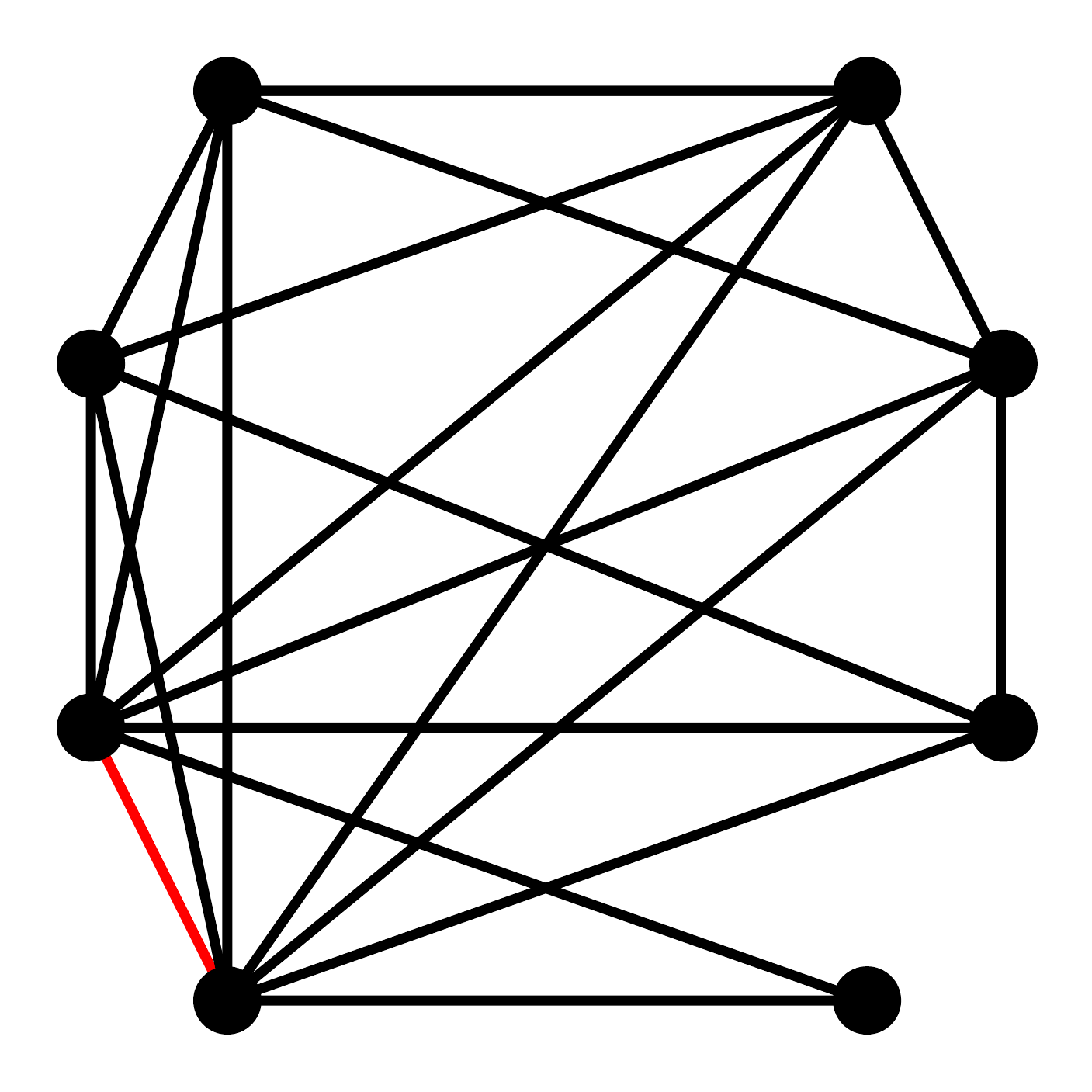}
			& \includegraphics[width= \figtabwidth]{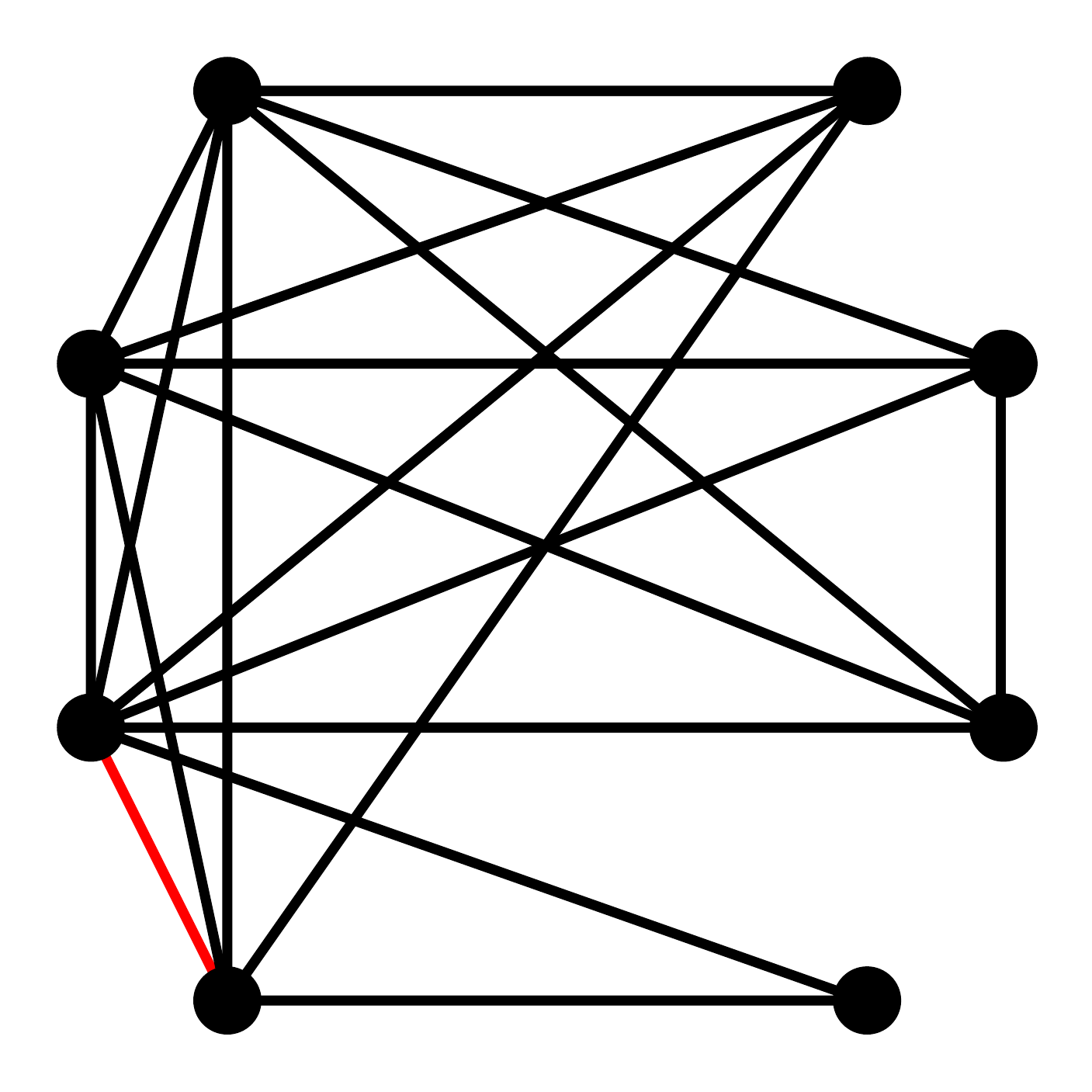}
			 \\
			$p=13$
				& $p=14$
				& $p=14$
				& $p=15$
				& $p=16$
				& $p=16$
				& $p=16$
				& $p =16$\\
				 \includegraphics[width=\figtabwidth]{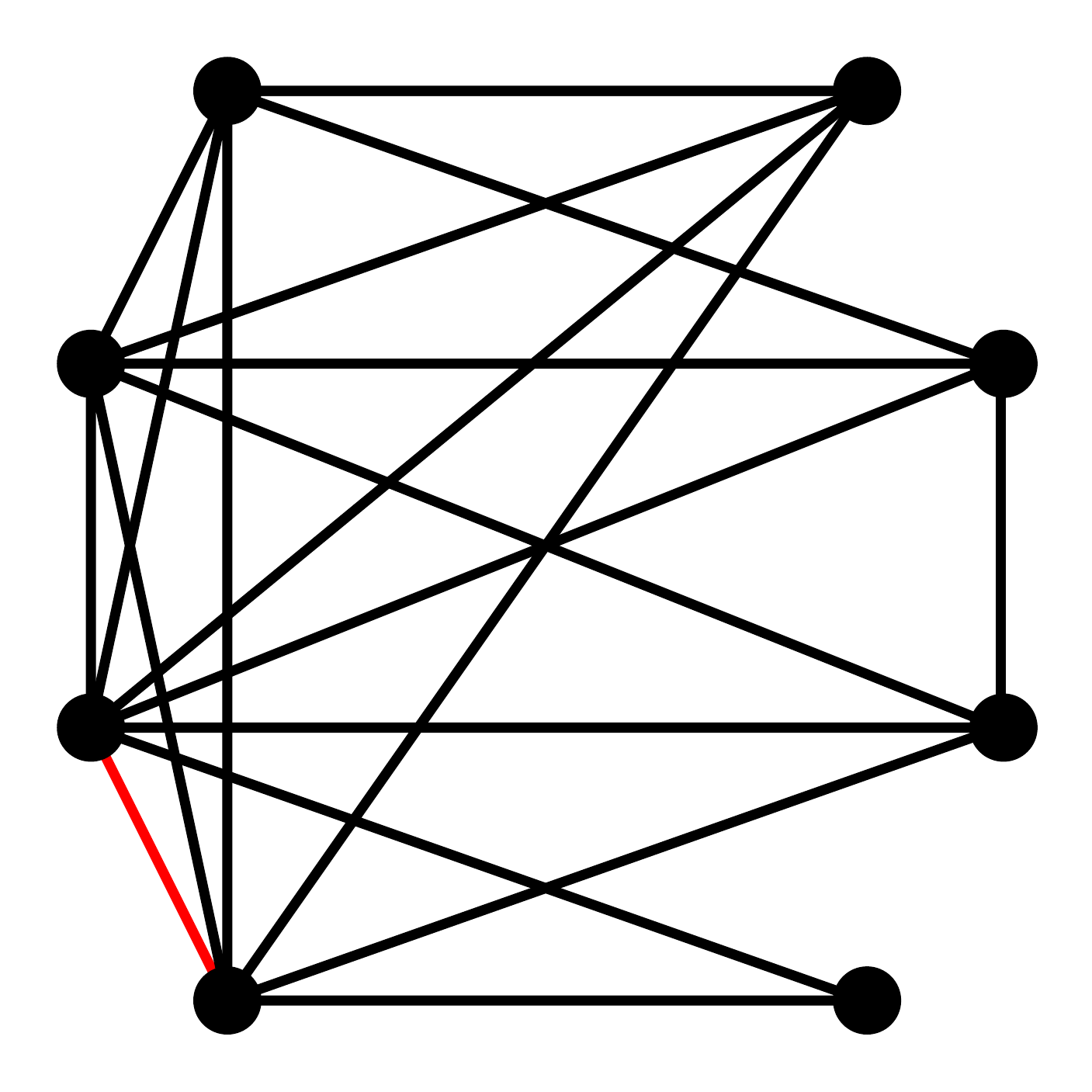}
				& \includegraphics[width=\figtabwidth]{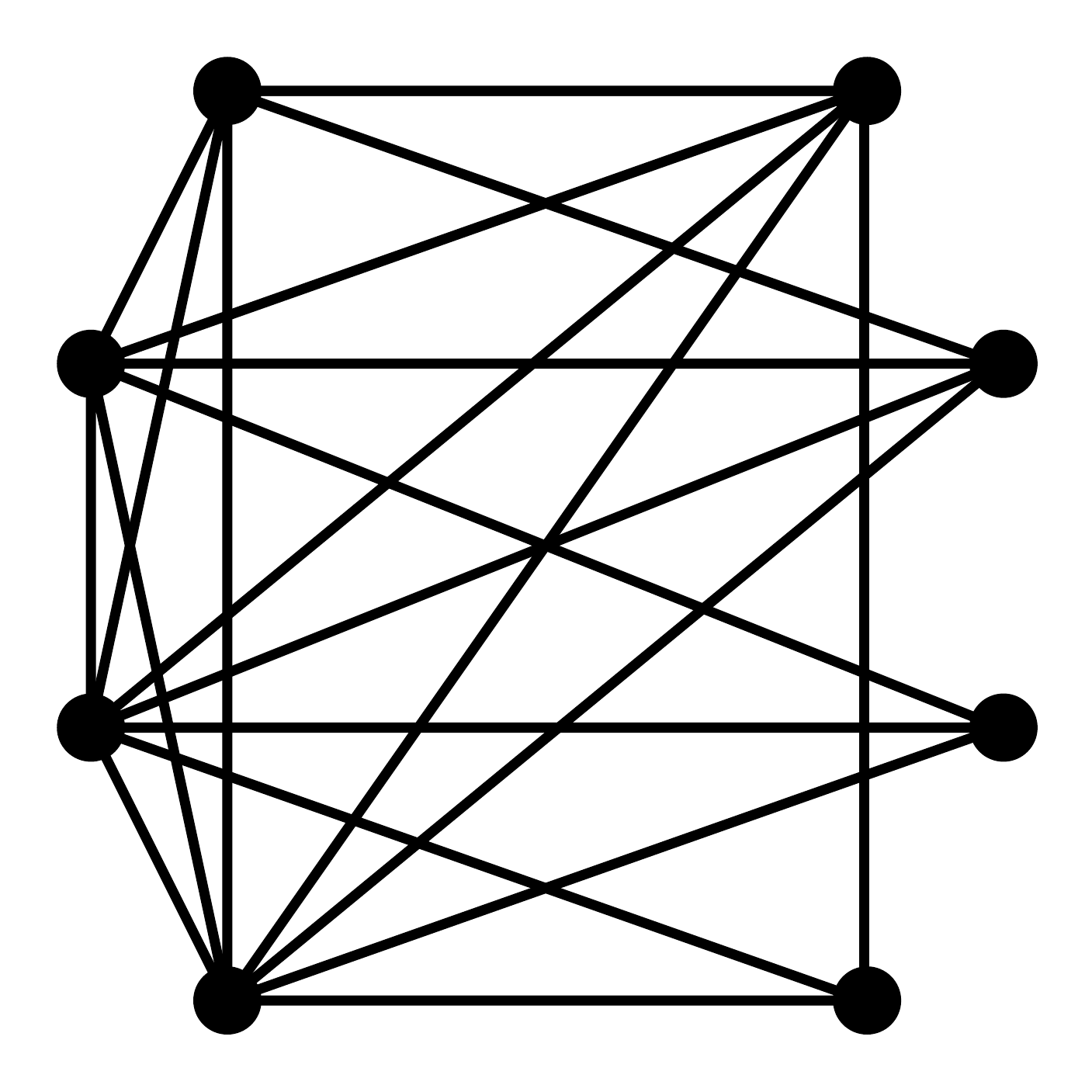}
				& \includegraphics[width=\figtabwidth]{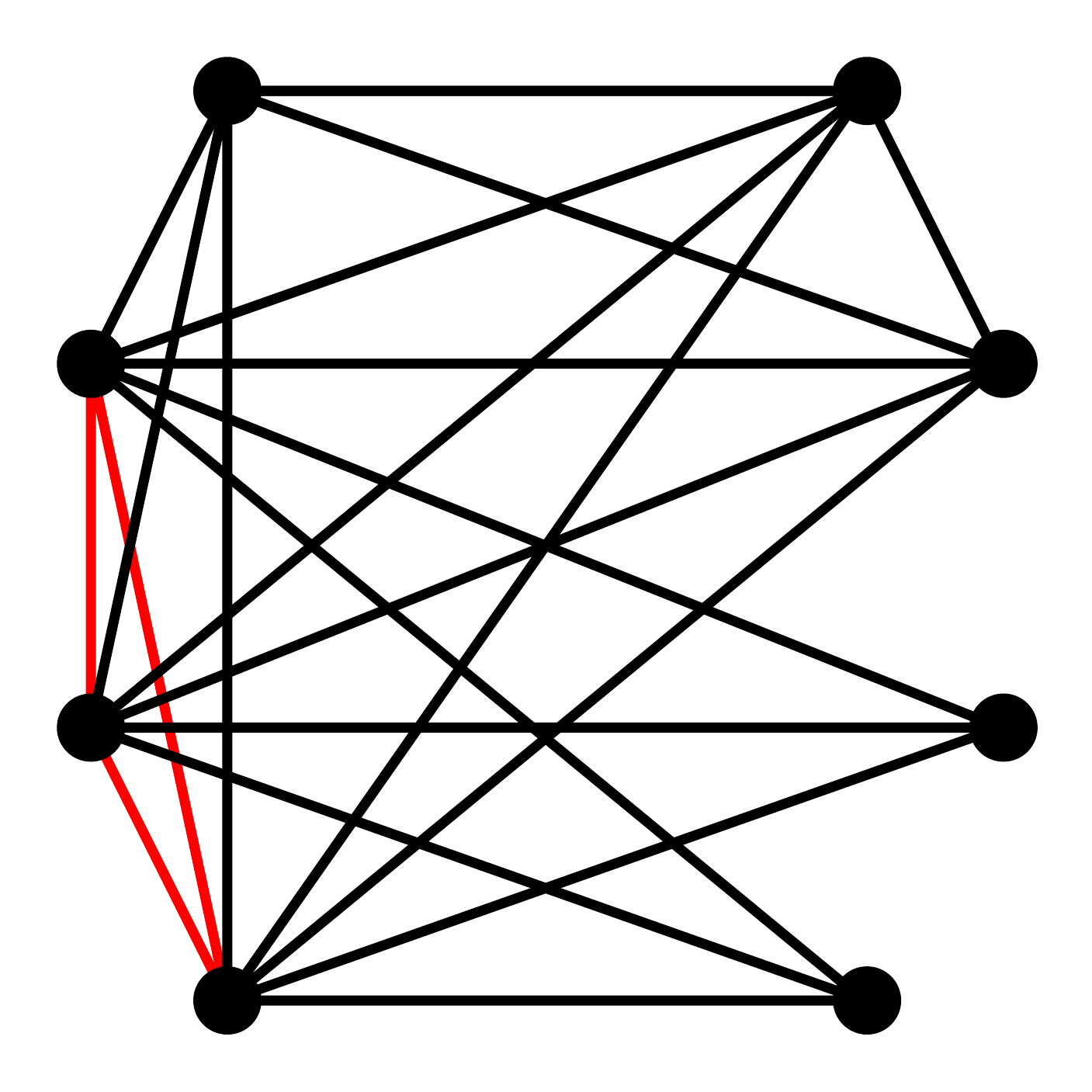}
				& \includegraphics[width=\figtabwidth]{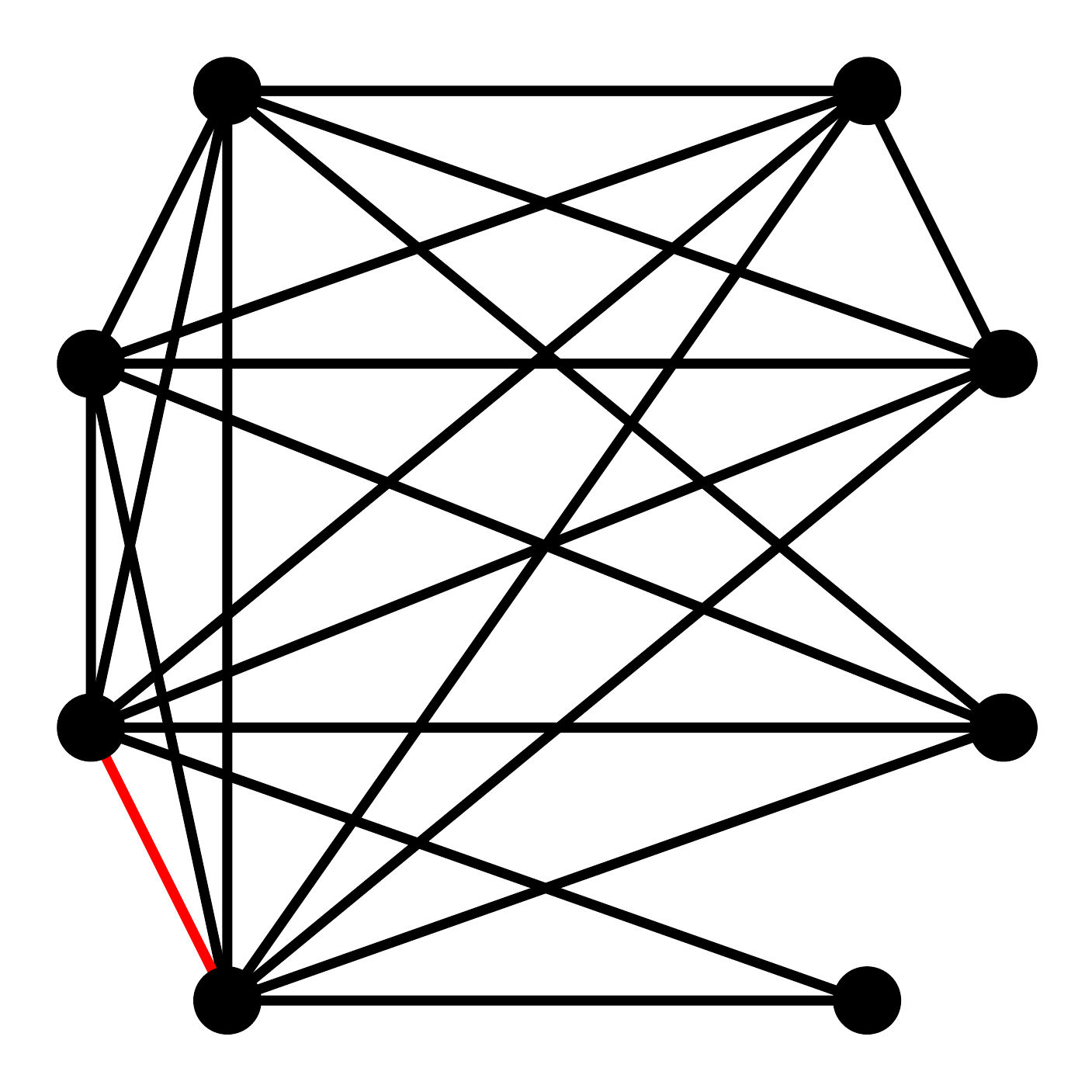}
				& \includegraphics[width=\figtabwidth]{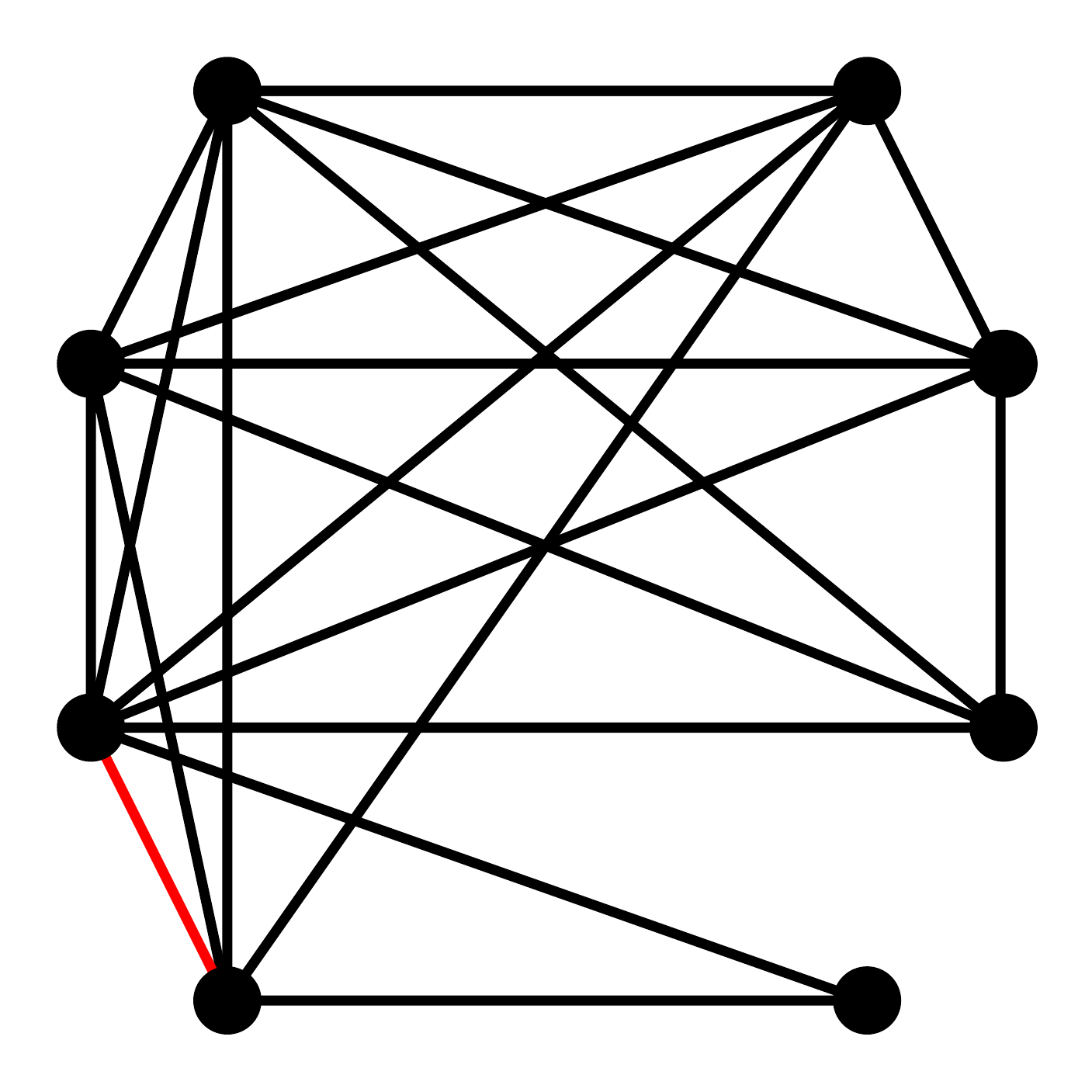}
				&\includegraphics[width=\figtabwidth]{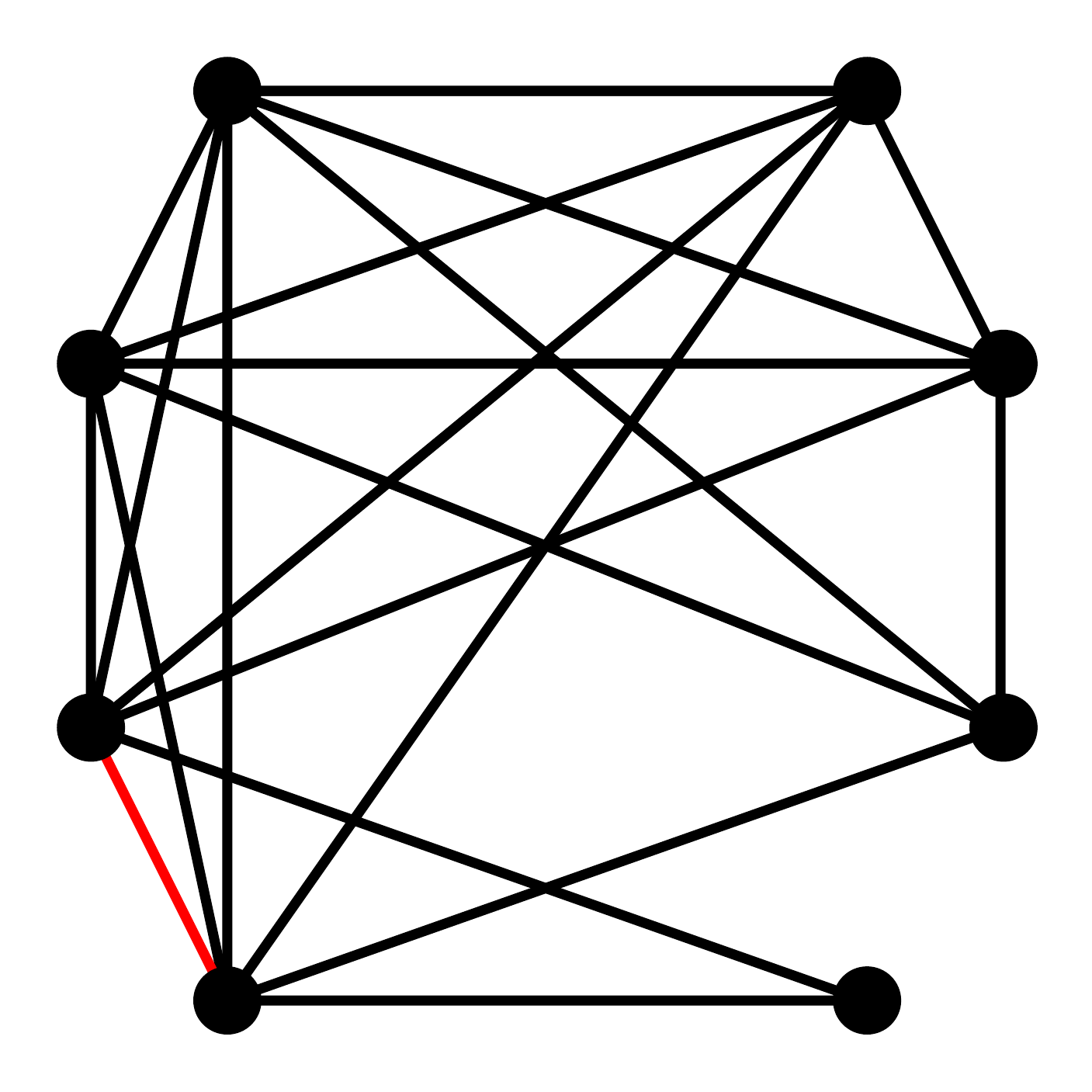}
				&\includegraphics[width=\figtabwidth]{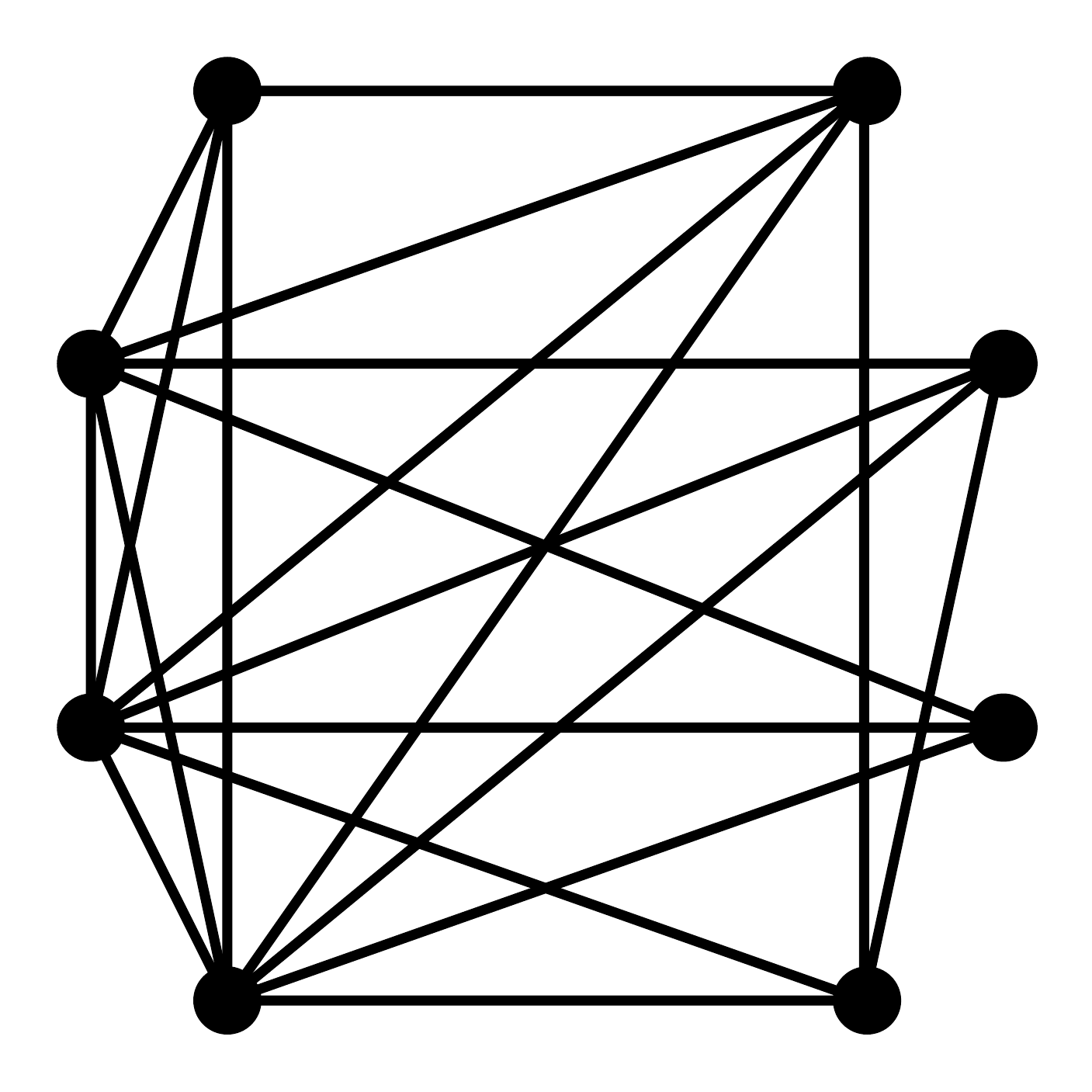}
				& \includegraphics[width=\figtabwidth]{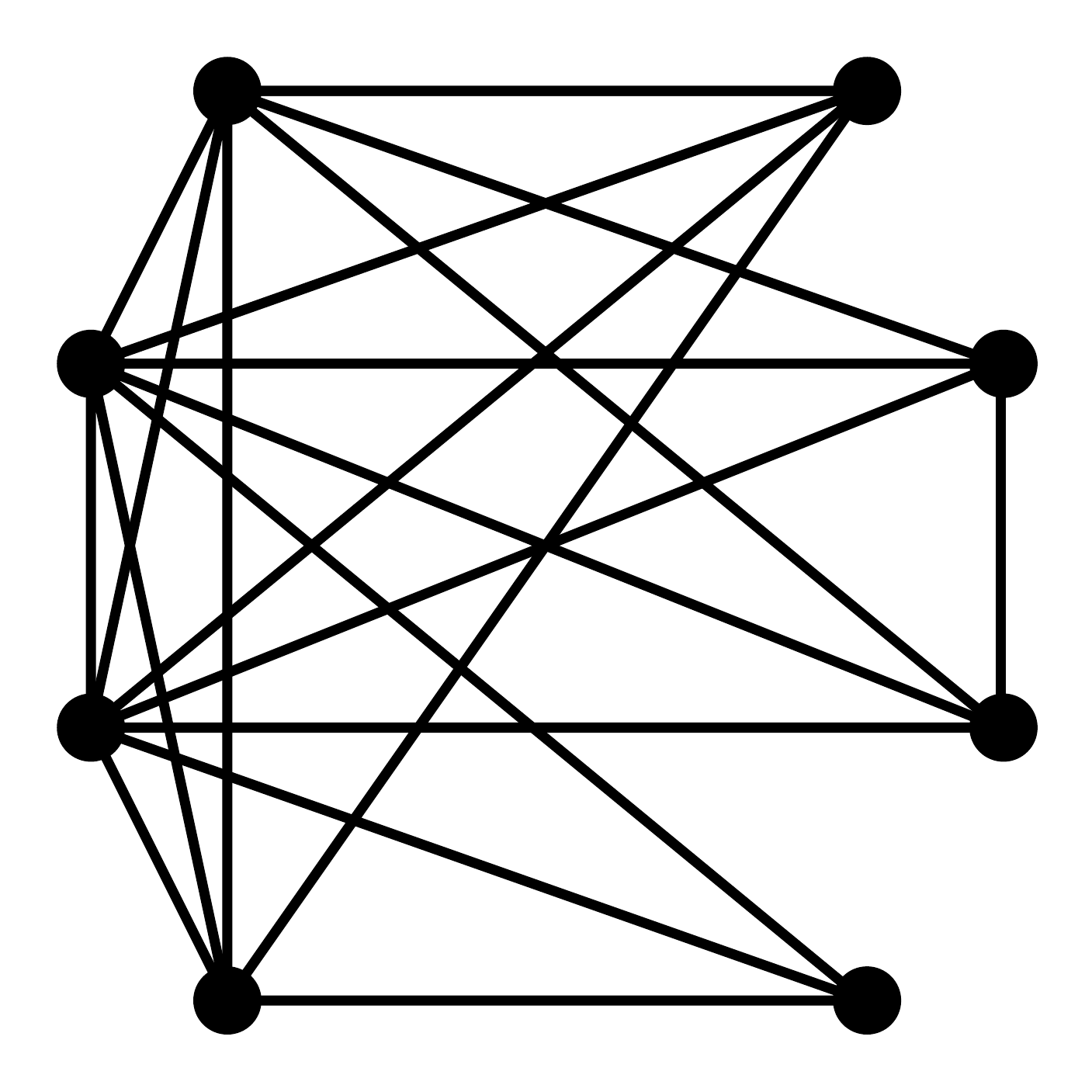}
				\\
			 $p=17$
			& $p=17$
			& $p=18$
			& $p=18$
			 & $p=19$
			 & $p=19$
			 & $p=19$
			 & $p=19$
			\\
		 	 \includegraphics[width=\figtabwidth]{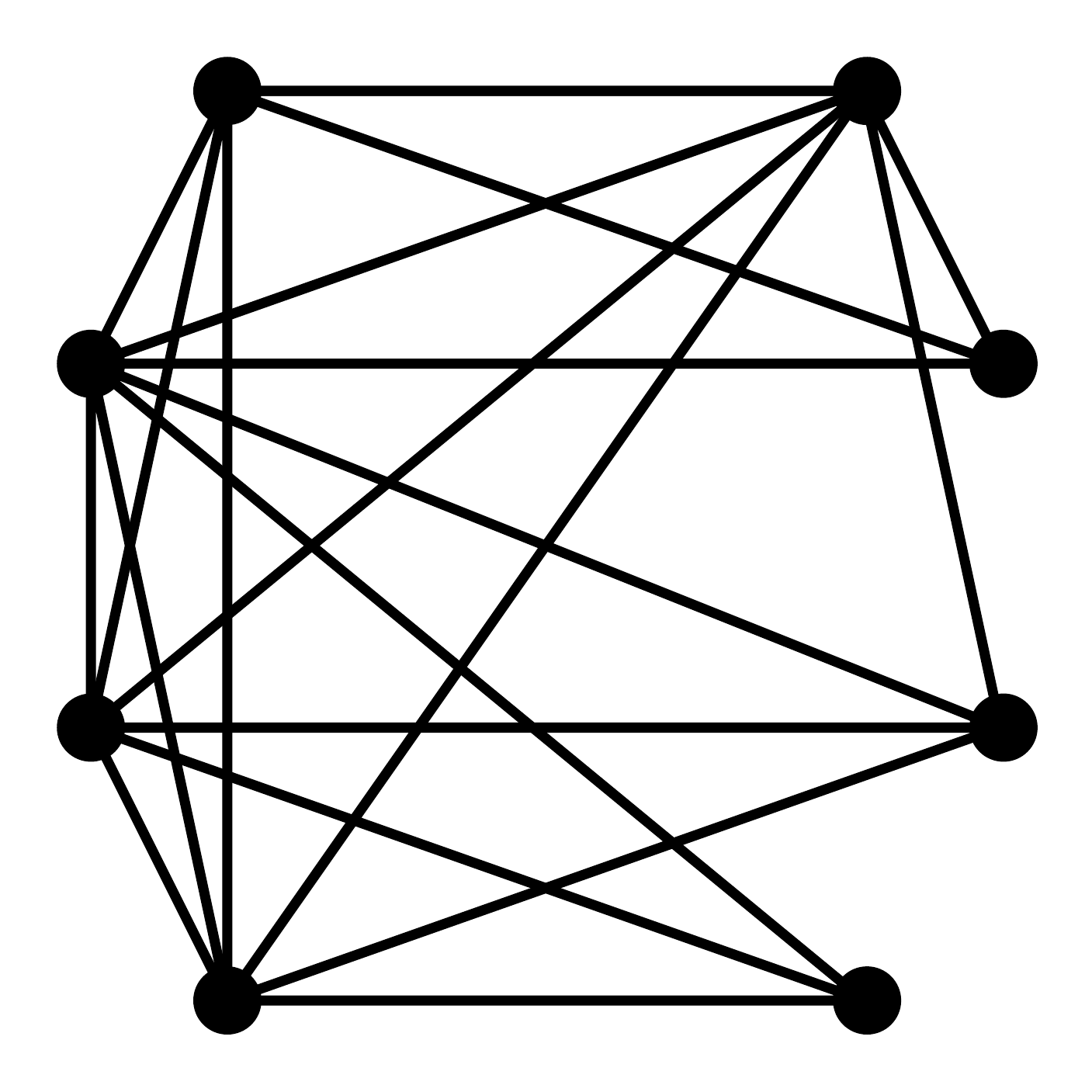}
			&  \includegraphics[width=\figtabwidth]{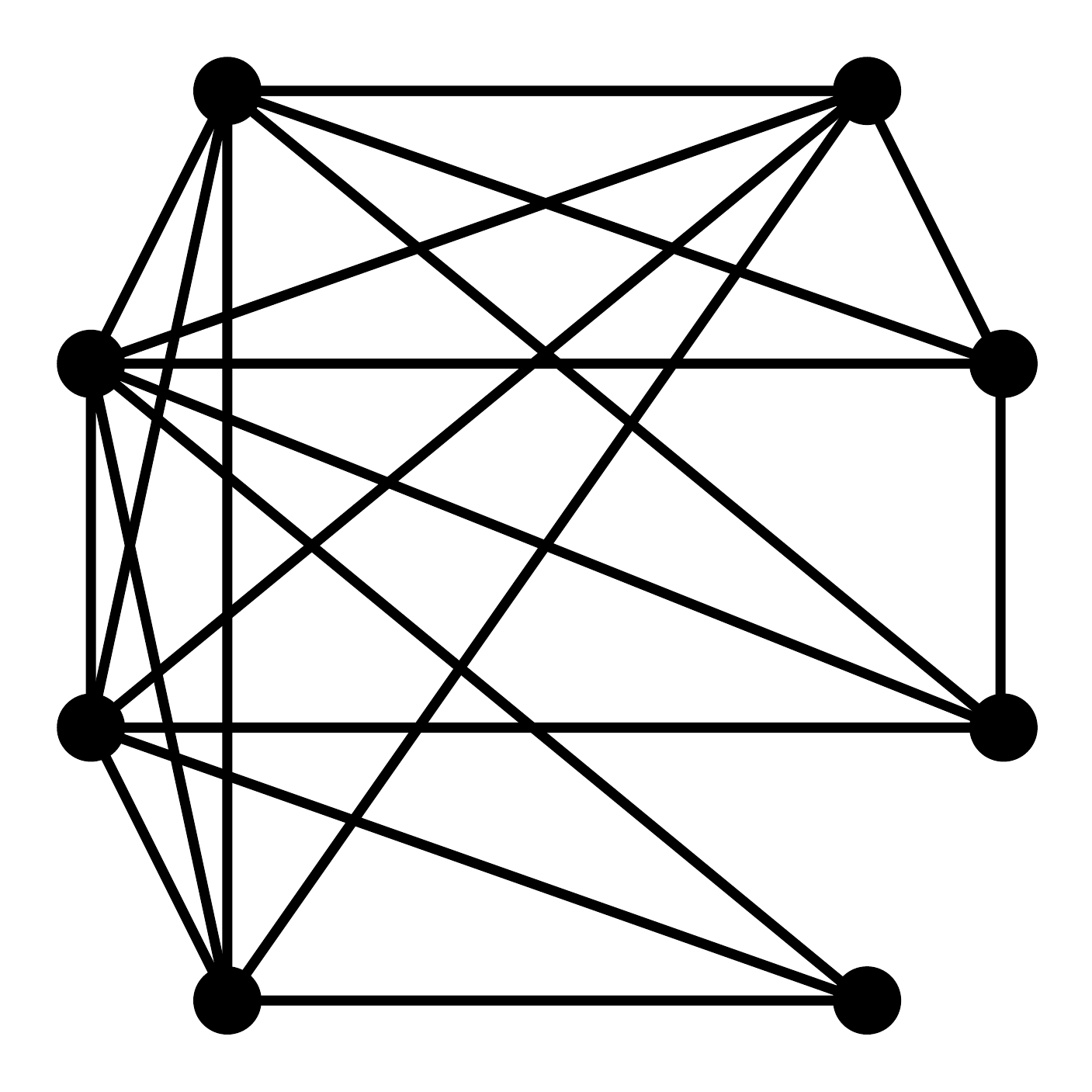}
			& \includegraphics[width=\figtabwidth]{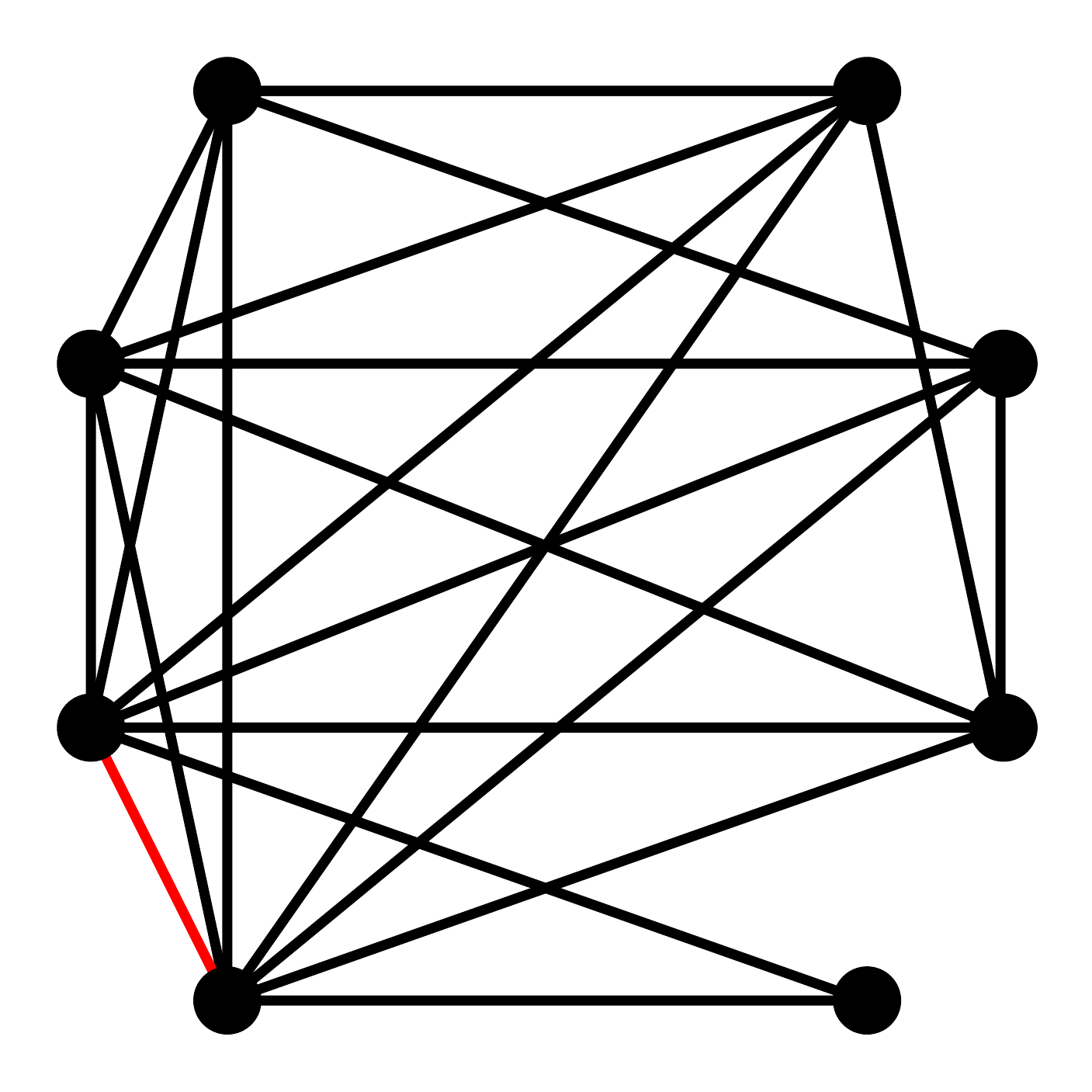}
		    & \includegraphics[width=\figtabwidth]{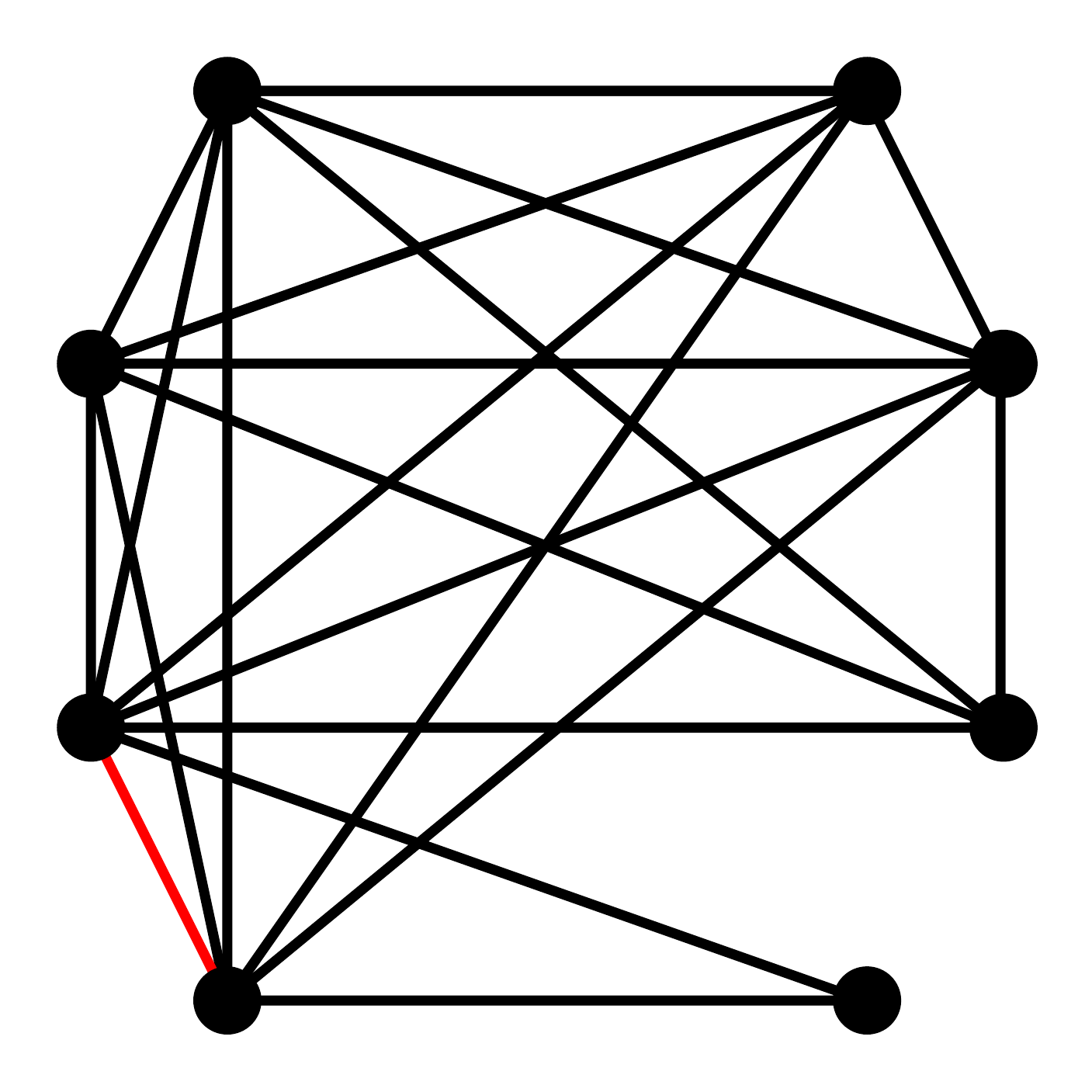}
		    & \includegraphics[width=\figtabwidth]{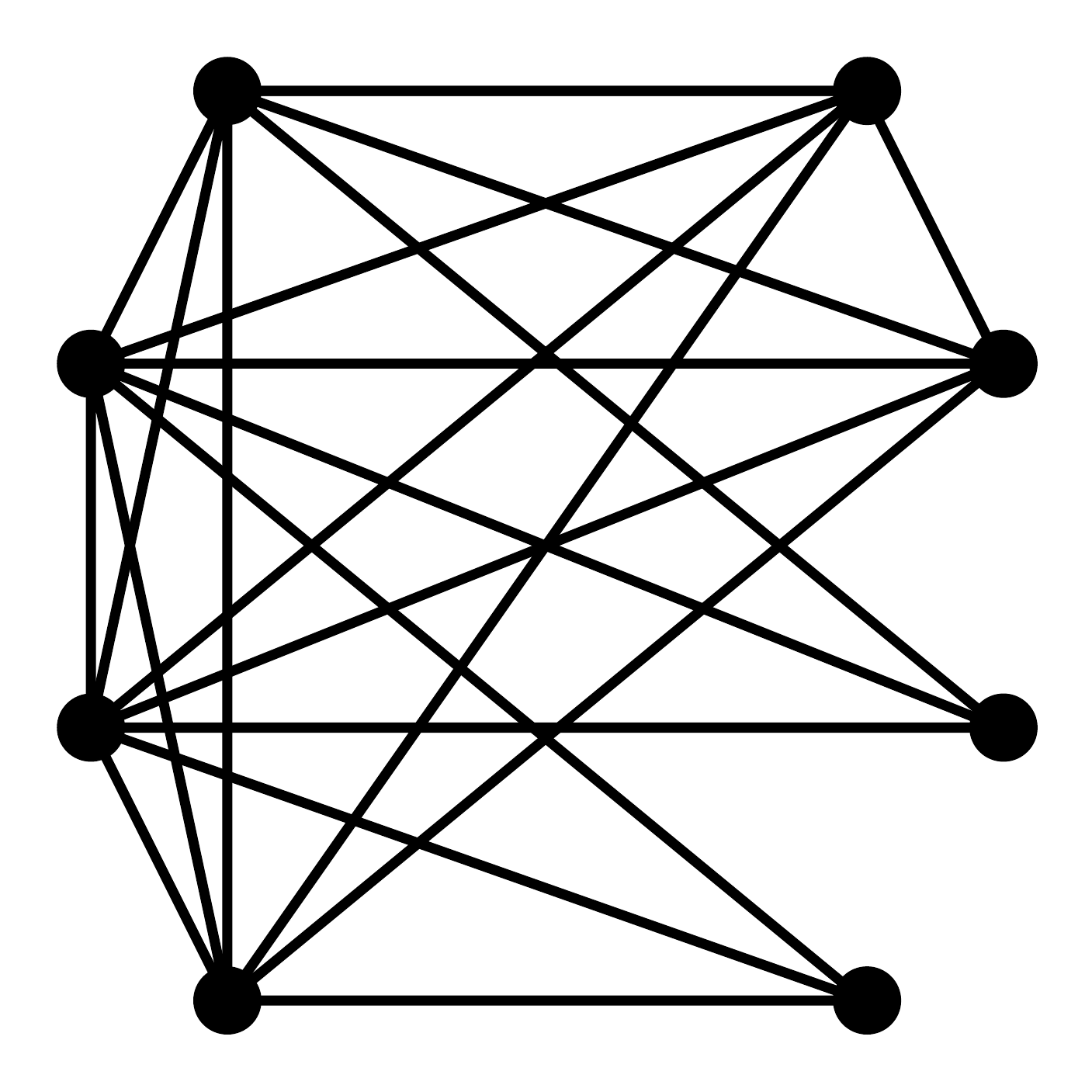}
		    & \includegraphics[width=\figtabwidth]{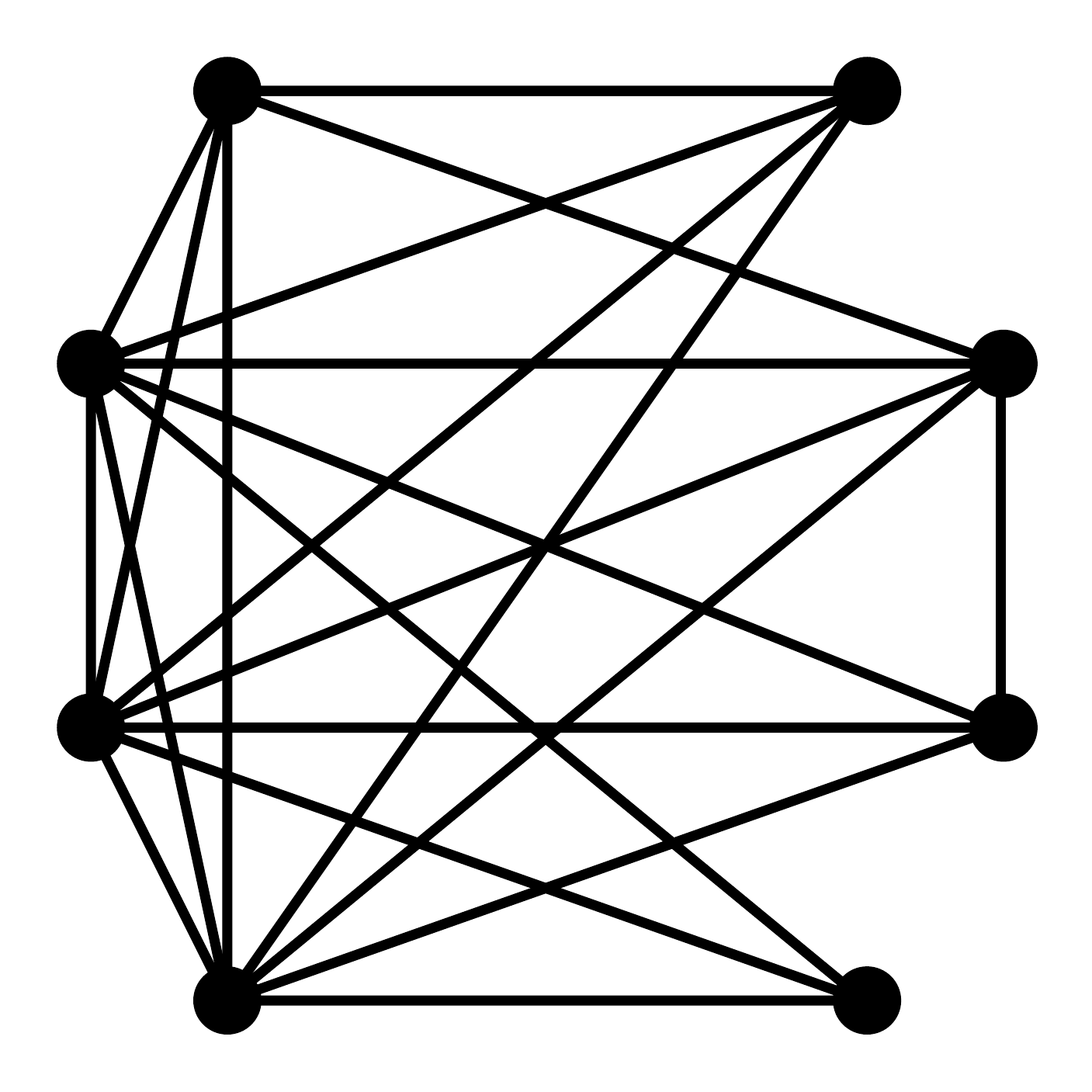}
		    & \includegraphics[width=\figtabwidth]{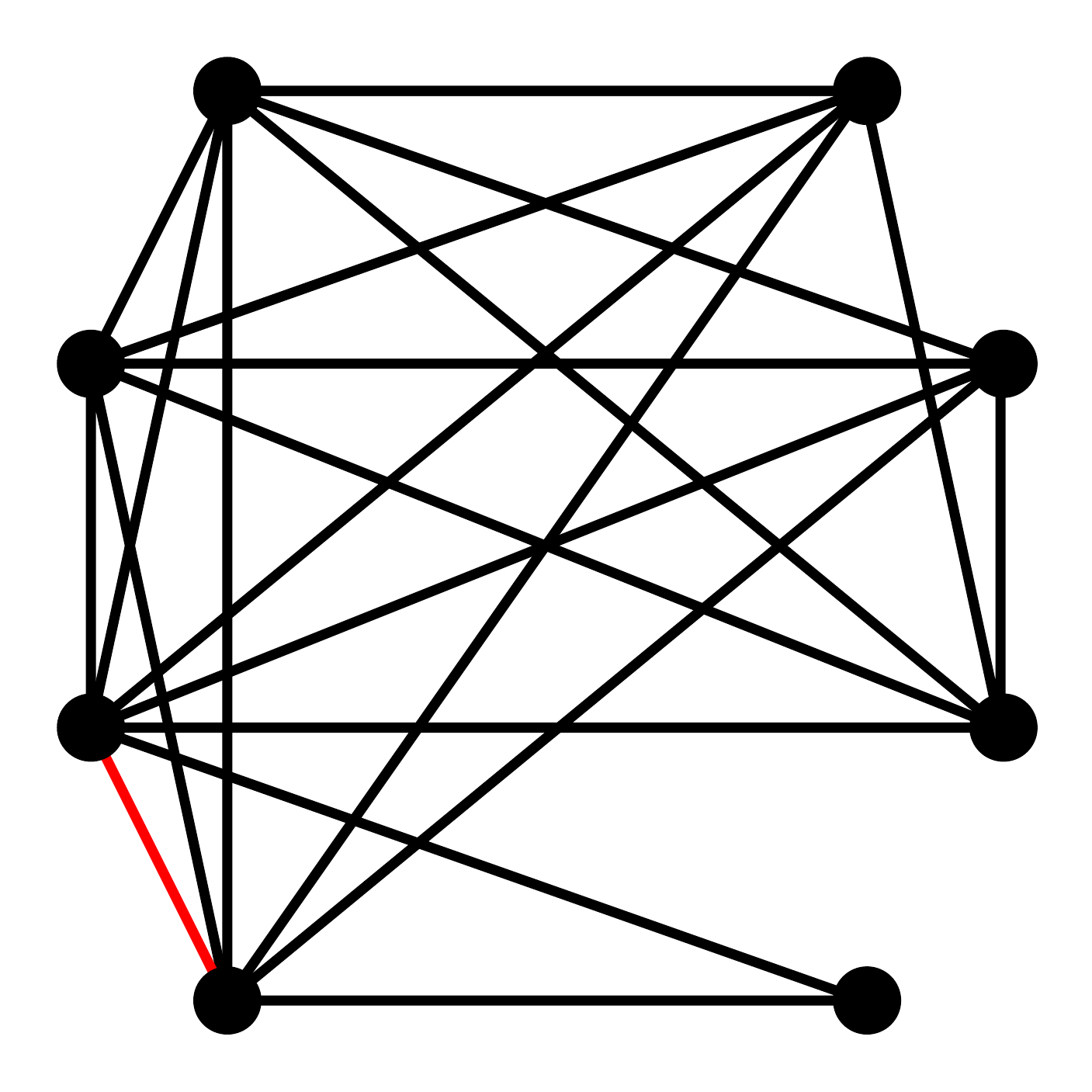}
		 	& \includegraphics[width=\figtabwidth]{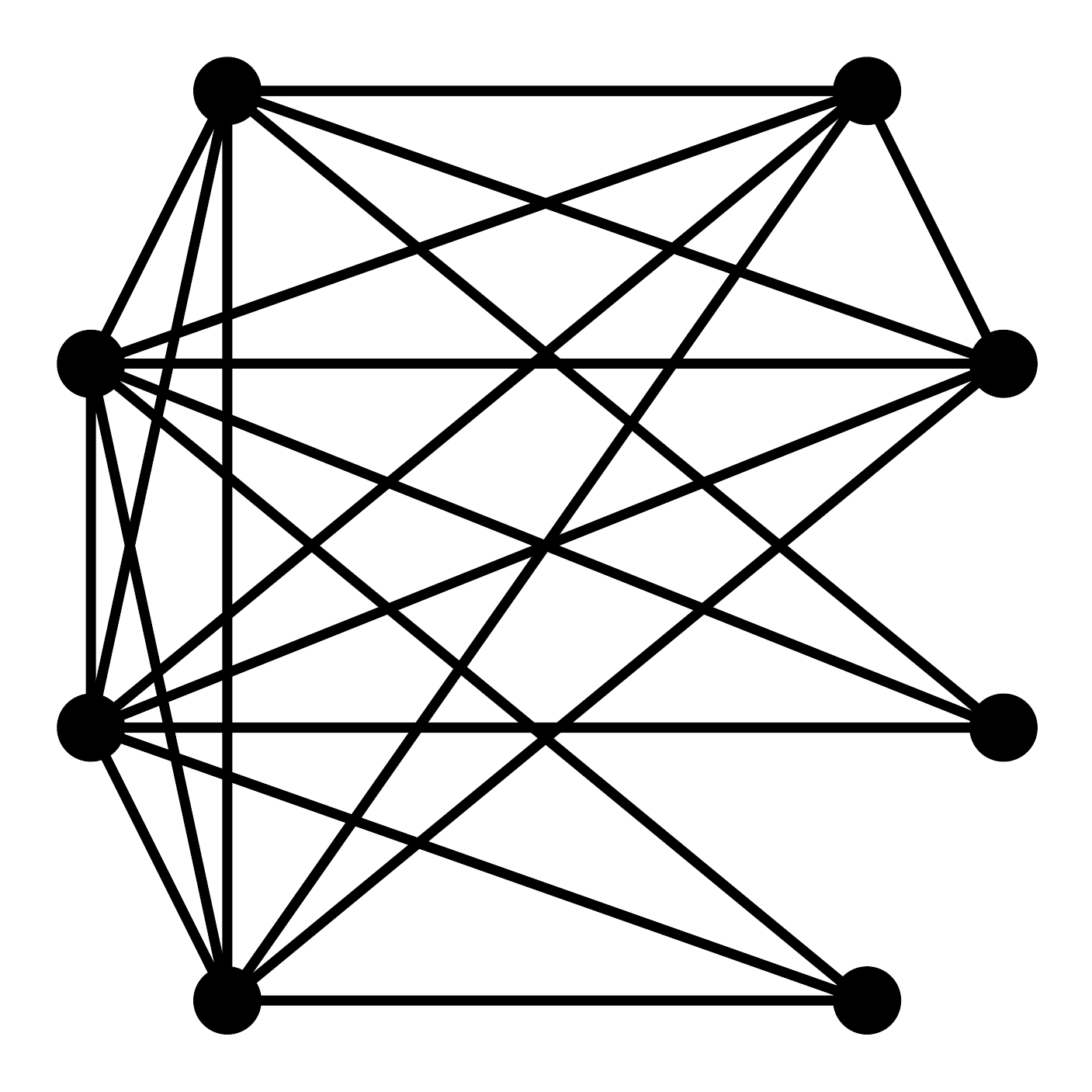}
			\\
			$p=19$
				& $p=19$
				& $p=20$
				& $p=21$
				& $p=21$ 
				& $p=21$
				& $p=22$
				& $p=23$ \\
		 \includegraphics[width=\figtabwidth]{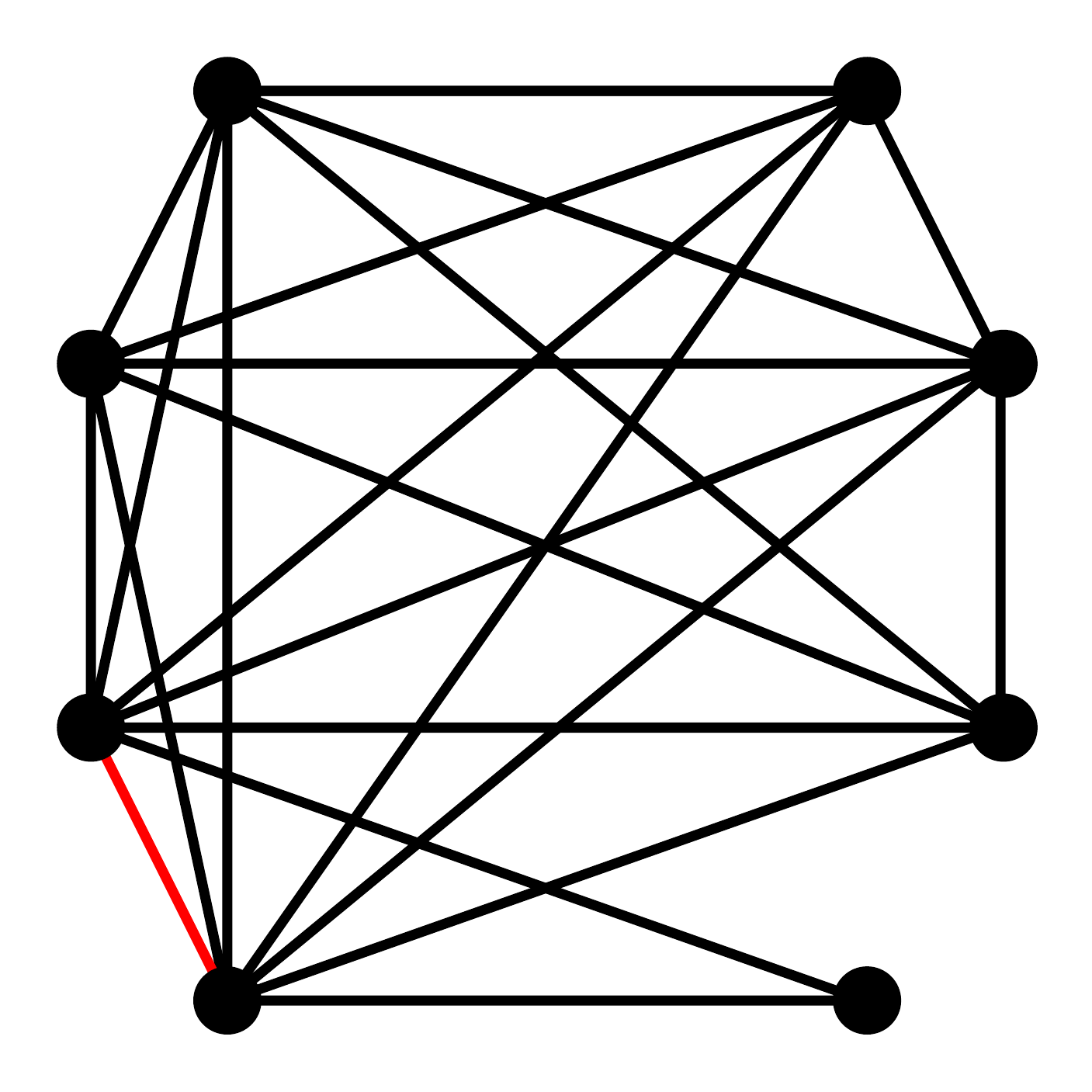}
		 	&  \includegraphics[width=\figtabwidth]{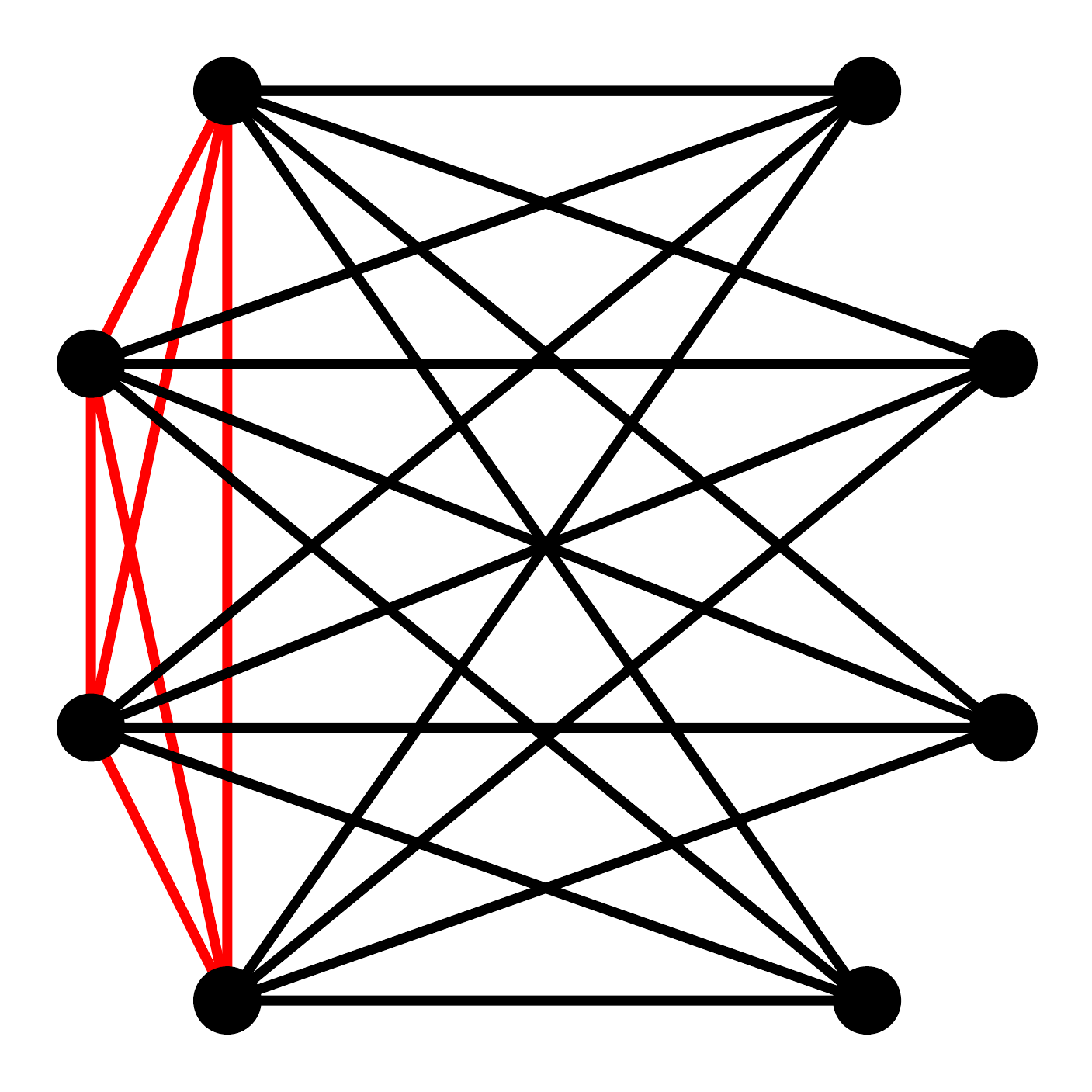}
		 	&  \includegraphics[width=\figtabwidth]{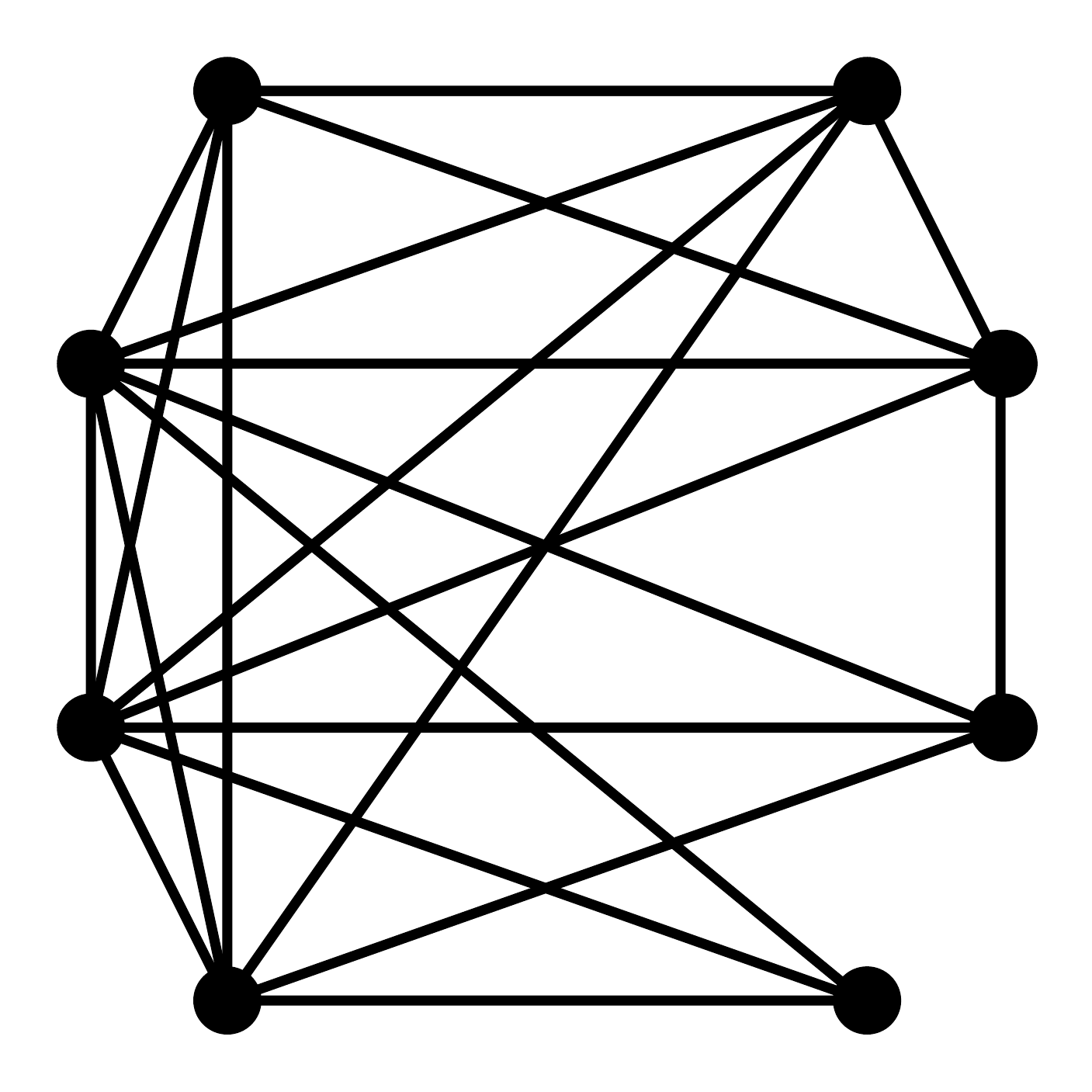}
		 	& \includegraphics[width=\figtabwidth]{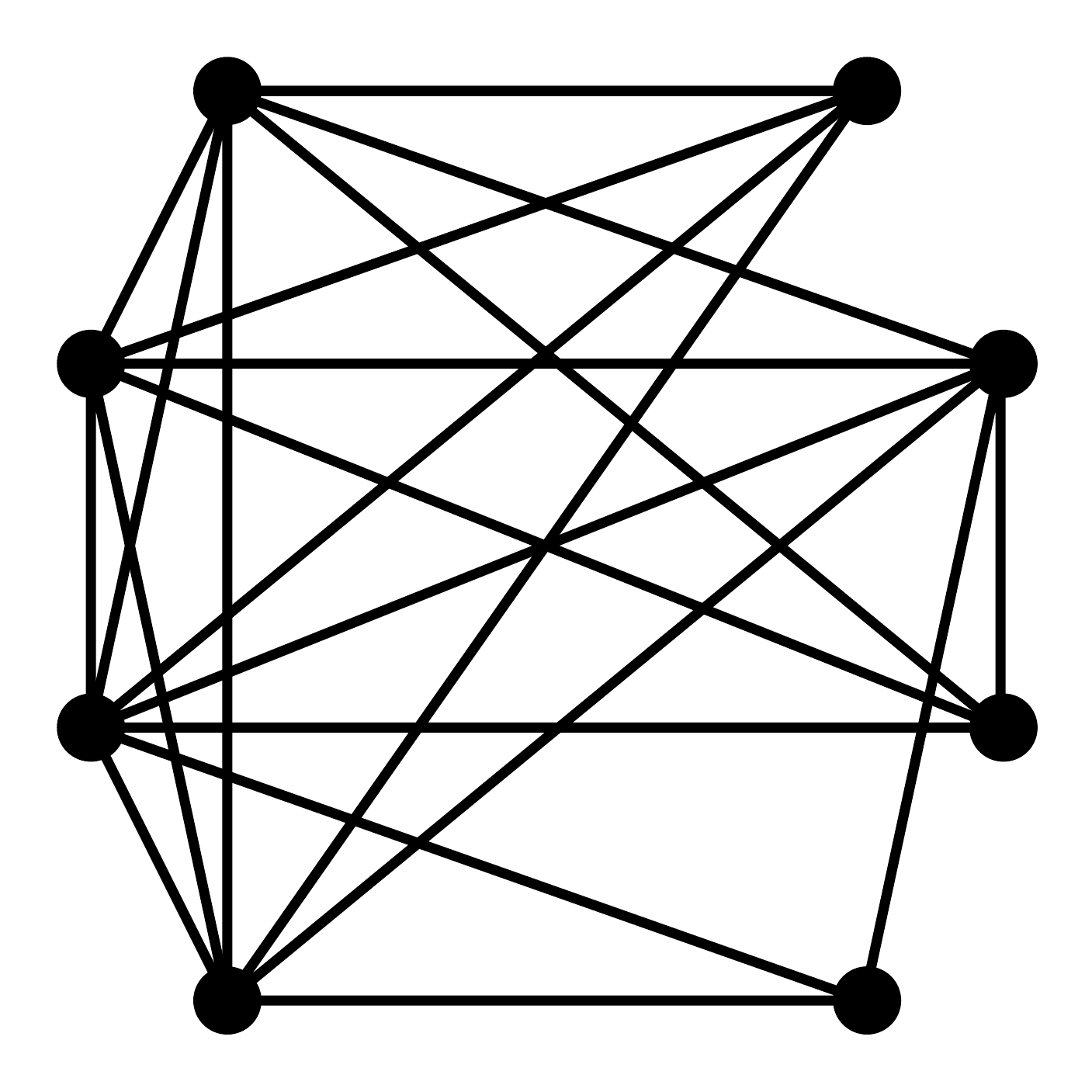}
			&\includegraphics[width=\figtabwidth]{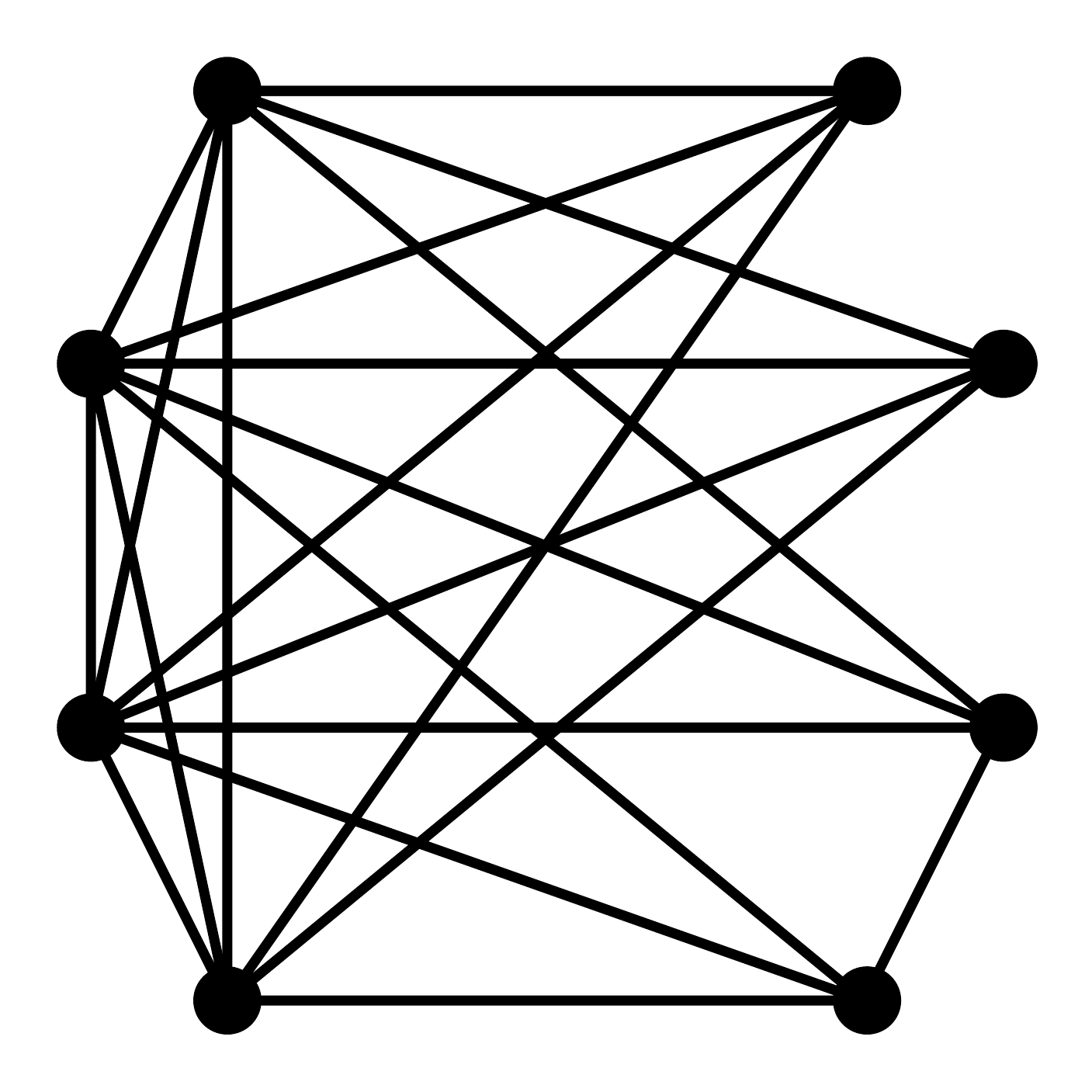}
		 	& \includegraphics[width=\figtabwidth]{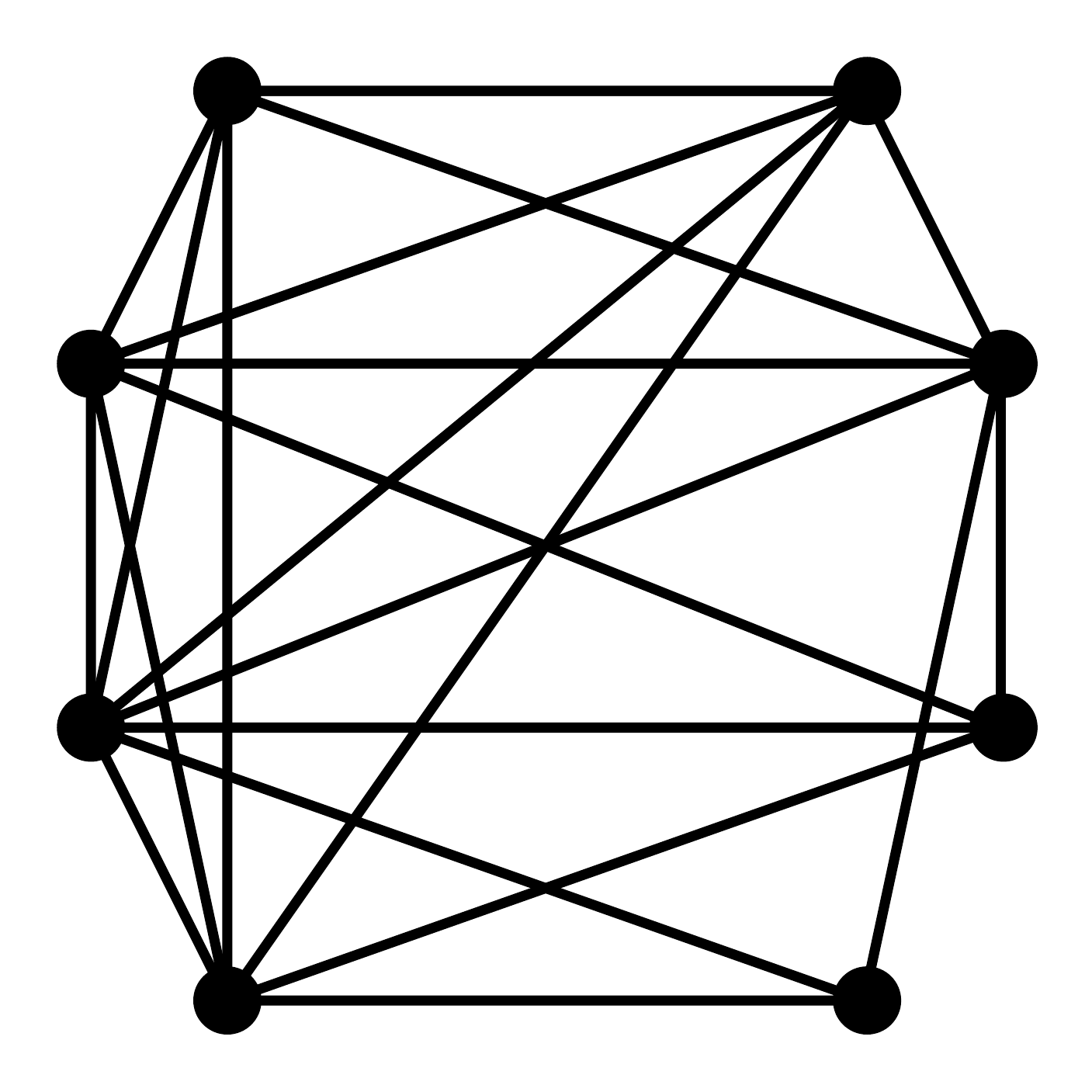}
			& \includegraphics[width=\figtabwidth]{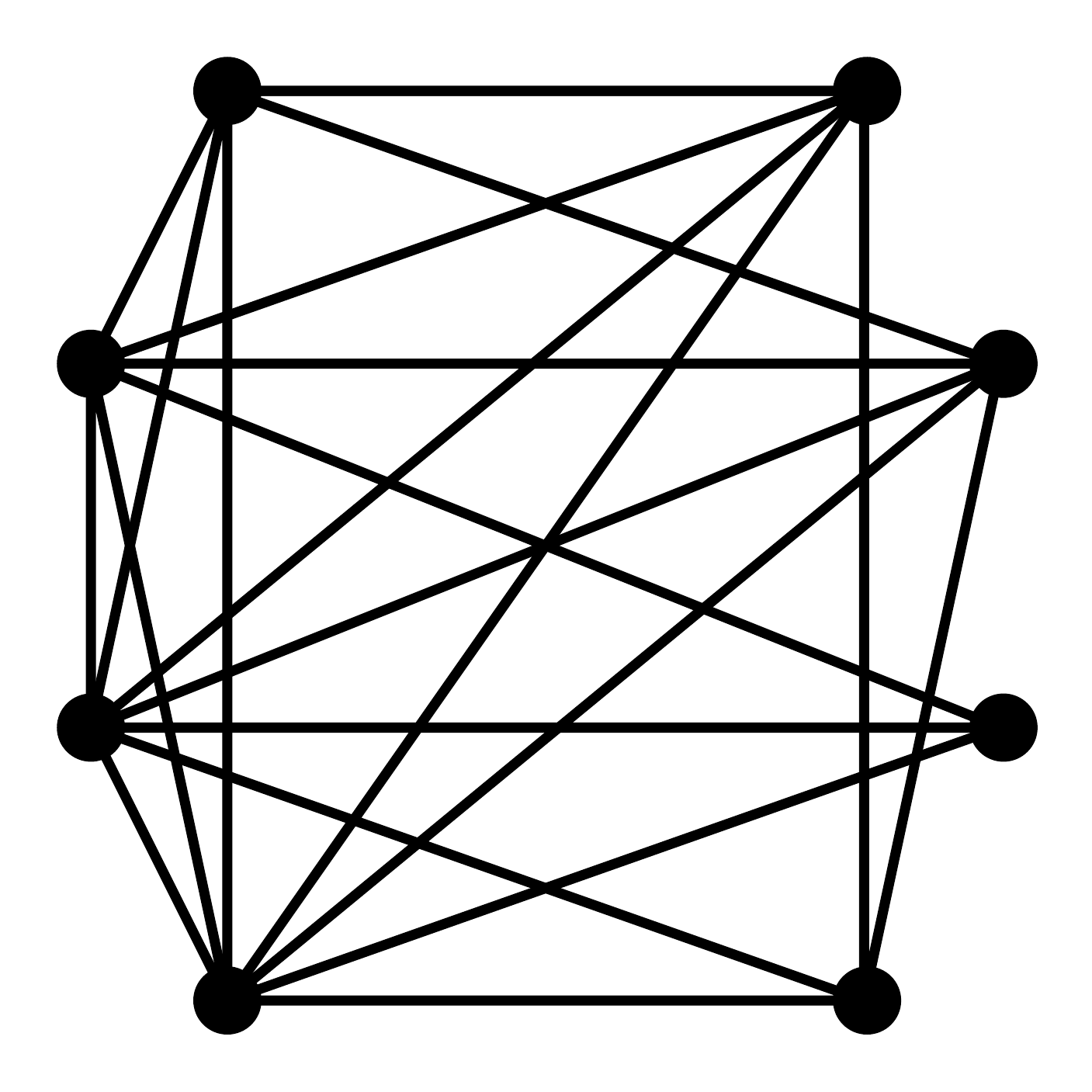}
		 	& \includegraphics[width=\figtabwidth]{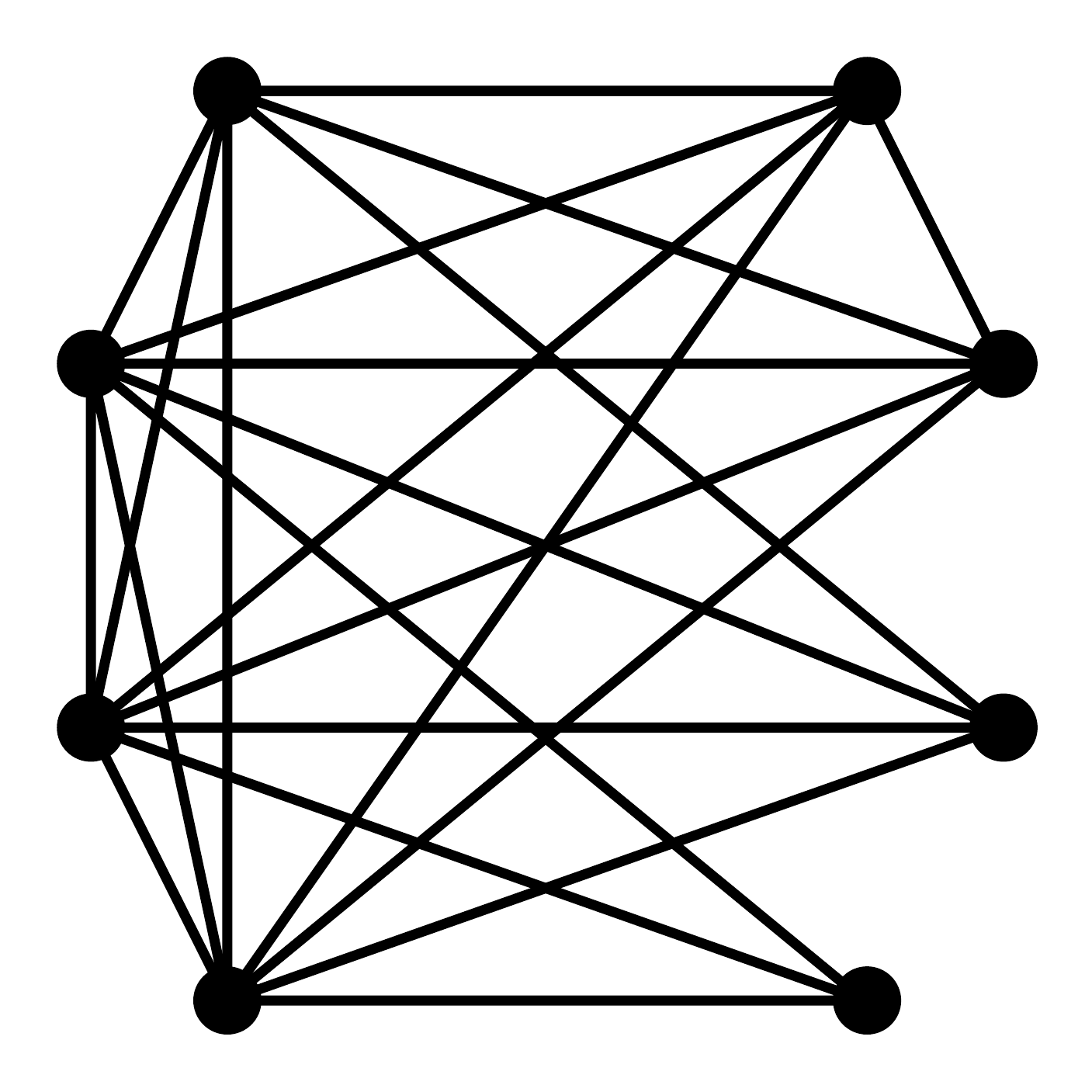}
			\\
			$p=24$ & $p=24$ 
		 	& $p=25$ & $p=25$&  $p=26$ & $p=26$ & $p=26$ & $p=27$ \\
	\end{tabular}
	\caption{\label{fig:ExtremalElemGraphs}The $p$-extremal elementary graphs %
				  where $1 \leq p \leq 27$.}
\end{figure}

To describe the complete structural characterization of $p$-extremal graphs
	on $n$ vertices for all even $n \geq n_p$,
	we apply Theorem \ref{thm:hswystructure}.
An important step in applying Theorem \ref{thm:hswystructure} is
	to consider every factorization $p = \prod p_i$
	and to check which spires 
	are generated by the $p_i$-extremal elementary graphs.
We describe these structures based on the types of constructions given by these factorizations.
It is necessary to consider the $p$-extremal elementary graphs for $1 \leq p \leq 10$,
	in Figure \ref{fig:PowersOf2}.
	
\def\powtabwidth{0.12\textwidth}
\begin{figure}[t]
	\centering
	\begin{tabular}[h]{ccccccc}
 \includegraphics[width=\figtabwidth]{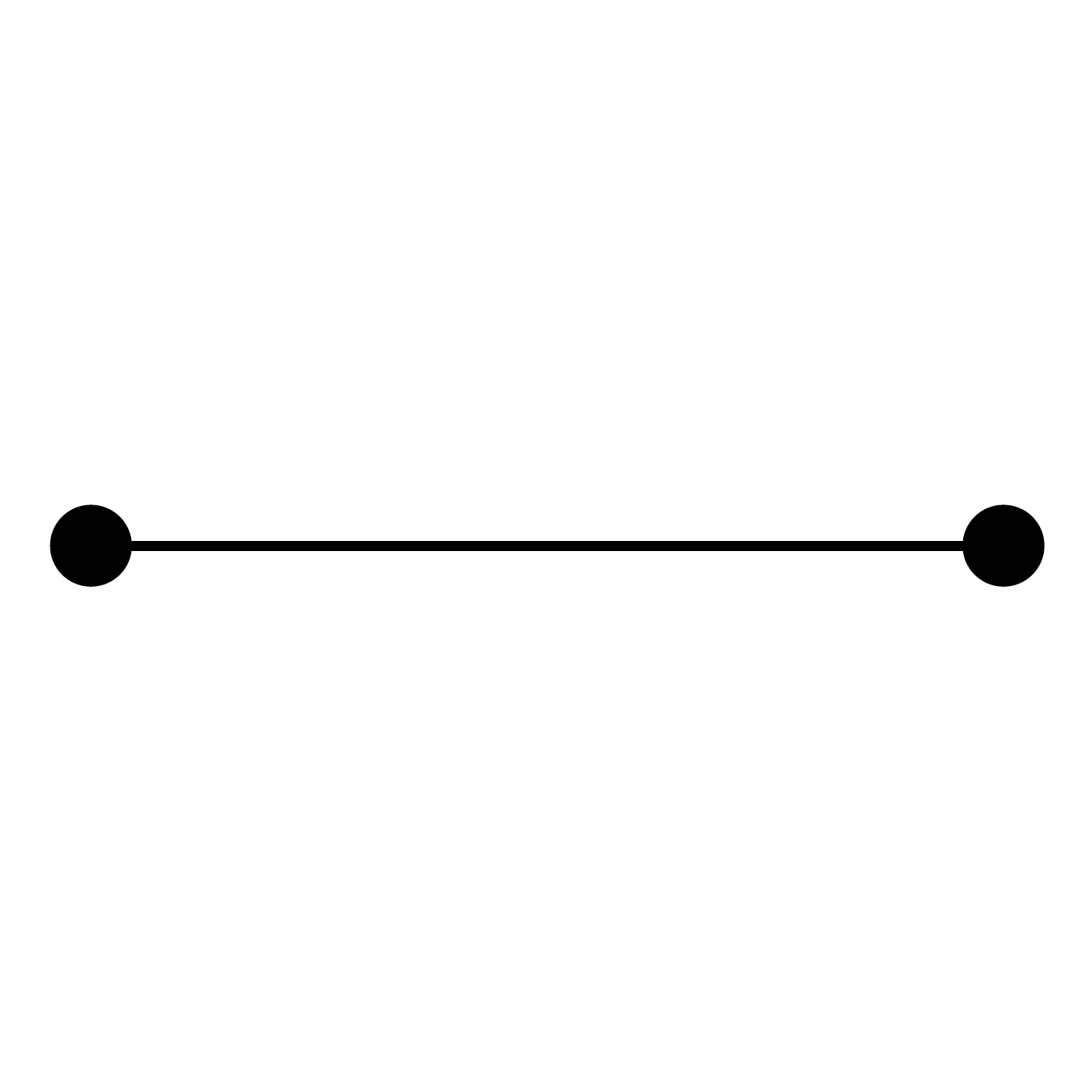}
			&  \includegraphics[width=\powtabwidth]{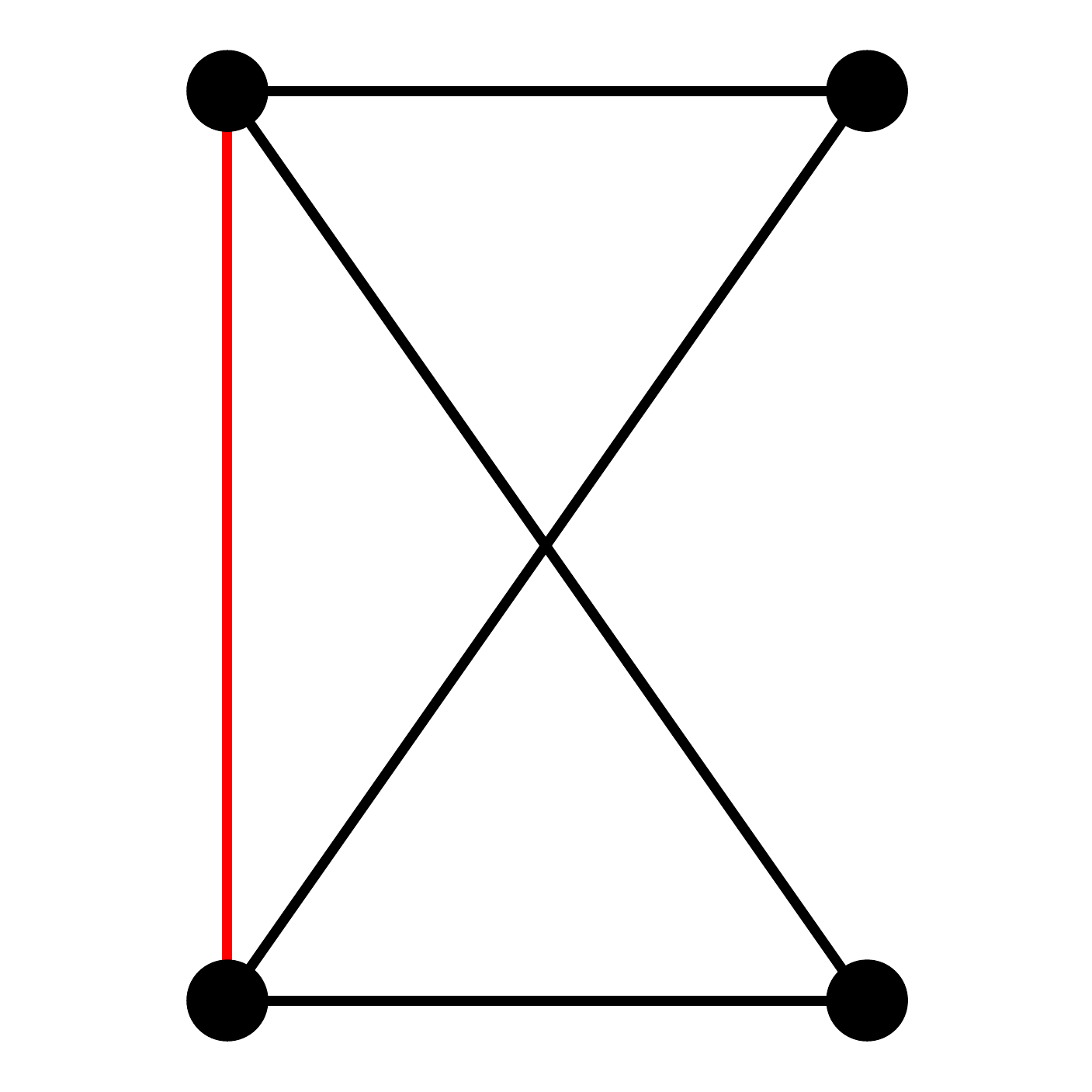}
		    &  \includegraphics[width= \powtabwidth]{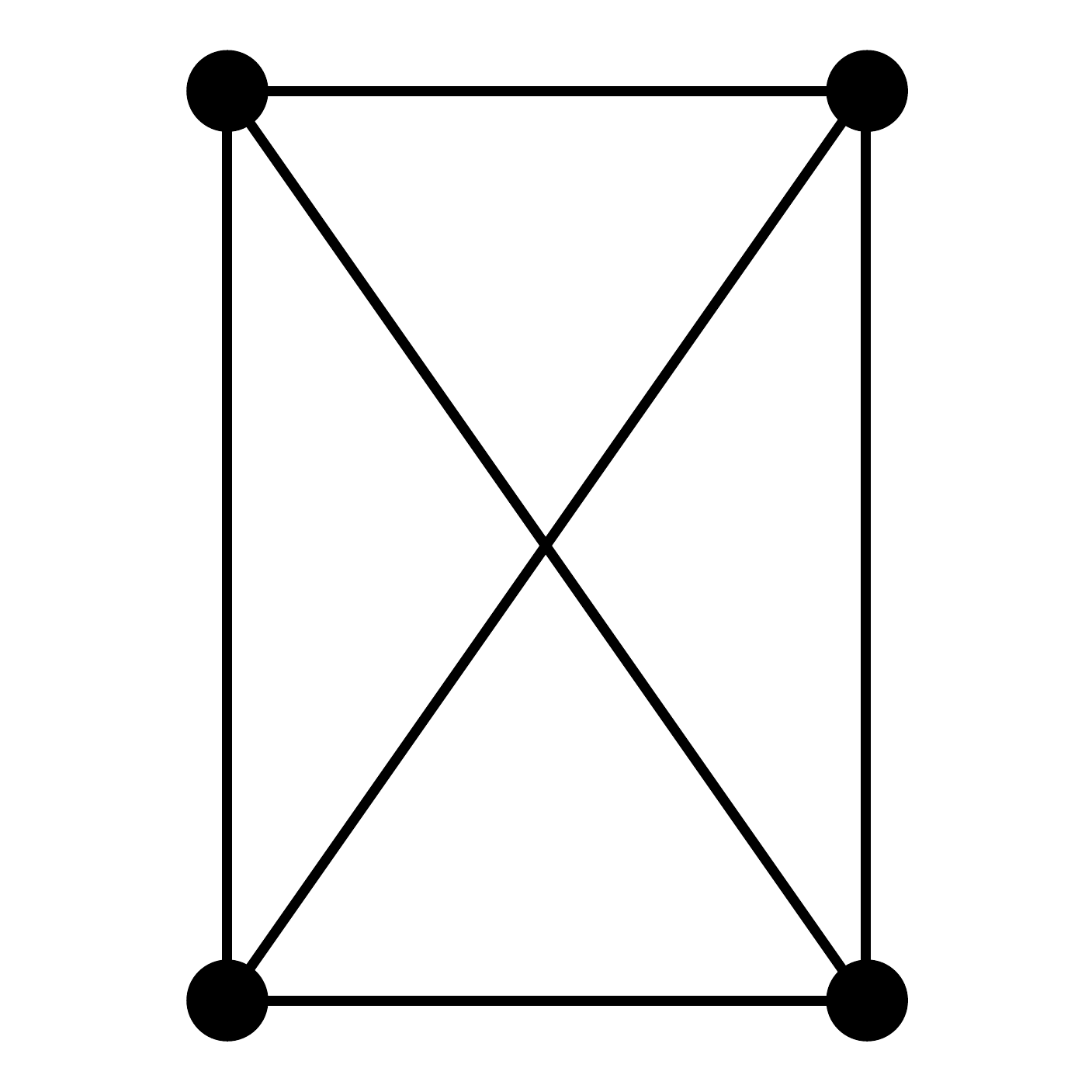}
		    & \includegraphics[width= \powtabwidth]{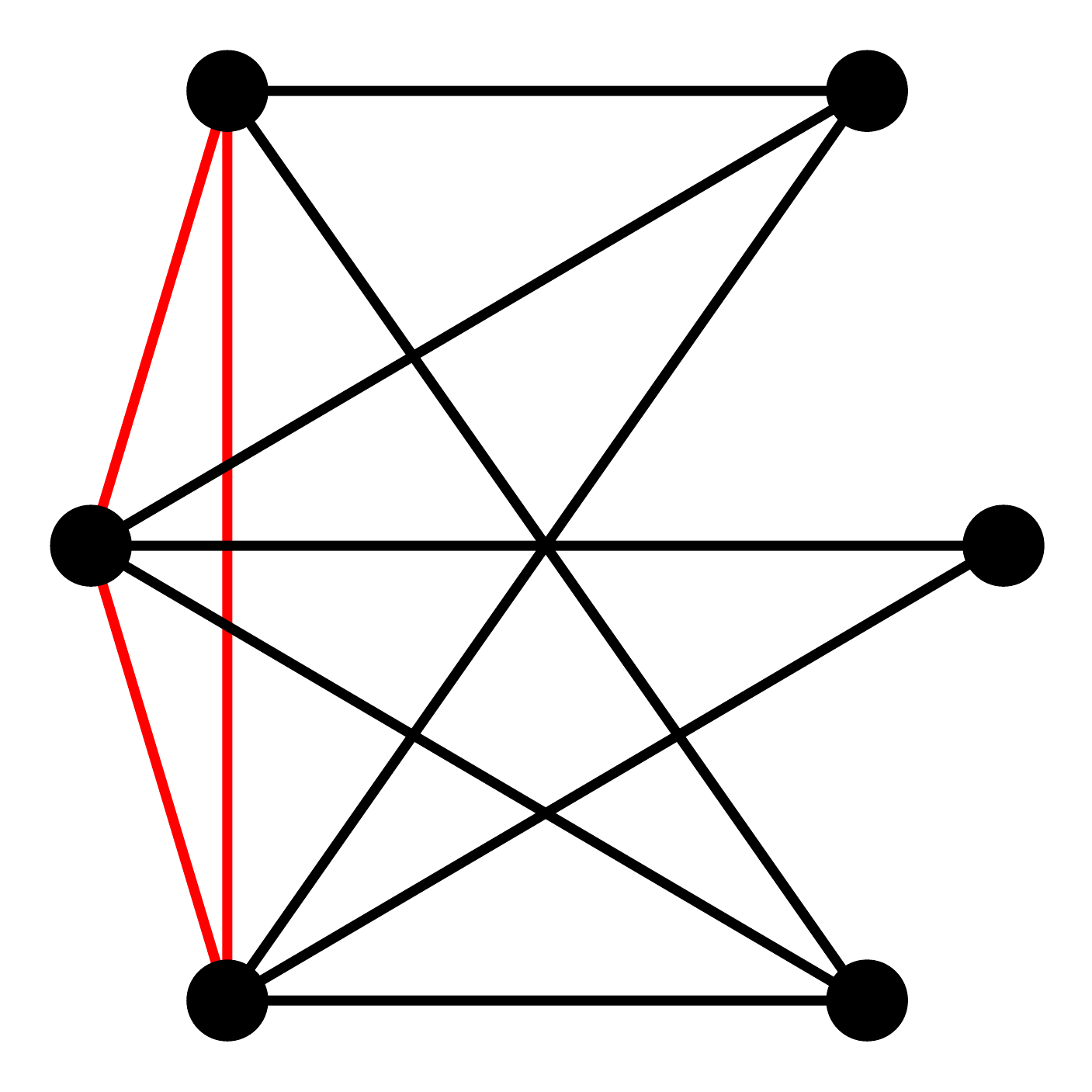}
		    & \includegraphics[width= \powtabwidth]{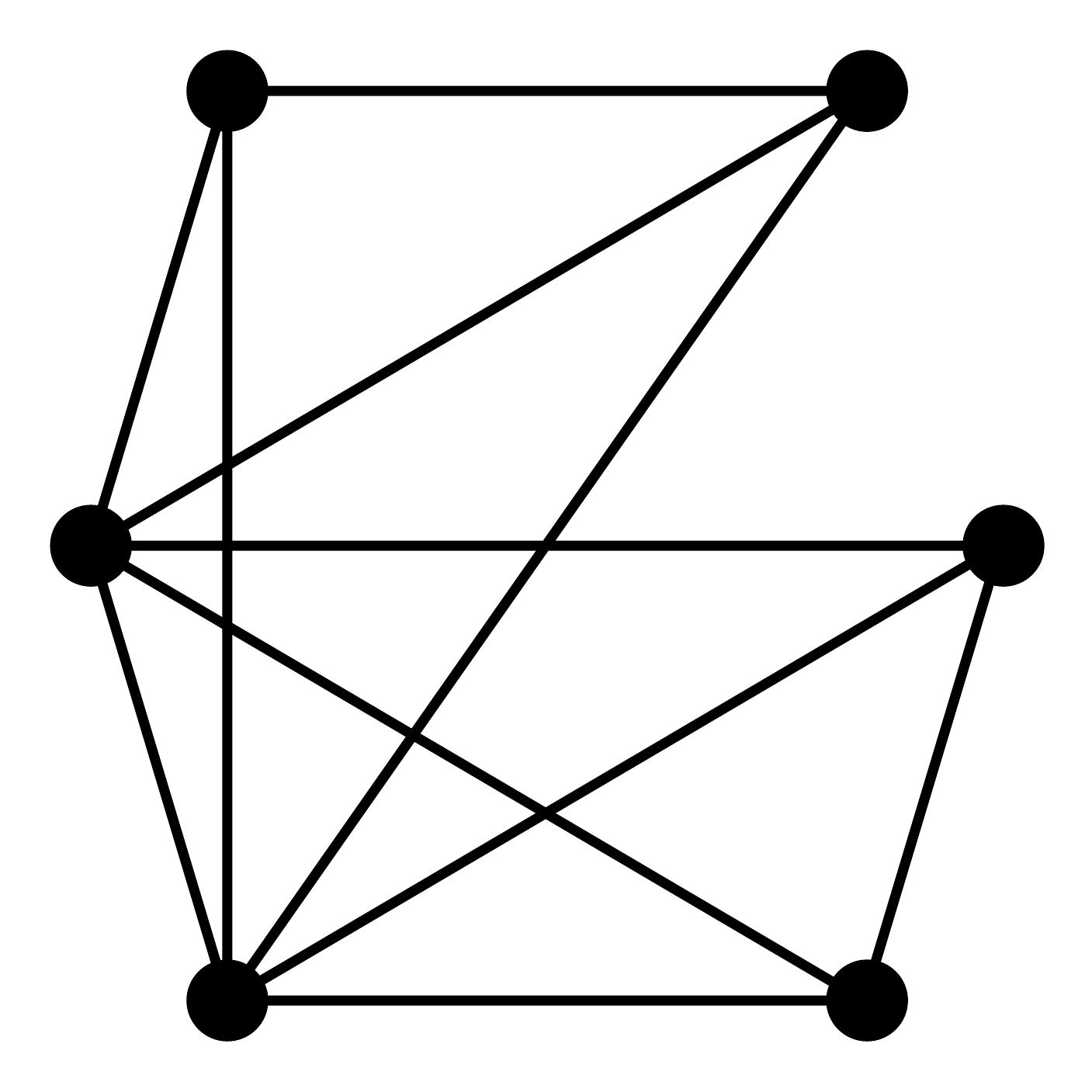}
		    & \includegraphics[width= \powtabwidth]{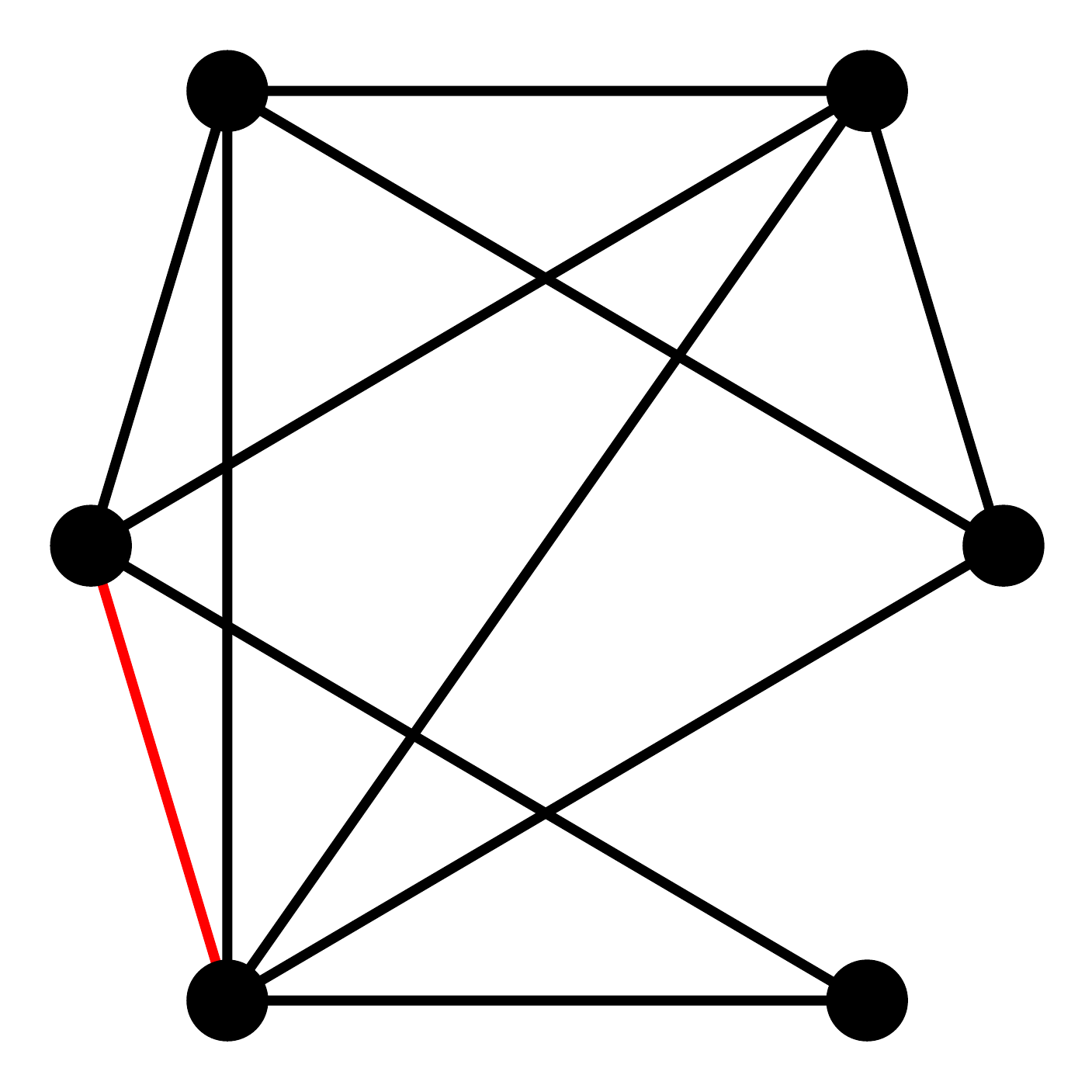}
		    & \includegraphics[width= \powtabwidth]{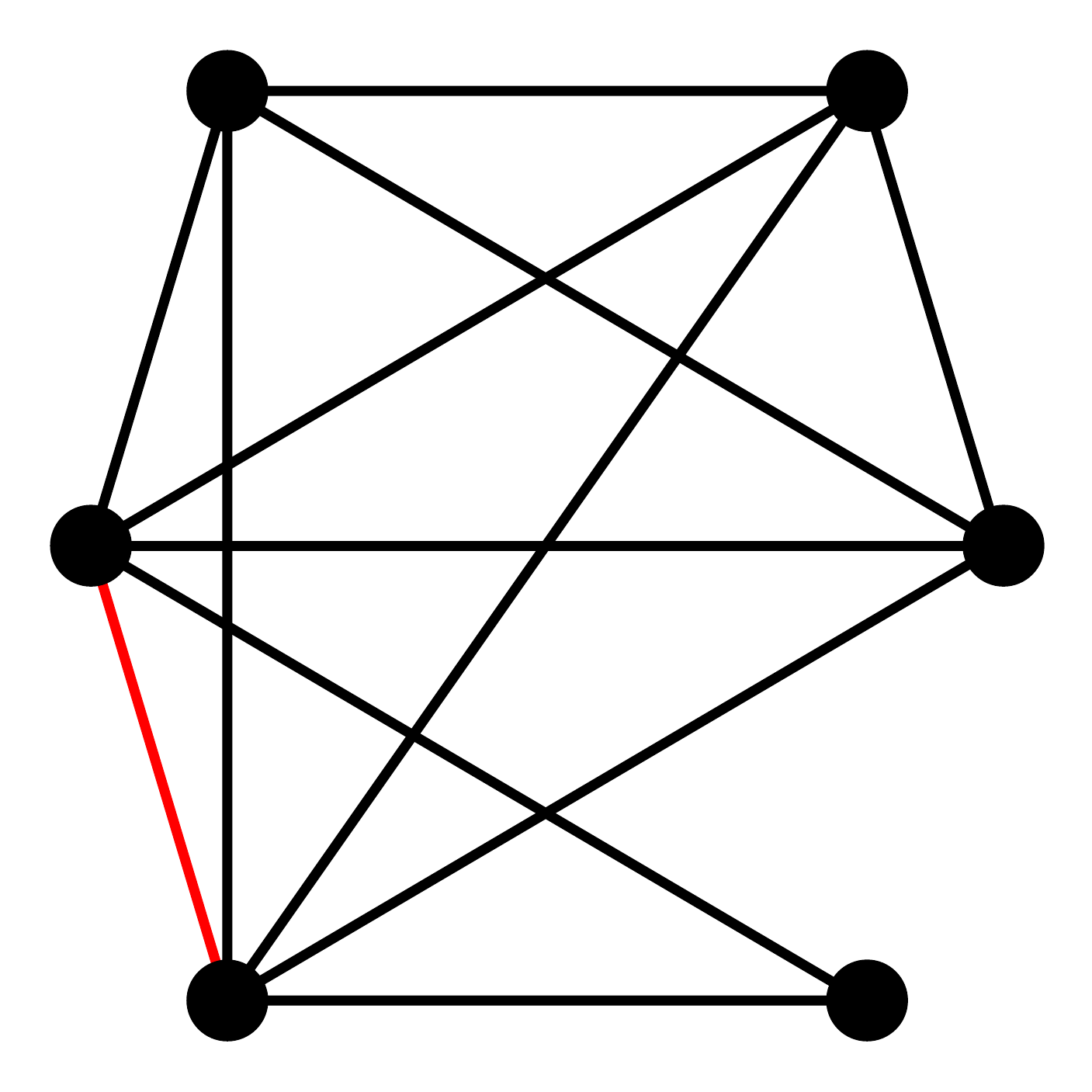}
			\\
		$p=1$
			& $p=2$
			& $p=3$
			& $p=4$
			& $p=5$
			& $p=5$
			& $p=6$
			\\
			 \includegraphics[width= \powtabwidth]{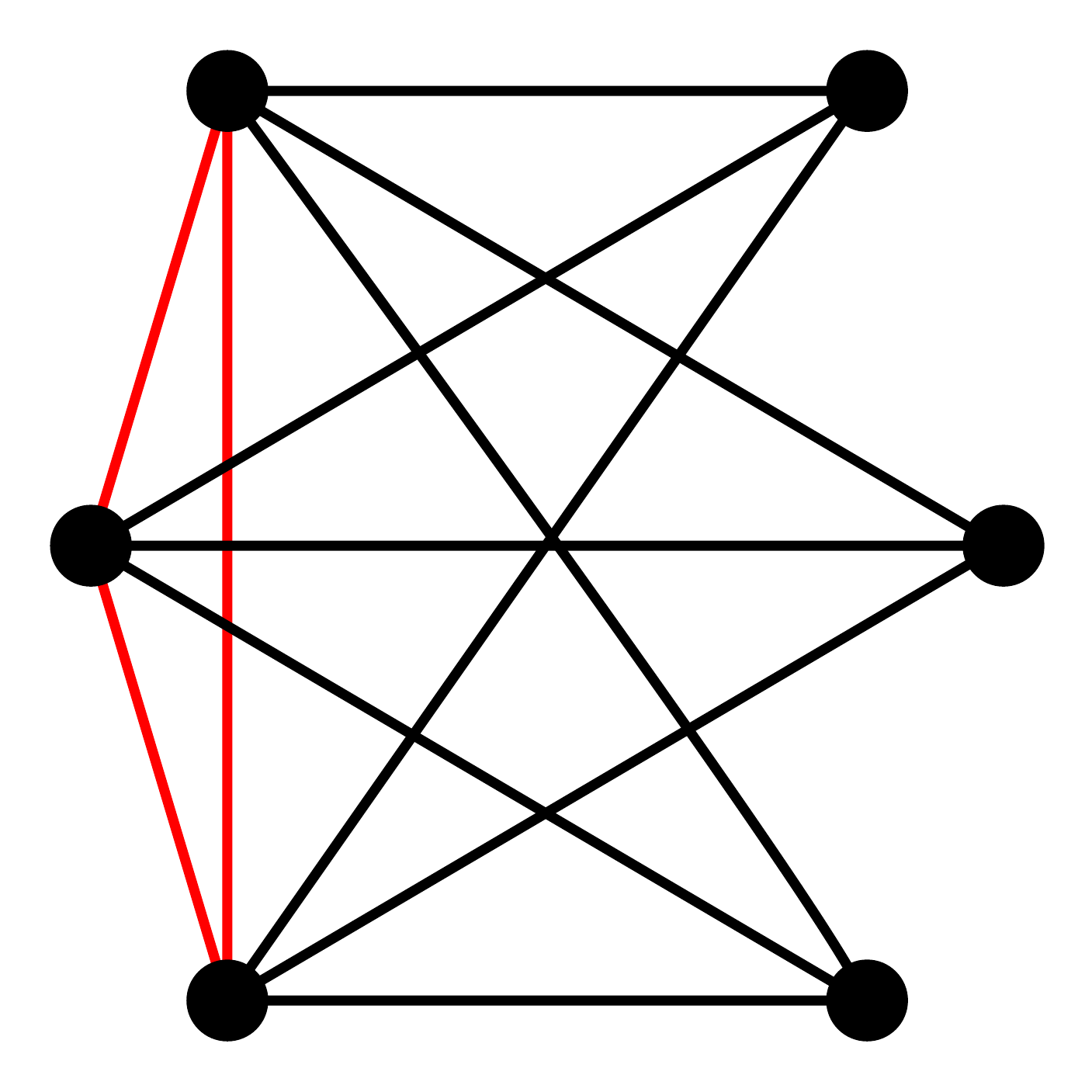}
			& \includegraphics[width= \powtabwidth]{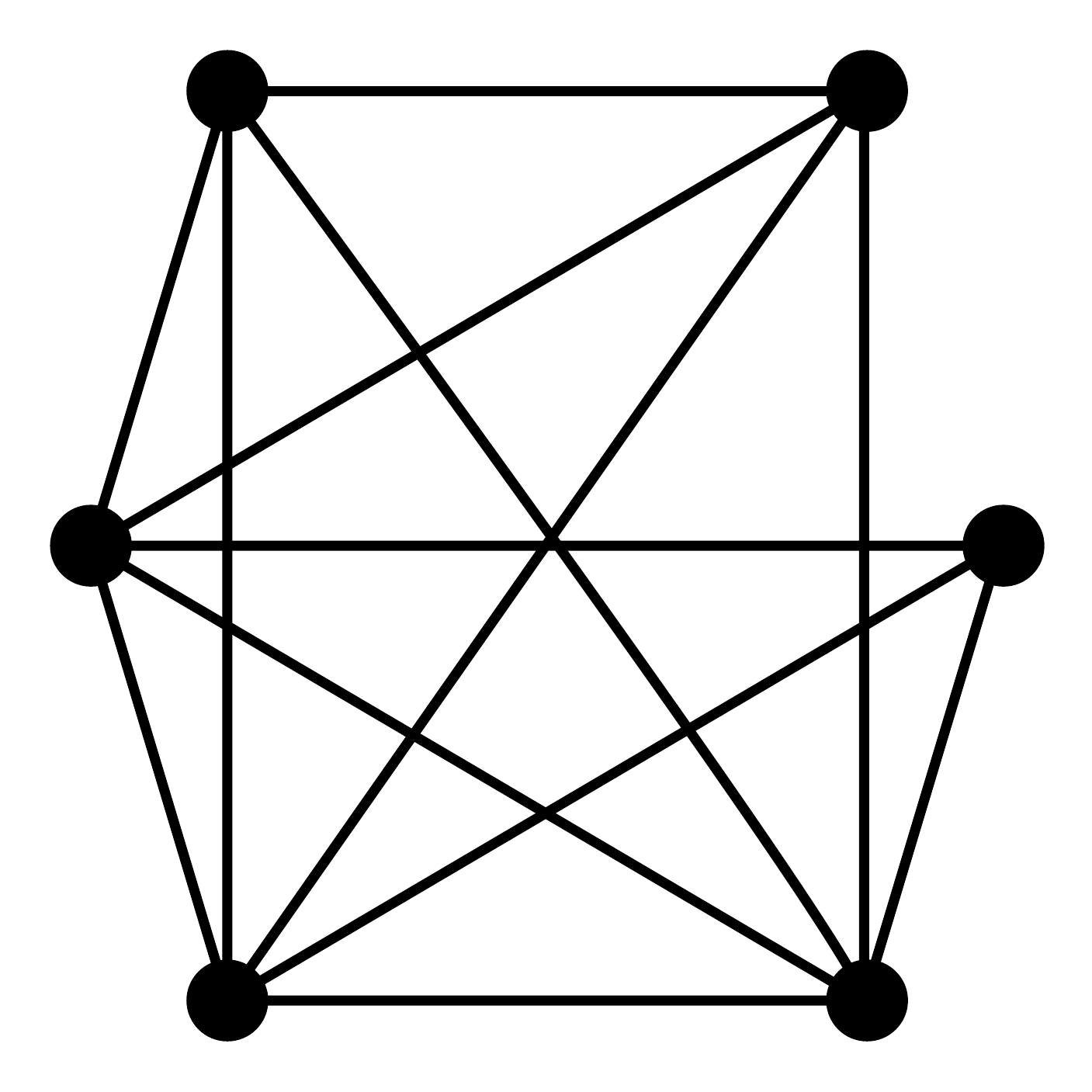}
			& \includegraphics[width= \powtabwidth]{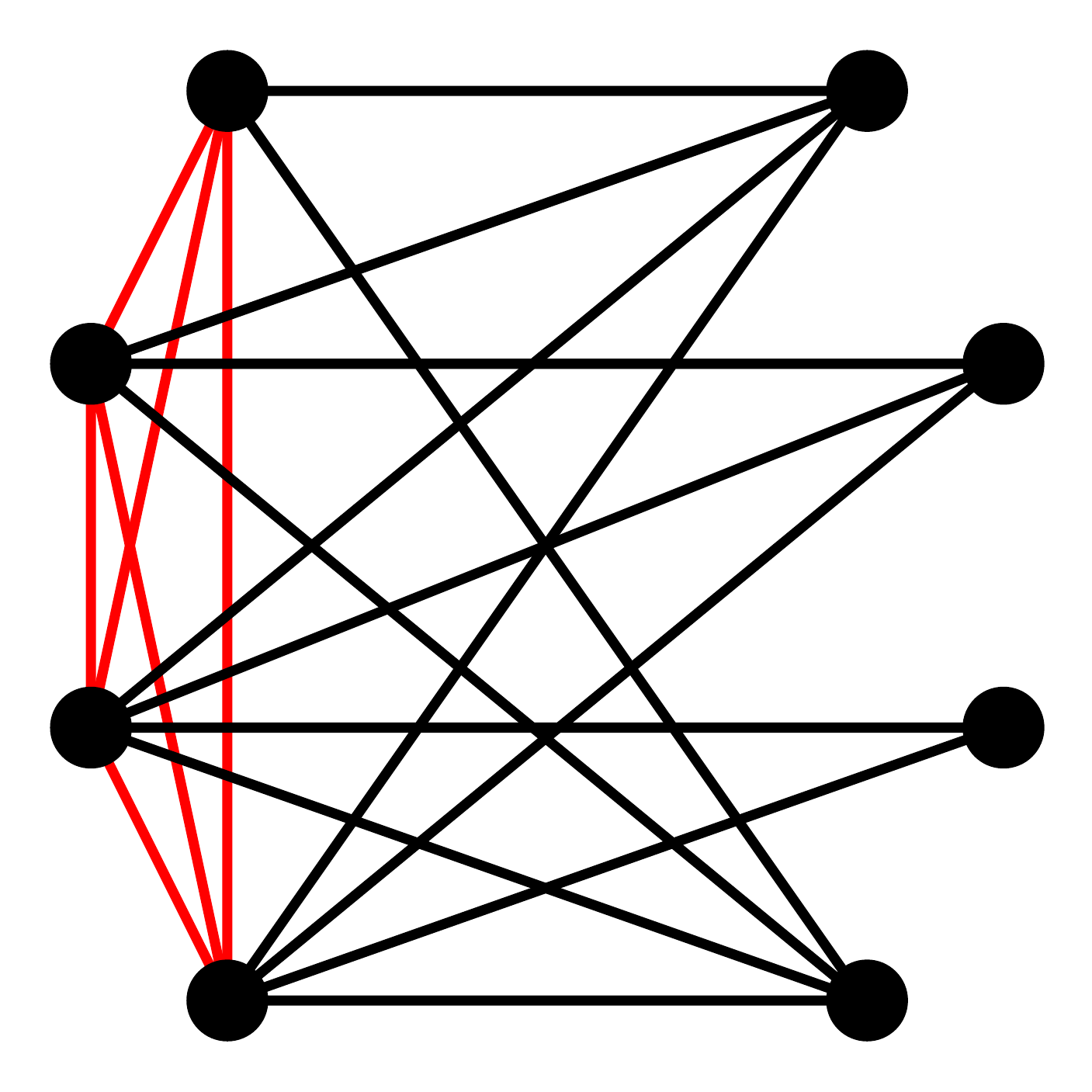}
			& \includegraphics[width= \powtabwidth]{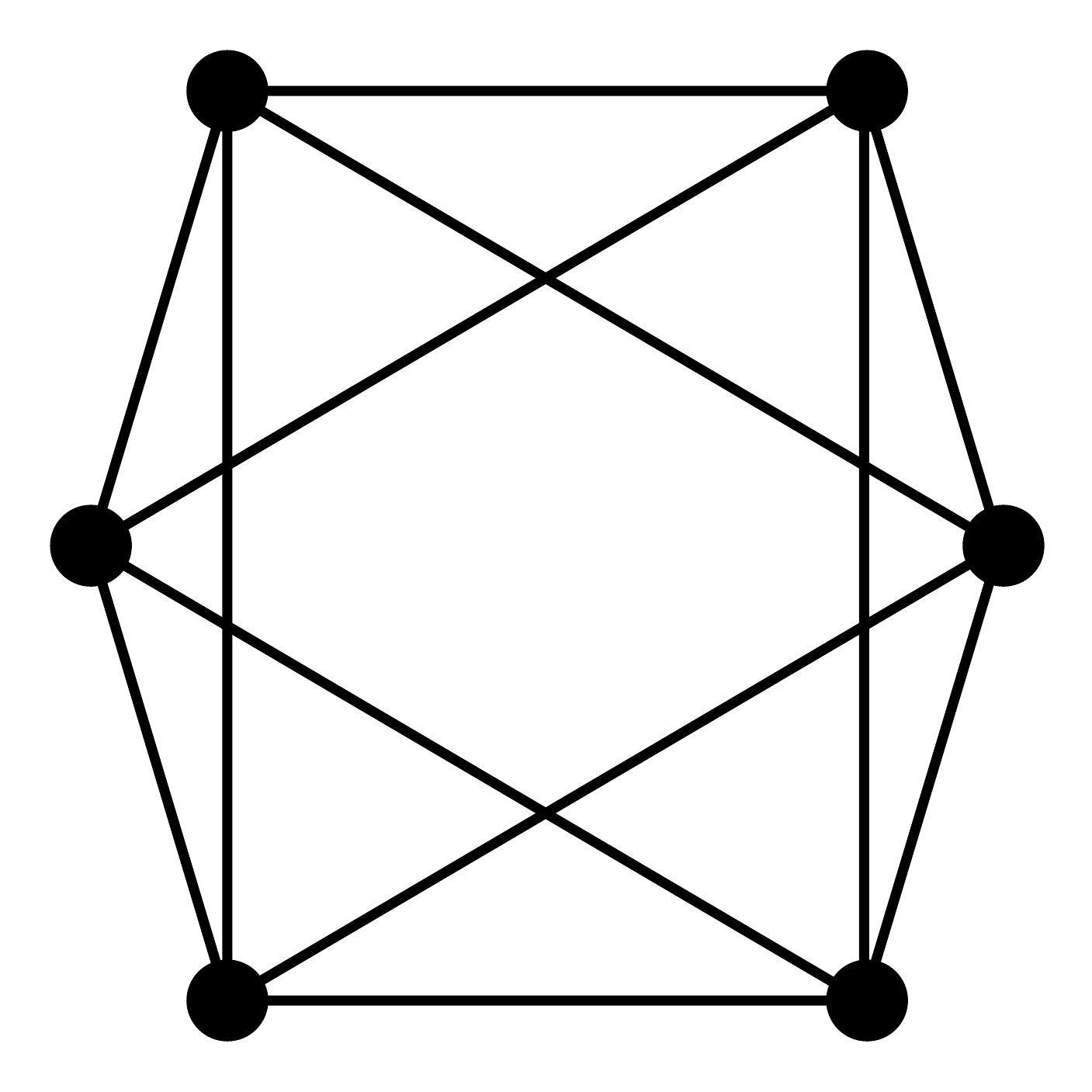}
			& \includegraphics[width= \powtabwidth]{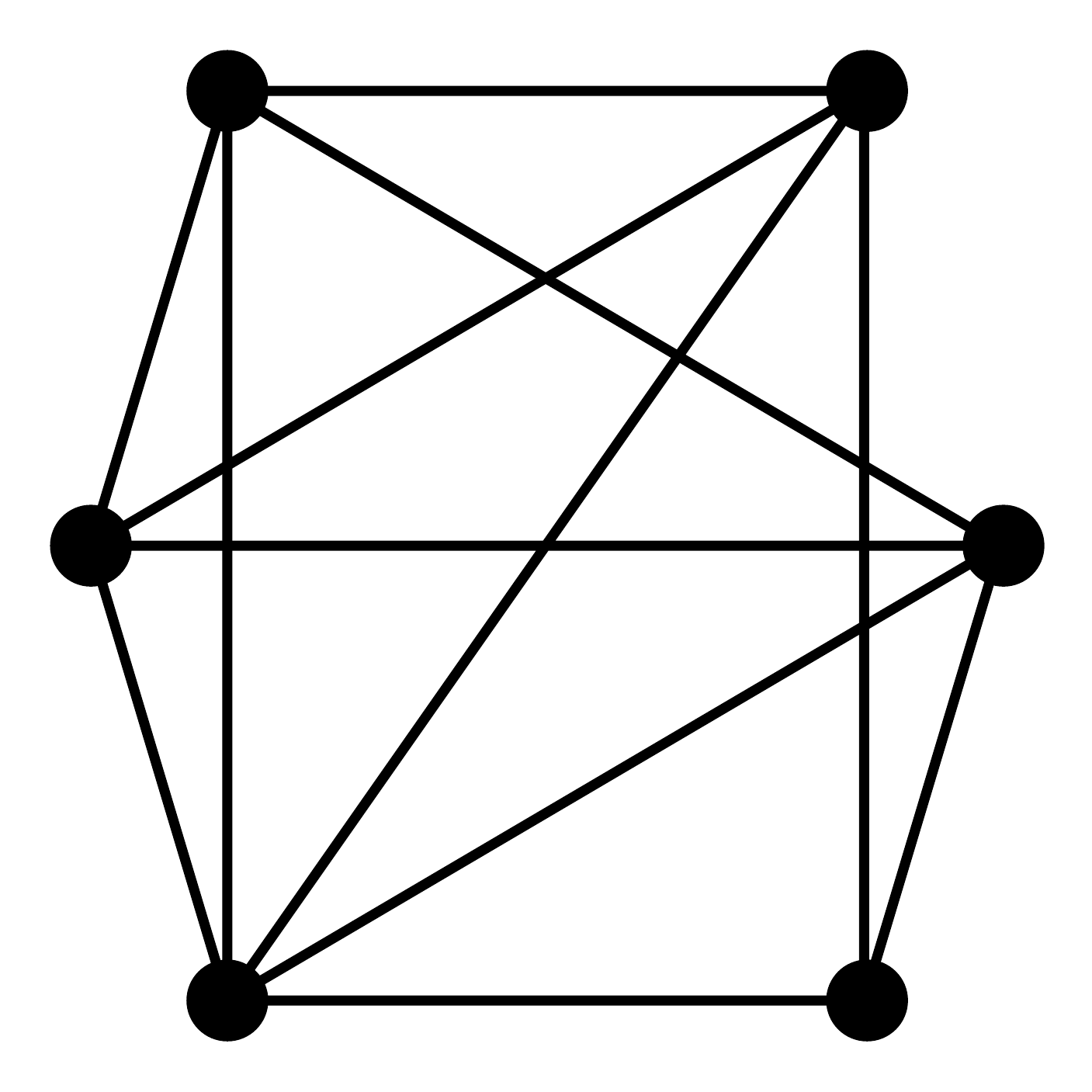}
			& \includegraphics[width= \powtabwidth]{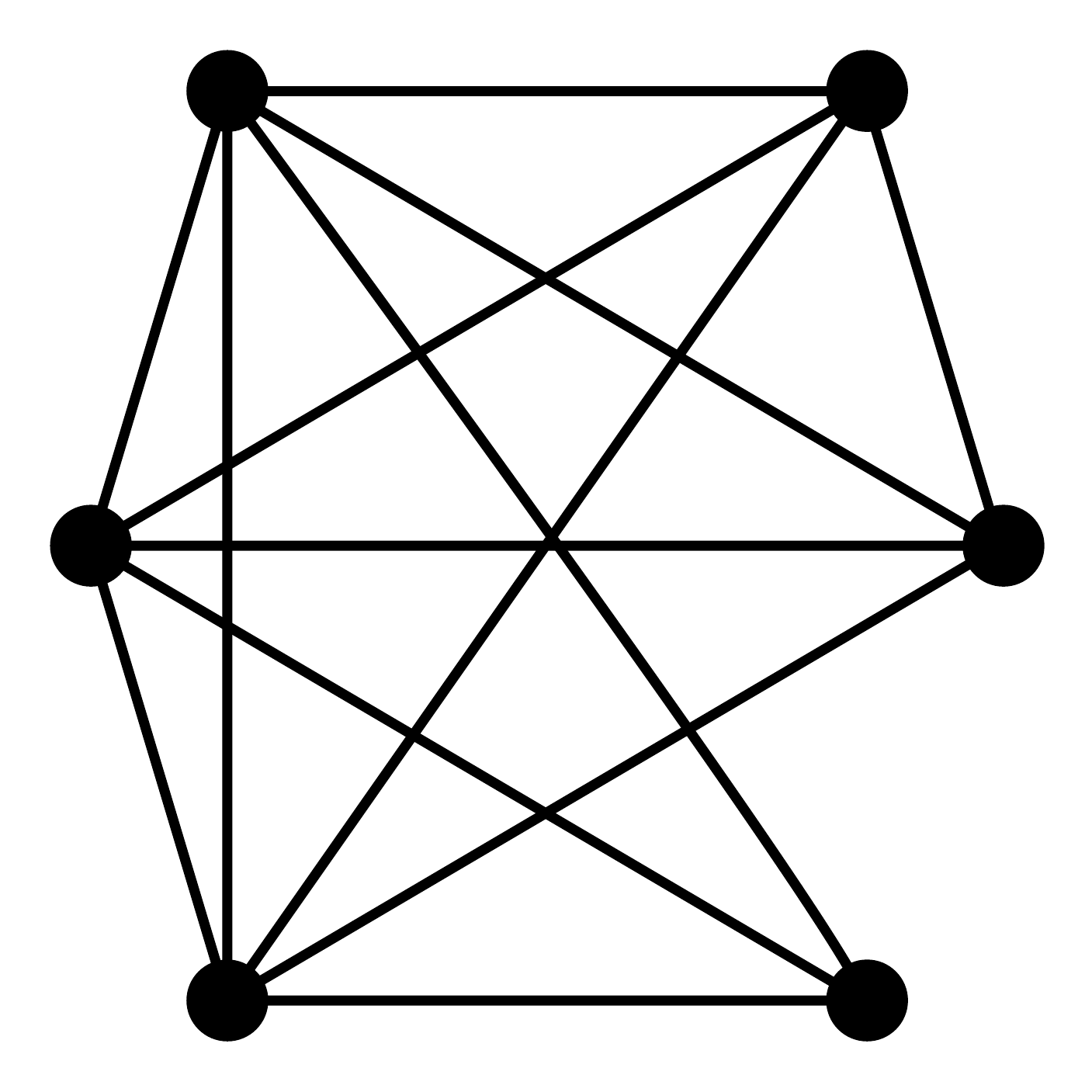}
			& \includegraphics[width= \powtabwidth]{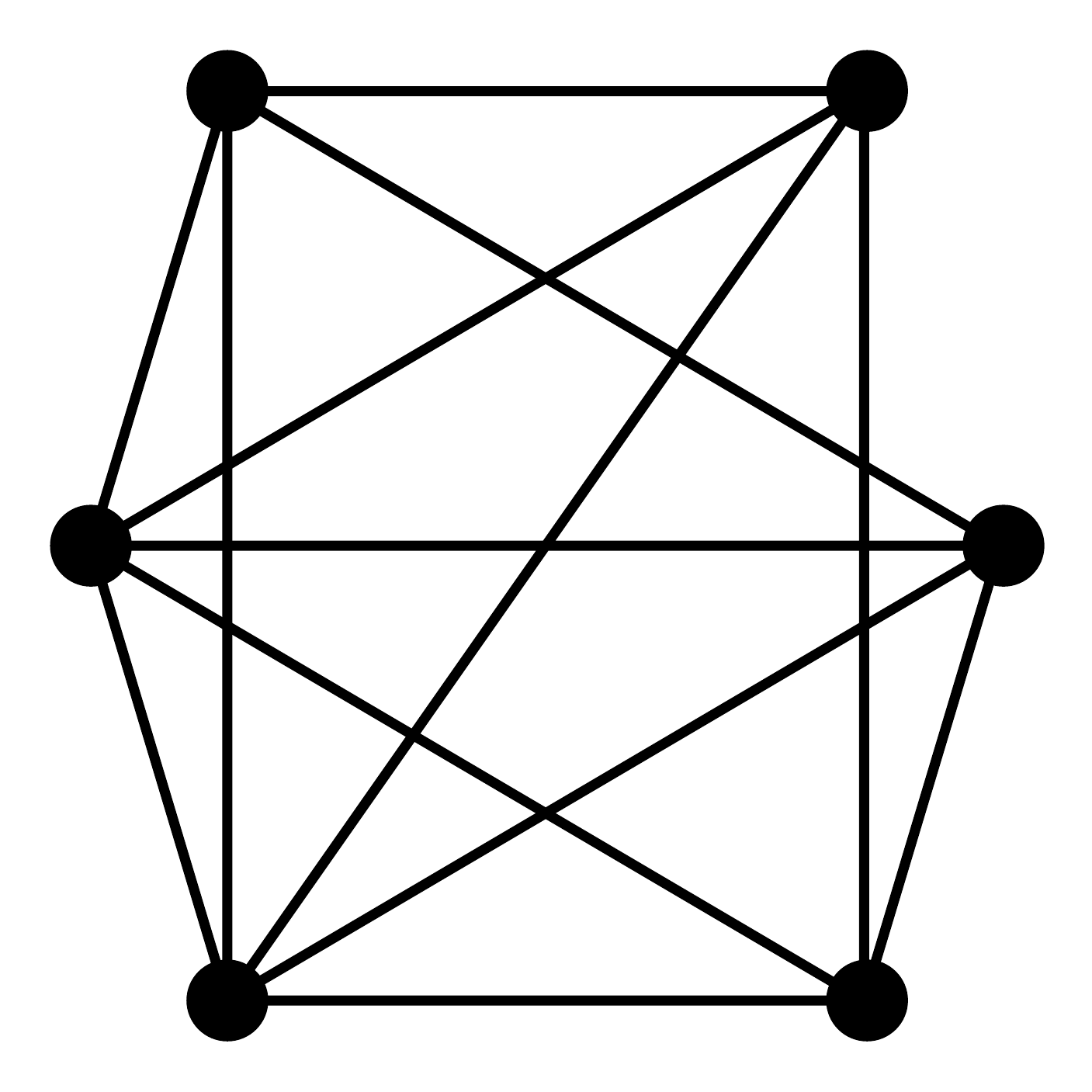}
			 \\
			 $p=6$
				& $p=7$
				& $p=8$
				& $p=8$
				& $p=8$
				& $p=9$
				& $p=10$
				\\
	\end{tabular}
	\caption{\label{fig:PowersOf2}The $p$-extremal elementary graphs with $1 \leq p \leq 10$~\cite{DudekSchmitt,HSWY}.}
\end{figure}

For $p \in \{11,13,17,19,23\}$, $p$ is prime, and there is no non-trivial factorization of $p$.
Hence, a $p$-extremal graph is a spire using exactly one $p$-extremal elementary graph
	with all other vertices within chambers isomorphic to $K_2$.
In most cases, the $p$-extremal elementary graph must appear at the top of the spire.
Only when $p = 11$ and the $11$-extremal elementary graph chosen 
	is the one with a barrier of size $4$ can this chamber be positioned anywhere
	in the spire.
	
For each $p \in \{15,22,25,26\}$, $p$ has at least one non-trivial factorization $p = \prod p_i$,
	but the sum of $c_{p_i}$ over the factors is strictly below $c_p$.
Hence, no $p$-extremal spire could contain more than one non-trivial chamber.
Also, all $p$-extremal elementary graphs have a barrier with relative size
	strictly below $\frac{1}{2}$, so the non-trivial chamber must appear at the top of the spire.
	
For each $p \in \{21,27\}$, there exists at least one factorization $p = \prod p_i$, 
	all with $\sum c_{p_i} \leq c_p$,
	and at least one factorization which reaches $c_p$ with equality.
For example, $21 = 3\cdot 7$, and $c_3 + c_7 = 2 + 3 = 5 = c_{21}$.
However, in these cases of equality, 
	$p_i$-extremal elementary graphs
	with large barriers 
	do not exist
	and it is impossible to achieve an excess of $c_p$
	over the entire spire using multiple non-trivial chambers.
Hence, the $p$-extremal graphs for these values of $p$
	have exactly one non-trivial chamber with $p$ perfect matchings
	and these chambers have small barriers, so they must 
	appear at the top of the spire.

For each $p\in \{14, 18,20,24\}$,
	there is at least one factorization $p = \prod p_i$ so that 
	$\sum c_{p_i} = c_p$
	and there are $p_i$-extremal graphs with large enough barriers
	to admit a spire with excess $c_p$.
These factorizations are $14 = 2\cdot 7$, $18 = 3\cdot 6 = 2\cdot 9$,
	$20 = 2 \cdot 10$, and $24 = 2 \cdot 12$.
There are also the $p$-extremal spires with exactly
	one non-trivial chamber, most of which 
	must appear at the top of the spire.
For $p \in \{14, 24\}$, there exists one $p$-extremal elementary graph
	with a large barrier
	that can appear anywhere in a $p$-extremal  spire.

The case $p = 16$ is special: every factorization admits at least one configuration 
	for a $16$-extremal spire.
The $q$-extremal elementary graphs
	for $q \in \{1,2,4,8\}$ as found by Hartke, Stolee, West, and Yancey~\cite{HSWY}
	are given in Figure \ref{fig:PowersOf2}.
Note that for each such $q$, there exists at least one $q$-extremal graph with a barrier
	with relative size $\frac{1}{2}$.
This allows any combination of values of $q$ that have product 16
	give a spire with $16$ perfect matchings
	and excess equal to the sum of the excesses of the chambers,
	which always adds to $c_{16} = 4$.
There are two $8$-extremal elementary graphs
	and three $16$-extremal elementary graphs 
	which have small barriers and 
	must appear at the top of a $16$-extremal spire.
All other chambers of a $16$-extremal spire can 
	take any order.

\section{Future work}

The $O(\sqrt{p})$ bound $N_p$ on the number of vertices in 
	a $p$-extremal elementary graph
	was sufficient for the computational technique described in this work
	to significantly extend the known values of $c_p$.
However, all of the
	elementary graphs we discovered to be $p$-extremal for $p \leq 27$
	have at most $10$ vertices,
	which could be generated using existing software,
	such as McKay's \texttt{geng} program~\cite{nauty}.
With a smaller $N_p$ value,
	the distributed search can also be improved.
Computation time would still be exponential in $p$ 
	because the depth of the search is a function of $p$,
	but the branching factor at each level would be reduced.
This delays the exponential behavior and potentially 
	makes searches over larger values of $p$ become tractable.
This motivates the following conjecture.

\begin{conjecture}
	Every $p$-extremal elementary graph
		has at most $2\log_2 p$ vertices. 
\end{conjecture}

This conjecture is tight for $p = 2^k$, with $k \in \{1,2,3,4\}$ and holds for all $p \leq 27$.
Note that $n_8 = 6$, but there is an $8$-extremal elementary graph with eight vertices.
Similarly, $n_{16} = 8$, but there is a $16$-extremal elementary graph with ten vertices.

The structure theorem requires searching over all factorizations of $p$
	in order to determine which factorizations yield 
	a spire with the 
	largest excess.
However, all known values of $p$ admit $p$-extremal elementary graphs.
Moreover, all composite values $p = p_1p_2$ admit $c_p \geq c_{p_1}+c_{p_2}$.
Does this always hold?

\begin{conjecture}
	For all $p \geq 1$, there exists a $p$-extremal elementary graph.
\end{conjecture}

\begin{conjecture}
	\label{conj:compositep}
	For all products $p = \prod_{i=1}^k p_i$ with $p \geq 1$, 
		$c_{p} \geq \sum_{i=1}^k c_{p_i}$.
\end{conjecture}

The closest known bound to Conjecture \ref{conj:compositep}
	is $c_p \geq c_{p_1} + \sum_{i=2}^k w(p_i)$, where
	$w(n)$ is the number of 1's in the binary representation of $n$ 
	~\cite[Proposition 7.1]{HSWY}.

\section*{Acknowledgements}

This work was completed utilizing the Holland Computing Center of the University of Nebraska.
Thanks to the Holland Computing Center faculty and staff
	including Brian Bockelman, Derek Weitzel, and David Swanson.
Thanks to the Open Science Grid for access to significant computer resources.
The Open Science Grid is supported by the National Science Foundation 
	and the U.S. Department of Energy's Office of Science.


\bibliographystyle{abbrv}
\bibliography{perfectmatchings}

\end{document}